\documentclass[12pt,reqno]{amsart}

\usepackage[margin=1in]{geometry}
\usepackage[fleqn,tbtags]{mathtools}
\usepackage[shortlabels]{enumitem}
\usepackage{graphicx}
\usepackage{amssymb}
\usepackage{amsthm}
\usepackage{mathrsfs}
\usepackage{mlmodern}
\usepackage{eucal}
\usepackage{microtype}

\usepackage[nodisplayskipstretch]{setspace}
\usepackage{tikz}
\usetikzlibrary{arrows.meta,positioning,calc}
\usepackage{xcolor}

\usepackage[colorlinks, breaklinks=true, pagebackref=false]{hyperref}
\usepackage[nameinlink, capitalize, noabbrev]{cleveref}
\newtheorem{theorem}{Theorem}[section]
\newtheorem{definition}{Definition}[section]
\newtheorem{lemma}{Lemma}[section]
\newtheorem{corollary}{Corollary}[section]
\newtheorem{proposition}{Proposition}[section]
\newtheorem{remark}{Remark}[section]
\newtheorem{example}{Example}[section]
\newtheorem{conjecture}{Conjecture}[section]
\newtheorem{que}{Question}[section]

\title[Capacity stability and moving singularities]
{Capacity Stability of Complex Monge-Amp\`ere Equations with Moving Prescribed Singularities}

\author[K. Pang]{Kai Pang}
\address{Kai Pang: School of Mathematical Sciences\\ Beijing Normal University\\ Beijing 100875\\ P. R. China}
\email{202331130031@mail.bnu.edu.cn}
\author[H. Sun]{Haoyuan Sun}
\address{Haoyuan Sun: School of Mathematical Sciences\\ Beijing Normal University\\ Beijing 100875\\ P. R. China}
\email{202531130037@mail.bnu.edu.cn}
\author[Z. Wang]{Zhiwei Wang}
\address{Zhiwei Wang: Laboratory of Mathematics and Complex Systems (Ministry of Education)\\ School of Mathematical Sciences\\ Beijing Normal University\\ Beijing 100875\\ P. R. China}
\email{zhiwei@bnu.edu.cn}

\author[X. Zhou]{Xiangyu Zhou}
\address{Xiangyu Zhou: Institute of Mathematics\\ Academy of Mathematics and Systems Science\\
	and Hua Loo-Keng Key Laboratory of Mathematics\\ Chinese Academy of Sciences\\ Beijing 100190\\ P. R. China}
\email{xyzhou@math.ac.cn}


\begin{document}
	
	\begin{abstract}

For complex Monge–Amp\`ere equations with moving big cohomology classes and prescribed model singularities of positive Monge-Amp\`ere mass, we prove that, under total variation convergence of the right-hand side non-pluripolar positive Radon measures, convergence of the prescribed model potentials in Monge–Amp\`ere capacity is equivalent to convergence in capacity of the associated normalized solutions.  We further prove that the ceiling operator coincides with the singularity envelope for potentials associated to a big $(1,1)$-class, regardless of their Monge-Amp\`ere mass, thereby resolving a conjecture of Darvas-Di Nezza-Lu.  Consequently, the singularity envelope is idempotent without  the  positivity assumption on the mass.

	\end{abstract}
	\maketitle
	\tableofcontents
	
	\section{Introduction}

	The complex Monge-Amp\`ere equation is a central object in complex geometry
	and pluripotential theory. In the smooth non-degenerate setting, it is closely
	related to the Calabi problem and K\"ahler-Einstein geometry
	\cite{Cal54,Cal57,Aub76,Aub78,Yau78,Aub82}. Its weak and degenerate forms were
	developed through the work of Bedford-Taylor \cite{BT76,BT82,BT87}, Cegrell
	\cite{Ceg98}, and Ko{\l}odziej \cite{Ko98,Ko03}, and subsequently in the
	global theory on compact K\"ahler manifolds and big cohomology classes
	\cite{GZ05,GZ07,BEGZ10,GZ17,BBGZ13,WN19}. These developments play an important
	role in singular K\"ahler-Einstein geometry and in the metric geometry of
	spaces of potentials \cite{EGZ09,Dar17,Dar19,DR17}. The corresponding pluripotential theory was also developed on compact Hermitian manifolds, see  \cite{Che87, Han96, GL10, TW10, Kol05,KN15, Ngu16, LWZ24,LWZ25,BGL25, SW25, PSWZ25} and references therein.  The relative
	pluripotential theory of Darvas-Di Nezza-Lu
	\cite{DDL18a,DDL18b,DDL18c,DDL21a,DDL21b,DDL25} provides the corresponding
	framework with prescribed singularities.
	
	Let \((X,\omega)\) be a compact K\"ahler manifold, and let \(\theta\) be a
	smooth closed real \((1,1)\)-form representing a big class. We write
	\(\operatorname{PSH}(X,\theta)\) for the set of \(\theta\)-plurisubharmonic
	functions, that is, upper semicontinuous functions
	\(\psi:X\to\mathbb R\cup\{-\infty\}\) which are locally the sum of a smooth
	function and a plurisubharmonic function and satisfy
	\(\theta+dd^c\psi\geq 0\) in the sense of currents. The potential with minimal
	singularities is
	\[
	V_\theta:=\sup\{u\in\operatorname{PSH}(X,\theta):u\leq 0\}.
	\]
	If \(\phi\in\operatorname{PSH}(X,\theta)\) is a positive mass model potential,
	the relative full mass class \(\mathcal E(X,\theta,\phi)\) is the natural
	domain for solutions of
	\[
	\theta_u^n=\mu,\qquad u\in\mathcal E(X,\theta,\phi),
	\]
	where \(\mu\) is a non-pluripolar positive Radon measure with
	\[
	\mu(X)=\int_X\theta_\phi^n.
	\]
	In this setting the relative theory gives existence, uniqueness, comparison
	principles, variational tools, relative energies, and a metric geometry of
	singularity types. In particular, Darvas-Di Nezza-Lu proved an important stability
	theorem for moving model potentials whose singularity types converge in the
	\(d_{\mathcal S}\)-metric \cite{DDL21a}. Related quantitative and
	\(L^1\)-type stability results were obtained in \cite{DV25,LZ25}.
	
	The \(d_{\mathcal S}\)-topology is well adapted to mixed non-pluripolar masses
	and records fine singularity data, such as Lelong numbers and multiplier
	ideals; see, for instance, \cite{RWN17,DX22,DX24}. It is, however, too rigid
	for certain geometric families with moving singularities. Logarithmic poles
	may move, collide, or deform in such a way that the corresponding potentials
	converge in Monge-Amp\`ere capacity, while the associated singularity types
	do not converge in \(d_{\mathcal S}\). Since capacity convergence is the
	natural stability notion for the complex Monge-Amp\`ere operator
	\cite{Xing96,CK06,Xing08,Xing09,GZ12}, this leads to the following problem.
	
	\begin{que}\label{que:capacity-topology}
		Identify the intrinsic topology on moving prescribed singularities which is
		detected by Monge-Amp\`ere stability.
	\end{que}
	
	Our main result answers this question. Under total variation convergence of
	the right-hand side measures, convergence in capacity of the prescribed model
	potentials is equivalent to convergence in capacity of the corresponding
	normalized solutions. Thus capacity is the intrinsic stability topology for
	moving prescribed singularities.
	
	\begin{theorem}[Capacity stability for moving prescribed singularities]
		\label{thm:intro-optimal-capacity-topology}
		Let \((X,\omega)\) be a compact K\"ahler manifold. Let
		\(\theta_j,\theta\) be smooth closed real \((1,1)\)-forms with big cohomology
		classes and assume
		\[
		-\varepsilon_j\omega\leq \theta_j-\theta\leq \varepsilon_j\omega,
		\qquad
		\varepsilon_j\to0.
		\]
		Let \(\phi_j\in\operatorname{PSH}(X,\theta_j)\) and
		\(\phi\in\operatorname{PSH}(X,\theta)\) be normalized positive mass model
		potentials. Let \(\mu_j,\mu\) be non-pluripolar positive Radon measures such
		that
		\[
		\|\mu_j-\mu\|_{\rm TV}\to0,\qquad
		\mu_j(X)=\int_X\theta_{j,\phi_j}^n,\qquad
		\mu(X)=\int_X\theta_\phi^n>0.
		\]
		Let \(u_j\in\mathcal E(X,\theta_j,\phi_j)\) and
		\(u\in\mathcal E(X,\theta,\phi)\) be the normalized solutions of
		\[
		\theta_{j,u_j}^n=\mu_j,\quad \sup_Xu_j=0,
		\qquad
		\theta_u^n=\mu,\quad \sup_Xu=0.
		\]
		Then
		\[
		\phi_j\to\phi \text{ in }\operatorname{Cap}_\omega
		\quad\Longleftrightarrow\quad
		u_j\to u \text{ in }\operatorname{Cap}_\omega .
		\]
	\end{theorem}

	The forward implication, proved in
	\Cref{thm:moving-background-capacity-stability}, refines the \(d_{\mathcal S}\)-stability theorem of Darvas-Di Nezza-Lu
	\cite{DDL21a}; the reverse implication is proved in
	\Cref{thm:no-mass-loss-singularity-envelope-capacity}. Together these two
	implications give the asserted optimality of the capacity topology.
	
	We briefly indicate the proof. For the forward implication, the main point is
	to replace the unavailable \(d_{\mathcal S}\)-control by a strict subbarrier
	with a quantitative slope gap. Combined with a minimal-truncation argument,
	this sends the comparison errors into deep singularity tails, which are small
	under capacity convergence of the prescribed model potentials.
	
	For the reverse implication, one recovers the moving model envelopes from the
	normalized solutions. The argument uses rooftop approximations and a
	no-mass-loss estimate to obtain \(L^1\)-recovery of the envelopes; Hartogs'
	lemma then gives convergence in capacity.

	The same recovery argument also applies to the ceiling operator introduced by
	Darvas-Di Nezza-Lu in the metric geometry of singularity types; see
	\cite[Section~2]{DDL21a}. This operator extends the model-envelope formalism
	to classes of arbitrary mass. Darvas-Di Nezza-Lu proved that, in positive
	mass, it coincides with the usual singularity envelope, and conjectured that
	the same identity holds in general.
	
	\begin{conjecture}[{\cite[Conjecture 2.5]{DDL21a}}]
		\label{conj:intro-DDL-ceiling}
		Let \((X,\omega)\) be a compact K\"ahler manifold, and let \(\theta\) be a
		smooth closed real \((1,1)\)-form representing a big class. Then, for every
		\(u\in\operatorname{PSH}(X,\theta)\),
		\[
		\mathscr C_\theta(u)=P_\theta[u].
		\]
	\end{conjecture}
	
	Lu also conjectured the idempotence of the singularity envelope.
	
	\begin{conjecture}[{\cite[Conjecture 1.8]{Lu22}}]
		\label{conj:intro-Lu-envelop}
		Let \((X,\omega)\) be a compact K\"ahler manifold, and let \(\theta\) be a
		smooth closed real \((1,1)\)-form representing a big class. Then, for every
		\(u\in\operatorname{PSH}(X,\theta)\),
		\[
		P_\theta[P_\theta[u]]=P_\theta[u].
		\]
	\end{conjecture}
	
	Our second main result proves these conjectures.
	
	\begin{theorem}
		\label{thm:intro-ceiling-equals-singularity-envelope}
		Let \((X,\omega)\) be a compact K\"ahler manifold, and let \(\theta\) be a
		smooth closed real \((1,1)\)-form representing a big class. Then, for every
		\(u\in\operatorname{PSH}(X,\theta)\),
		\[
		\mathscr C_\theta(u)=P_\theta[u],
		\qquad
		P_\theta[P_\theta[u]]=P_\theta[u].
		\]
		In other words, \cref{conj:intro-DDL-ceiling} and \cref{conj:intro-Lu-envelop} hold true.
	\end{theorem}
	
	Thus the ceiling operator coincides with the singularity envelope even in
	arbitrary mass. In particular, the singularity envelope is idempotent without
	any positive mass assumption. This gives the arbitrary-mass completion of the
	Darvas-Di Nezza-Lu model-envelope formalism.
	
	The following example illustrates that the topology in 
	\Cref{thm:intro-optimal-capacity-topology} is strictly weaker than
	\(d_{\mathcal S}\). The details are given in
	\Cref{ex:moving-poles-P1-cap-not-dS}. Let \(X=\mathbb P^1\),
	\(\omega=\omega_{\rm FS}\), \(D_1=\{Z_0=0\}\), and
	\[
	D_j=\{Z_0-\varepsilon_j Z_1=0\},\qquad \varepsilon_j\to0.
	\]
	For \(0<c<1\), let \(s_j\) be a defining section of \(D_j\), and set
	\[
	u_j=c\log |s_j|_{h_{\rm FS}}^2,\qquad
	\phi_j=P_\omega[u_j].
	\]
	Then
	\[
	\phi_j\to\phi_1
	\quad\text{in }\operatorname{Cap}_\omega,
	\qquad
	\int_{\mathbb P^1}\omega_{\phi_j}
	=
	\int_{\mathbb P^1}\omega_{\phi_1}
	=
	1-c.
	\]
	On the other hand, for \(j\geq2\), the divisors \(D_j\) and \(D_1\) are
	distinct, and \(\max(\phi_j,\phi_1)\) has minimal singularity type. Using the
	Darvas-Di Nezza-Lu formula for \(d_{\mathcal S}\), one obtains a uniform
	positive lower bound for	$d_{\mathcal S}([\phi_j],[\phi_1])$.
	Thus \([\phi_j]\) does not converge to \([\phi_1]\) in \(d_{\mathcal S}\),
	although \(\phi_j\to\phi_1\) in capacity. This illustrates the additional
	scope of the capacity stability theorem beyond the \(d_{\mathcal S}\)-setting.

	We also obtain several applications beyond the Monge-Amp\`ere equation itself.
	The first concerns envelopes with prescribed singularities. Such envelopes  have
	appeared in the work of Ross-Witt Nystr\"om \cite{RWN17}, Berman-Boucksom
	\cite{BB10}, and Darvas-Xia \cite{DX22,DX24}. We prove the following capacity
	continuity theorem.
	
	\begin{theorem}[Capacity continuity of relative envelopes]
		\label{intro_thm:relative-bounded-obstacle-capacity-continuity}
		Let \((X,\omega)\) be a compact K\"ahler manifold and let \(\theta\) be a smooth closed real $(1,1)$-form whose cohomology class is big.  Let
		\(\phi_j,\phi\in\operatorname{PSH}(X,\theta)\) be normalized positive mass model
		potentials such that
		\[
		\phi_j\to\phi
		\quad\text{in }\operatorname{Cap}_\omega.
		\]
		Let \(h_j,h\) be quasi-continuous functions such that \(h_j\to h\) in capacity.
		Assume that, for some \(C>0\),
		\[
		\phi_j-C\le h_j\le\phi_j,\qquad
		\phi-C\le h\le\phi
		\]
		q.e. on \(X\). Then
		\[
		P_\theta(h_j)\to P_\theta(h)
		\quad\text{in }\operatorname{Cap}_\omega .
		\]
	\end{theorem}
	
	As a consequence, rooftop envelopes are continuous in capacity under moving
	prescribed singularity constraints. This is a capacity analogue of the
	continuity phenomena for envelopes and Monge-Amp\`ere measures, see for example \cite{RWN17,DDL18a,Dar19,DDL21a}.
	
	The second consequence concerns weak geodesic segments. Weak geodesics in the
	space of K\"ahler potentials are central in the variational approach to
	canonical metrics \cite{RWN14,Dar17,Dar19,DR17}. Using the rooftop-Legendre
	representation, together with Kiselman's minimum principle \cite{Kis78}, we
	obtain capacity stability of geodesic segments with moving prescribed
	singularities.
	
	\begin{theorem}\label{intro_thm:weak geodesic segments convergence} Let \((X,\omega)\) be a compact K\"ahler manifold and let \(\theta\) be a smooth closed real $(1,1)$-form whose cohomology class is big. 
		Let \(\phi_j,\phi\in\operatorname{PSH}(X,\theta)\) be normalized positive mass
		model potentials such that
		\[
		\sup_X\phi_j=\sup_X\phi=0,
		\qquad
		\phi_j\to\phi
		\quad\text{in }\operatorname{Cap}_\omega .
		\]
		Let \(u_{j,0},u_{j,1},u_0,u_1\in\operatorname{PSH}(X,\theta)\). Assume that
		there exists \(C>0\) such that
		\[
		\phi_j-C\le u_{j,0},u_{j,1}\le\phi_j,
		\qquad
		\phi-C\le u_0,u_1\le\phi
		\]
		for all \(j\), and assume that
		\[
		u_{j,0}\to u_0,
		\qquad
		u_{j,1}\to u_1
		\quad\text{in }\operatorname{Cap}_\omega .
		\]
		Let \(t\mapsto u_{j,t}\) be the weak geodesic segment joining
		\(u_{j,0},u_{j,1}\), and let \(t\mapsto u_t\) be the weak geodesic segment
		joining \(u_0,u_1\). Then $u_{j,t}\to u_t$ in $\text{Cap}_\omega$ uniformly. Namely, for every \(\varepsilon>0\),
		\[
		\lim_{j\to\infty}
		\sup_{0\le t\le1}
		\operatorname{Cap}_\omega(\{|u_{j,t}-u_t|>\varepsilon\})=0 .
		\]
	\end{theorem}

	We next apply the stability theorem to twisted K\"ahler-Einstein equations in
	big classes. In the Fano case, the relation between K\"ahler-Einstein metrics
	and algebraic stability has been extensively studied; see, for instance,
	\cite{DT92,Yau93,CDS15a,CDS15b,CDS15c,Tian15a,Tian15b,DS16,CSW18,BBJ21,LTW22,Zha24}.
	Big-class analogues have recently been developed by Darvas-Zhang
	\cite{DZ24} and Trusiani \cite{Tru22,Tru23,Tru24}. In this setting prescribed
	and movable singularities enter the variational problem itself. We prove the
	following capacity stability statement for moving tame data.

	\begin{theorem}[Stability for twisted K\"ahler-Einstein equations]
		\label{intro:stability-of-Darvas-Zhang}
		Let \((X,\omega)\) be a compact K\"ahler manifold, and let \(\theta\) be a
		smooth closed real \((1,1)\)-form representing a big class. Set
		\[
		V:=V_\theta,\qquad
		m:=\int_X\theta_V^n>0,\qquad
		\mu_j:=e^{f_j-\psi_j}\omega^n,\qquad
		\mu:=e^{f-\psi}\omega^n .
		\]
		Assume that \(\mu_j,\mu\) are tame measures and that
		\(\delta_\psi(\{\theta\})>1\). Let \(\eta_j,\eta\) be smooth closed real
		\((1,1)\)-forms such that:
		\begin{enumerate}
			\item \(c_1(X)=\{\theta\}+\{\eta_j\}=\{\theta\}+\{\eta\}\);
			\item \(\operatorname{Ric}(\omega)
			=\theta+\eta_j+dd^cf_j
			=\theta+\eta+dd^cf\);
			\item \(\eta_j+dd^c\psi_j\geq0\) and \(\eta+dd^c\psi\geq0\).
		\end{enumerate}
		Assume moreover that \(f_j\to f\) uniformly, that
		\(\psi_j\to\psi\) in \(\operatorname{Cap}_\omega\), and that
		\(\psi_j\geq\psi-A\) for some \(A>0\). Let
		\(u_j\in\mathcal E(X,\theta,V)\) be the normalized solutions
		\[
		\sup_Xu_j=0,\qquad
		\theta_{u_j}^n
		=
		m\frac{e^{-u_j}\mu_j}{\int_Xe^{-u_j}\,d\mu_j}.
		\]
		Then every subsequence of \((u_j)\) admits a further subsequence converging in
		capacity to some \(u\in\mathcal E(X,\theta)\) satisfying
		\[
		\theta_u^n
		=
		m\frac{e^{-u}\mu}{\int_Xe^{-u}\,d\mu}.
		\]
	\end{theorem}

	The proof combines the capacity stability theorem with uniform Ding
	properness, a Skoda-type estimate on Ding-bounded sets \cite{DZ24}, the strong
	openness theorem of Guan-Zhou \cite{GZ15-1,GZ15-2}, and the stability of
	multiplier ideal sheaves due to Guan-Li-Zhou \cite{GLZ22}; see also
	Xiao-Zhou \cite{XZ26}.

	The final application  concerns moving quantization of partial equilibrium
	measures.
	Following Berman-Boucksom-Witt Nystr\"om \cite{BBWN11}, Ross-Witt Nystr\"om \cite{RWN17} and Darvas-Xia
	\cite{DX24}, let \(P[u]_I\) be the \(I\)-model projection defined using
	multiplier ideal sheaves. For a weighted compact set \((K,v)\), let
	\(P_K[u]_I(v)\) be the corresponding partial \(I\)-equilibrium envelope.
	\begin{theorem}\label{intro: stability of Darvas-Xia}
		Let $(X,\omega)$ be a compact K\"ahler manifold, and let $\theta$ be a smooth
		closed real $(1,1)$-form whose cohomology class is big. Let
		$K\subset X$ be compact and non-pluripolar. Let $u_j,u\in\operatorname{PSH}(X,\theta)$ with $\sup_Xu_j=\sup_Xu=0$ and put
		\[
		\Phi_j:=P[u_j]_I,\qquad \Phi:=P[u]_I ,\qquad W_j:=P_K[u_j]_I(v_j),
		\qquad
		W:=P_K[u]_I(v).
		\]
		Assume that $\Phi_j\to\Phi$ in capacity. Let $v_j,v\in C^0(K)$ satisfy $v_j\to v$ uniformly on $K$. Then $W_j\to W$ in capacity.
		
		Consequently, if we moreover have $ \int_X\theta_{\Phi_j}^n\to \int_X\theta_\Phi^n>0$, then we have the weak convergence $\theta_{W_j}^n\rightharpoonup \theta_W^n$.
	\end{theorem}

	Combining this with the quantization theorem of Darvas-Xia
	\cite[Theorem~1.2]{DX24}, we obtain the weak convergence
	\[
	\lim_{j\to\infty}\lim_{k\to\infty}
	\beta^k_{v_j,u_j,\nu_j}
	=
	\theta^n_{P_K[u]_I(v)} .
	\]
	For moving logarithmic poles on \(\mathbb P^1\), this gives a concrete
	quantization by degree \(k\) sections satisfying the moving vanishing
	conditions
	\[
	\mathcal I(ku_j)
	=
	\mathcal O_{\mathbb P^1}(-\lfloor kc\rfloor D_j).
	\]
	This may be viewed as a moving-singularity version of partial Bergman measure
	convergence, in the spirit of equidistribution results for Fekete points
	\cite{BB10,BBWN11}.
	
	The paper is organized as follows. In \Cref{sec:preliminaries} we recall
	background from pluripotential theory, including non-pluripolar products,
	relative full mass classes, capacities, envelopes, the \(d_{\mathcal S}\)
	metric, multiplier ideals, and partial Bergman measures. In
	\Cref{sec:stability} we prove the direct capacity stability theorem and give
	the moving-pole example on \(\mathbb P^1\). In \Cref{sec-4} we prove capacity
	continuity of envelopes, the no-mass-loss recovery theorem, and the
	optimality of the capacity topology. In \Cref{sec-5} we prove the
	Darvas-Di Nezza-Lu ceiling conjecture. The remaining sections treat weak
	geodesic segments, twisted K\"ahler-Einstein equations in big classes, and
	moving quantization of partial equilibrium measures with multiplier ideal
	constraints.

	\subsection*{Acknowledgements}
	This research is supported by the National Key R\&D Program of China (Grant No. 2021YFA1002600 and  No. 2021YFA1003100).  Z. Wang and X. Zhou are partially supported by grants from the National Natural Science Foundation of China (NSFC) (No. 12571085) and (No. 12288201) respectively. Z. Wang is also supported by the Fundamental Research Funds for the Central Universities.
	
	\section{Preliminaries}
	\label{sec:preliminaries}
	
	Throughout the paper \((X,\omega)\) is a compact K\"ahler manifold of
	dimension \(n\). Unless otherwise specified, \(\theta\) is a smooth closed
	real \((1,1)\)-form whose cohomology class is big.  We use the convention
	\[
	d^c:=\frac{\sqrt{-1}}{2\pi}(\bar\partial-\partial),
	\qquad
	dd^c=\frac{\sqrt{-1}}{\pi}\partial\bar\partial .
	\]
	A function $u:X\rightarrow \mathbb R\cup \{-\infty\}$ is called quasi-plurisubharmonic (quasi-psh) if locally $u$ can be written as the sum of a smooth function and a plurisubharmonic (psh) function. We say that $u$ is $\theta$-plurisubharmonic ($\theta$-psh) if it is quasi-psh and $\theta_u:=\theta+dd^cu\geq 0$ in the sense of currents on $X$.
	We write PSH$(X,\theta)$ as the space of all $\theta$-psh functions on $X$ which are not identically $-\infty$. The class $\{\theta\}$ is big if there exists $\psi\in \mbox{PSH}(X,\theta)$ satisfying $\theta+dd^c\psi\geq \varepsilon\omega$ for some $\varepsilon>0$.

	\subsection{Singularity types, monotonicity, and capacity convergence}
	
	Let \(u,v\in\operatorname{PSH}(X,\theta)\). We write \(u\preceq v\) if
	\(u\leq v+O(1)\), and \(u\simeq v\) if both \(u\preceq v\) and
	\(v\preceq u\). The singularity type of \(u\) is denoted by \([u]\). The
	potential with minimal singularities in \(\operatorname{PSH}(X,\theta)\) is
	\[
	V_\theta:=
	\sup\{u\in\operatorname{PSH}(X,\theta):u\leq0\}.
	\]
	
	We recall the non-pluripolar Monge-Amp\`ere product in big classes, introduced
	by Boucksom-Eyssidieux-Guedj-Zeriahi \cite{BEGZ10} as a global counterpart
	of the Bedford-Taylor construction \cite{BT76,BT82,BT87}.
	
	\begin{definition}[{\cite{BEGZ10}}]
		\label{def:non-pluripolar-MA-measure}
		Let \(\theta\) be a smooth closed real \((1,1)\)-form representing a big class,
		and let \(u\in\operatorname{PSH}(X,\theta)\). The non-pluripolar
		Monge-Amp\`ere measure of \(u\) is
		\[
		\theta_u^n
		=
		\big\langle(\theta+dd^cu)^n\big\rangle
		:=
		\lim_{k\to+\infty}
		\mathbf 1_{\{u>V_\theta-k\}}
		\theta_{\max\{u,V_\theta-k\}}^n .
		\]
	\end{definition}
	
	This product does not charge pluripolar sets, is local in the plurifine
	topology, and is multilinear; see \cite[Proposition~1.4]{BEGZ10}.
	
	\begin{theorem}[
		{\cite[Theorem~1.1, Proposition~2.1]{DDL18b}}]
		\label{thm:mixed-monotonicity}
		Let \(X\) be a compact K\"ahler manifold of dimension \(n\), and let
		\(\theta_1,\ldots,\theta_n\) be smooth closed real \((1,1)\)-forms
		representing big classes. For each \(1\leq i\leq n\), let
		\(u_i,v_i\in\operatorname{PSH}(X,\theta_i)\). If
		\[
		u_i\preceq v_i,\qquad 1\leq i\leq n,
		\]
		then
		\[
		\int_X
		\theta_{1,u_1}\wedge\cdots\wedge\theta_{n,u_n}
		\leq
		\int_X
		\theta_{1,v_1}\wedge\cdots\wedge\theta_{n,v_n},
		\]
		where the products are understood in the non-pluripolar sense. In particular,
		if \(u_i\simeq v_i\) for all \(i\), then the two mixed masses are equal.
	\end{theorem}
	
	The total mass monotonicity for the Monge-Amp\`ere product was proved by
	Witt Nystr\"om \cite{WN19}. The mixed version, together with invariance under
	replacing each potential by one with the same singularity type, is due to
	Darvas-Di Nezza-Lu \cite[Theorem~1.1, Proposition~2.1]{DDL18b}; see also
	\cite[Proposition~3.1]{DDL25}.
	
	Following Bedford-Taylor \cite{BT82} and its compact K\"ahler formulation in
	\cite{GZ05}, the Monge-Amp\`ere capacity is defined by
	\[
	\operatorname{Cap}_\omega(E)
	:=
	\sup\left\{
	\int_E(\omega+dd^c\rho)^n:
	\rho\in\operatorname{PSH}(X,\omega),\ 0\leq\rho\leq1
	\right\}.
	\]
	We say that \(u_j\to u\) in \(\operatorname{Cap}_\omega\), or in
	Monge-Amp\`ere capacity, if for every \(\varepsilon>0\),
	\[
	\operatorname{Cap}_\omega(\{|u_j-u|>\varepsilon\})\to0 .
	\]
	This convergence notion is standard in the stability theory of the complex
	Monge-Amp\`ere operator; see \cite{Xing96,Xing08,Xing09}.
	
	Monotone almost everywhere convergence of plurisubharmonic functions implies
	convergence in capacity \cite[Proposition~4.25]{GZ17}. A property holds
	quasi-everywhere, abbreviated q.e., if it holds outside a pluripolar set. We
	use the quasi-continuity of quasi-psh functions without further comment.
	
	The following lemma will be used to pass from \(L^1\)-convergence of quasi-psh
	functions to convergence with respect to non-pluripolar measures.
	
	\begin{lemma}[{\cite[Lemma~4.2]{DV25}}]
		\label{lem:DoVu-nonpluripolar-measure-convergence}
		Let \(A>0\), and let \(u_j,u\in\operatorname{PSH}(X,A\omega)\). Assume that
		\(u_j\to u\) in \(L^1(X)\). Let \(\nu\) be a finite positive Radon measure on
		\(X\) which does not charge pluripolar sets. Then
		\[
		\int_X \min\{|u_j-u|,1\}\,d\nu \longrightarrow 0 .
		\]
		Consequently, if \(\nu_j\to\nu\) in total variation, then for every
		\(\delta>0\),
		\[
		\nu_j\bigl(\{|u_j-u|>\delta\}\bigr)\longrightarrow 0 .
		\]
	\end{lemma}

	\subsection{Envelopes and model potentials}
	
	All envelopes are understood with upper semicontinuous regularization. For
	\(f:X\to[-\infty,+\infty]\), set
	\[
	P_\theta(f)
	:=
	\left(
	\sup\{u\in\operatorname{PSH}(X,\theta):u\leq f\}
	\right)^* .
	\]
	For two functions \(f,g\), we write
	\[
	P_\theta(f,g):=P_\theta(\min\{f,g\}).
	\]
	If \(\psi,\varphi\in\operatorname{PSH}(X,\theta)\), the relative singularity
	envelope is
	\[
	P_\theta[\psi](\varphi)
	:=
	\left(
	\lim_{C\to+\infty}P_\theta(\psi+C,\varphi)
	\right)^* .
	\]
	When \(\varphi=V_\theta\), we write
	\[
	P_\theta[\psi]:=P_\theta[\psi](V_\theta).
	\]
	
	A potential \(\phi\in\operatorname{PSH}(X,\theta)\) is called a model
	potential if
	\[
	\phi=P_\theta[\phi].
	\]
	When a model potential is used as a prescribed singularity, it will be assumed
	to have positive mass,
	\[
	\int_X\theta_\phi^n>0.
	\]
	Potentials with analytic singularity type are model; see \cite{DDL25}.

	We shall use the following contact results for envelopes and rooftop
	envelopes. Their compact K\"ahler and big-class forms are based on the
	envelope theory developed in \cite{DDL18b,GLZ19}; the relative formulations
	below are taken from \cite{DDL25}.
	
	\begin{theorem}[{\cite[Theorem~2.2]{DDL25}}]
		\label{thm:envelope-contact}
		Let \(X\) be a compact K\"ahler manifold, and let \(\theta\) be a smooth
		closed real \((1,1)\)-form representing a big class. Let
		\(h:X\to[-\infty,+\infty]\) be quasi-continuous, and assume that
		\(P_\theta(h)\in\operatorname{PSH}(X,\theta)\). Then
		\[
		\mathbf 1_{\{P_\theta(h)<h\}}\,
		\theta_{P_\theta(h)}^n
		=
		0 .
		\]
		Equivalently, \(\theta_{P_\theta(h)}^n\) is concentrated on the contact set
		\(\{P_\theta(h)=h\}\).
	\end{theorem}
	
	\begin{theorem}[{\cite[Lemma 3.7]{DDL18b}, \cite[Theorem~3.2]{DDL25}}]
		\label{thm:rooftop-contact-inequality}
		Let \(X\) be a compact K\"ahler manifold, and let \(\theta\) be a smooth
		closed real \((1,1)\)-form representing a big class. Let
		\(u,v\in\operatorname{PSH}(X,\theta)\), and assume that
		\[
		P_\theta(u,v)\in \operatorname{PSH}(X,\theta).
		\]
		Then
		\[
		\theta_{P_\theta(u,v)}^n
		\leq
		\mathbf 1_{\{P_\theta(u,v)=u\}}\,\theta_u^n
		+
		\mathbf 1_{\{P_\theta(u,v)=v\}}\,\theta_v^n .
		\]
	\end{theorem}

	\subsection{Review of relative pluripotential theory}
	
	Let \(\phi\in\operatorname{PSH}(X,\theta)\) be a positive mass model
	potential. The \(\phi\)-relative full mass class is
	\[
	\mathcal E(X,\theta,\phi)
	:=
	\left\{
	u\in\operatorname{PSH}(X,\theta):
	u\preceq\phi,\ 
	\int_X\theta_u^n=\int_X\theta_\phi^n
	\right\}.
	\]
	When \(\phi=V_\theta\), we write \(\mathcal E(X,\theta)\).
	
	We shall use the comparison principle in relative full mass classes, due to
	Darvas-Di Nezza-Lu \cite[Corollary~3.6]{DDL18b}. It extends the
	Bedford-Taylor comparison principle \cite{BT82}, its global full mass form
	\cite{GZ07}, and the big-class minimal singularity case of \cite{BEGZ10}.
	
	\begin{theorem}[{\cite[Corollary~3.6]{DDL18b}}]
		\label{thm:DDL-relative-comparison}
		Let \(X\) be a compact K\"ahler manifold, and let \(\theta\) be a smooth
		closed real \((1,1)\)-form representing a big class. Let
		\(\phi\in\operatorname{PSH}(X,\theta)\) be a positive mass model potential.
		If \(u,v\in\mathcal E(X,\theta,\phi)\), then
		\[
		\int_{\{u<v\}}\theta_v^n
		\leq
		\int_{\{u<v\}}\theta_u^n .
		\]
	\end{theorem}
	
	For a Borel set \(E\subset X\), the relative Monge-Amp\`ere capacity with
	respect to \(\phi\) is defined by
	\[
	\operatorname{Cap}_{\theta,\phi}(E)
	:=
	\sup\left\{
	\int_E\theta_u^n:
	u\in\operatorname{PSH}(X,\theta),\ \phi-1\leq u\leq\phi
	\right\}.
	\]
	This notion was introduced in \cite[Definition~4.1]{DDL25}. For
	\(\phi=V_\theta\), it recovers the big-class Monge-Amp\`ere capacity of
	\cite[Section~4.1]{BEGZ10}.
	
	The relative capacity is dominated by the fixed K\"ahler capacity in the
	following sense: there exists an increasing continuous function \(F_\phi\),
	with \(F_\phi(0)=0\), such that
	\[
	\operatorname{Cap}_{\theta,\phi}(E)
	\leq
	F_\phi(\operatorname{Cap}_\omega(E))
	\]
	for all Borel sets \(E\subset X\); see \cite[Theorem~4.1]{DDL25}. In
	particular, if \(V_\theta-T\leq u\leq V_\theta\), then
	\[
	\int_E\theta_u^n
	\leq
	T^n\operatorname{Cap}_{\theta,V_\theta}(E)
	\leq
	T^nF_{V_\theta}(\operatorname{Cap}_\omega(E)).
	\]
	
	We shall also use the relative Ko{\l}odziej-type \(L^\infty\)  estimate of
	\cite[Theorem~5.6]{DDL25}, whose proof follows the quasi-psh envelope method
	of Guedj-Lu \cite{GL25}.
	
	\begin{theorem}[{\cite[Theorem~5.6]{DDL25}}]
		\label{thm:relative-Linfty-estimate}
		Let \(\theta\) be a smooth closed real \((1,1)\)-form representing a big
		class, and let \(\phi\in\operatorname{PSH}(X,\theta)\) be a positive mass
		model potential. Suppose that \(u\in\mathcal E(X,\theta,\phi)\),
		\(\sup_Xu=0\), and
		\[
		\theta_u^n=f\omega^n,\qquad f\in L^p(X,\omega^n),\quad p>1 .
		\]
		Then \(u\) has the same singularity type as \(\phi\). More precisely,
		\[
		\phi-C\leq u\leq\phi,
		\]
		where
		\[
		C=C\left(\|f\|_{L^p},p,\omega,\theta,
		\int_X\theta_\phi^n\right)>0 .
		\]
	\end{theorem}
	
	We need the following uniform form when the background class varies.
	
	\begin{corollary}
		\label{cor:relative-Linfty-estimate-moving}
		Let \(\theta_j,\theta\) be smooth closed real \((1,1)\)-forms representing big
		classes and assume
		\[
		-\varepsilon_j\omega\leq\theta_j-\theta\leq\varepsilon_j\omega,
		\qquad
		\varepsilon_j\to0 .
		\]
		Let \(\phi_j\in\operatorname{PSH}(X,\theta_j)\) be normalized model potentials
		satisfying
		\[
		\inf_j\int_X\theta_{j,\phi_j}^n=:m_0>0 .
		\]
		Suppose that \(u_j\in\mathcal E(X,\theta_j,\phi_j)\), \(\sup_Xu_j=0\), and
		\[
		\theta_{j,u_j}^n=f_j\omega^n,\qquad
		\sup_j\|f_j\|_{L^p(X,\omega^n)}\leq M
		\]
		for some \(p>1\). Then there exists \(C>0\), independent of \(j\), such that
		\[
		\phi_j-C\leq u_j\leq\phi_j .
		\]
	\end{corollary}
	
	\begin{proof}
		By \Cref{thm:relative-Linfty-estimate}, \(u_j\) has the same singularity type
		as \(\phi_j\) for each \(j\). It remains to make the a priori constants
		uniform.
		
		Choose a K\"ahler form \(\Omega=A\omega\) such that \(\theta_j\leq\Omega\) for
		all large \(j\). Then every normalized \(\theta_j\)-psh function is a
		normalized \(\Omega\)-psh function. Since the classes \(\{\theta_j\}\)
		converge to the big class \(\{\theta\}\), they are uniformly big: after
		choosing \(\rho\in\operatorname{PSH}(X,\theta)\) with
		\(\theta+dd^c\rho\geq2a\omega\), one has
		\[
		\theta_j+dd^c\rho\geq a\omega
		\]
		for all large \(j\).
		
		Put \(\mu_j:=f_j\omega^n\), and let \(q=p/(p-1)\). For \(r>n\), define
		\[
		A_r(\mu_j):=
		\sup_{\substack{h\in\operatorname{PSH}(X,\Omega)\\ \sup_Xh=0}}
		\int_X(-h)^r\,d\mu_j .
		\]
		H\"older's inequality and the uniform \(L^s\)-bounds for normalized
		\(\Omega\)-psh functions give
		\[
		A_r(\mu_j)
		\leq
		M\left(
		\sup_{\substack{h\in\operatorname{PSH}(X,\Omega)\\ \sup_Xh=0}}
		\int_X |h|^{rq}\,\omega^n
		\right)^{1/q}
		\leq C_r ,
		\]
		with \(C_r\) independent of \(j\). Moreover,
		\[
		m_0\leq\mu_j(X)\leq
		M\left(\int_X\omega^n\right)^{1/q}.
		\]
		Thus the quantities entering the a priori part of
		\cite[Theorem~5.6]{DDL25} are uniformly controlled. Repeating that argument
		gives a constant \(C>0\), independent of \(j\), such that
		\(u_j\geq\phi_j-C\). The upper bound follows from the normalization and the
		model property of \(\phi_j\); see \cite[Lemma~3.2]{DDL25}.
	\end{proof}
	
	We shall  use the solvability theorem for Monge-Amp\`ere equations
	with prescribed singularity type.
	
	\begin{theorem}[{\cite[Theorem~4.28]{DDL18b}, \cite[Theorem~5.4]{DDL25}}]
		\label{thm:relative-solvability}
		Let \(X\) be a compact K\"ahler manifold, and let \(\theta\) be a smooth
		closed real \((1,1)\)-form representing a big class. Let
		\(\phi\in\operatorname{PSH}(X,\theta)\) be a positive mass model potential,
		that is,
		\[
		\phi=P_\theta[\phi],
		\qquad
		\int_X\theta_\phi^n>0 .
		\]
		Let \(\mu\) be a non-pluripolar positive Radon measure on \(X\) satisfying
		\[
		\mu(X)=\int_X\theta_\phi^n .
		\]
		Then there exists \(u\in\mathcal E(X,\theta,\phi)\) such that
		\[
		\theta_u^n=\mu .
		\]
		Moreover, the solution is unique up to an additive constant.
	\end{theorem}

	\subsection{Tame measures, singularity exponents, and delta invariants.}
	\label{subsec:tame-measures-delta}
	
	We recall the notation from Darvas-Zhang \cite{DZ24} used in
	\Cref{sec:DZ-stability}. Let \(\chi\) be smooth, let \(\psi\) be quasi-psh,
	and assume
	\[
	\int_X e^{\chi-\psi}\omega^n<+\infty .
	\]
	After a harmless normalization, set
	\[
	\mu:=e^{\chi-\psi}\omega^n .
	\]
	Such a measure will be called tame. For a quasi-psh function \(\varphi\), its
	complex singularity exponent with respect to \(\mu\) is
	\[
	c_\mu[\varphi]
	:=
	\sup\left\{
	\lambda>0:
	\int_X e^{-\lambda\varphi}\,d\mu<+\infty
	\right\}.
	\]
	It depends only on the singularity type of \(\varphi\). We use the convention
	\(c_\mu[\varphi]=+\infty\) if the integral is finite for all \(\lambda>0\).
	In particular,
	\[
	c_\mu[V_\theta]
	=
	\sup\left\{
	\lambda>0:
	\int_X e^{-\lambda V_\theta}\,d\mu<+\infty
	\right\}.
	\]
	
	Let \(\pi:Y\to X\) be a smooth birational model, and let \(E\subset Y\) be a
	prime divisor. For a quasi-psh function \(\varphi\) on \(X\), denote by
	\(\nu(\varphi,E)\) the generic Lelong number of \(\pi^*\varphi\) along \(E\).
	The log discrepancy is
	\[
	A_X(E):=1+\operatorname{ord}_E(K_Y-\pi^*K_X),
	\]
	and we set
	\[
	A_{\chi,\psi}(E):=
	A_X(E)+\nu(\chi,E)-\nu(\psi,E).
	\]
	The integrability assumption on \(e^{\chi-\psi}\omega^n\) implies
	\(A_{\chi,\psi}(E)>0\) for all such \(E\).
	
	For a prime divisor \(E\subset Y\xrightarrow{\pi}X\), define
	\[
	\tau_\theta(E)
	:=
	\sup\left\{
	t\geq0:
	\{\pi^*\theta\}-t\{E\}\ \text{is big}
	\right\}.
	\]
	The expected Lelong number of the class \(\{\theta\}\) along \(E\) is
	\[
	S_\theta(E)
	:=
	\frac1{\operatorname{Vol}(\theta)}
	\int_0^{\tau_\theta(E)}
	\operatorname{vol}\bigl(\{\pi^*\theta\}-t\{E\}\bigr)\,dt,
	\qquad
	\operatorname{Vol}(\theta):=\int_X\theta_{V_\theta}^n .
	\]
	Here \(\operatorname{vol}(\cdot)\) denotes the non-pluripolar volume of a big
	class. The delta invariant associated with \(\mu\) is
	\[
	\delta_\mu(\{\theta\})
	:=
	\inf_E
	\frac{A_{\chi,\psi}(E)}{S_\theta(E)} .
	\]
	When the class is fixed we write simply \(\delta_\mu\). In the twisted case
	\(\chi=0\), \(\mu=e^{-\psi}\omega^n\), this becomes the twisted delta invariant
	\[
	\delta_\psi(\{\theta\})
	=
	\inf_E
	\frac{A_X(E)-\nu(\psi,E)}{S_\theta(E)} .
	\]
	
	For \(\lambda>0\), define
	\[
	L^\lambda_\mu(u)
	:=
	-\frac1\lambda
	\log\int_X e^{-\lambda u}\,d\mu
	\]
	whenever the integral is finite. The \(\lambda\)-Ding functional is
	\[
	D^\lambda_\mu(u)
	:=
	L^\lambda_\mu(u)-I_\theta(u),
	\qquad
	u\in\mathcal E^1(X,\theta),
	\]
	where
	\[
	I_\theta(u)
	:=
	\frac{1}{\mbox{Vol}(\theta)({n+1})}
	\sum_{k=0}^n
	\int_X
	(u-V_\theta)\,
	\theta_u^k\wedge\theta_{V_\theta}^{n-k}.
	\]
	Here \(\mathcal E^1(X,\theta)\) denotes the finite-energy subspace of
	\(\mathcal E(X,\theta,V_\theta)\), namely the set of \(u\) with
	\(I_\theta(u)>-\infty\). The functional \(D^\lambda_\mu\) is well defined on
	\(\mathcal E^1(X,\theta)\) whenever \(\lambda<c_\mu[V_\theta]\). It is
	coercive, or proper, if there exist \(\varepsilon,C>0\) such that
	\[
	D^\lambda_\mu(u)
	\geq
	\varepsilon\bigl(\sup_Xu-I_\theta(u)\bigr)-C,
	\qquad
	u\in\mathcal E^1(X,\theta).
	\]
	Under the normalization \(\sup_Xu=0\), this is equivalent to
	\[
	D^\lambda_\mu(u)\geq -\varepsilon I_\theta(u)-C .
	\]

	\subsection[Multiplier ideals and I-singularity types]
	{Multiplier ideals and \texorpdfstring{\(I\)}{I}-singularity types}
	
	Let \(\psi\) be quasi-psh. The multiplier ideal sheaf \(\mathcal I(\psi)\) is
	defined by
	\[
	\mathcal I(\psi)_x
	:=
	\left\{
	f\in\mathcal O_{X,x}:
	|f|^2e^{-\psi}\in L^1_{\rm loc}\ \text{near }x
	\right\}.
	\]
	For \(u,v\in\operatorname{PSH}(X,\theta)\), following Darvas-Xia
	\cite{DX22,DX24}, we write \(u\preceq_I v\) if
	\[
	\mathcal I(tu)\subset \mathcal I(tv)
	\qquad\text{for all }t>0.
	\]
	We write \(u\simeq_I v\) if both \(u\preceq_I v\) and \(v\preceq_I u\). Define
	\[
	P_\theta[u]_I
	:=
	\left(
	\sup\{
	\psi\in\operatorname{PSH}(X,\theta):
	\psi\leq0,\ \psi\preceq_I u
	\}
	\right)^* .
	\]
	When \(\theta\) is fixed, we often write \(P[u]_I\). A potential \(\Phi\) is
	called \(I\)-model if
	\[
	\Phi=P_\theta[\Phi]_I .
	\]
	By \cite[Propositions~2.18 and 2.20]{DX22}, \(P_\theta[u]_I\) is \(I\)-model
	and
	\[
	u\simeq_I P_\theta[u]_I .
	\]
	Moreover, if \(u\) has analytic singularity type, then
	\[
	P_\theta[u]_I=P_\theta[u].
	\]

	\subsection[Partial I-equilibrium envelopes]
	{Partial \texorpdfstring{\(I\)}{I}-equilibrium envelopes}
	
	Let \(K\subset X\) be compact and non-pluripolar, and let \(v\in C^0(K)\).
	Following Berman-Boucksom-Witt Nystr\"om \cite{BBWN11} and Darvas-Xia
	\cite{DX24}, define the partial \(I\)-equilibrium envelope associated with
	\(u\) and \((K,v)\) by
	\[
	P_K^\theta[u]_I(v)
	:=
	\left(
	\sup\{
	\psi\in\operatorname{PSH}(X,\theta):
	\psi\leq v\ \text{q.e. on }K,\ 
	\psi\preceq_I u
	\}
	\right)^* .
	\]
	When \(\theta\) is fixed, we write \(P_K[u]_I(v)\). We also use the ordinary
	singularity-constrained envelope
	\[
	P_K^\theta[\phi](v)
	:=
	\left(
	\sup\{
	\psi\in\operatorname{PSH}(X,\theta):
	\psi\leq v\ \text{q.e. on }K,\ 
	\psi\preceq\phi
	\}
	\right)^* .
	\]
	If \(\Phi=P_\theta[u]_I\) has positive mass, then in the situations considered
	below the partial \(I\)-envelope depends only on \(\Phi\):
	\[
	P_K^\theta[u]_I(v)=P_K^\theta[\Phi](v),
	\qquad
	\Phi=P_\theta[u]_I .
	\]
	We shall use the contact property and the mass identity from \cite{DX24}: the
	Monge-Amp\`ere measure of \(P_K[u]_I(v)\) is concentrated on the contact set
	in \(K\), and
	\[
	\int_X\theta_{P_K[u]_I(v)}^n
	=
	\int_X\theta_{P[u]_I}^n .
	\]

	\subsection{Partial Bergman measures}
	
	In the quantization part, \(L\to X\) is a holomorphic line bundle with smooth
	Hermitian metric \(h\), and
	\[
	\theta=c_1(L,h).
	\]
	A fixed twisting line bundle may be included; we suppress it from the notation.
	
	For \(u\in\operatorname{PSH}(X,\theta)\), set
	\[
	V_{k,u}
	:=
	H^0\bigl(X,L^k\otimes\mathcal I(ku)\bigr),
	\qquad
	N_{k,u}:=\dim V_{k,u}.
	\]
	Let \(K\subset X\) be compact and non-pluripolar, let \(v\in C^0(K)\), and let
	\(\nu\) be a probability measure on \(K\). The corresponding \(L^2\)-norm on
	\(V_{k,u}\) is
	\[
	\|s\|^2_{k,v,\nu}
	:=
	\int_K |s|_{h^k}^2 e^{-kv}\,d\nu .
	\]
	For an orthonormal basis \(\{s_\alpha\}_{\alpha=1}^{N_{k,u}}\), define
	\[
	B^k_{v,u,\nu}
	:=
	\sum_{\alpha=1}^{N_{k,u}} |s_\alpha|_{h^k}^2 e^{-kv},
	\]
	and
	\[
	\beta^k_{v,u,\nu}
	:=
	\frac{n!}{k^n}B^k_{v,u,\nu}\,\nu .
	\]
	If \(\nu\) is Bernstein-Markov for \((K,v)\), then
	\[
	\beta^k_{v,u,\nu}
	\rightharpoonup
	\theta^n_{P_K[u]_I(v)}
	\]
	by \cite[Theorem~1.2]{DX24}.
	
	The Bernstein-Markov property means that, for every \(\varepsilon>0\), there
	exists \(C_\varepsilon>0\) such that
	\[
	\sup_K |s|_{h^k}^2e^{-kv}
	\leq
	C_\varepsilon e^{\varepsilon k}
	\int_K |s|_{h^k}^2e^{-kv}\,d\nu
	\]
	for all \(k\geq1\) and all \(s\in H^0(X,L^k)\). The same estimate then holds
	on every subspace \(V_{k,u}\).

	\subsection[The \(d_{\mathcal S}\)-metric on singularity types]
	{The \texorpdfstring{\(d_{\mathcal S}\)}{dS}-metric on singularity types}
	
	We shall use only the following consequence of the \(d_{\mathcal S}\)-geometry.
	When \(\theta=\omega\) is K\"ahler, Darvas-Di Nezza-Lu proved that
	\[
	d_{\mathcal S}([u],[v])
	\leq
	\sum_{\ell=0}^n
	\left(
	2\int_X\omega^\ell\wedge\omega_{\max(u,v)}^{n-\ell}
	-
	\int_X\omega^\ell\wedge\omega_u^{n-\ell}
	-
	\int_X\omega^\ell\wedge\omega_v^{n-\ell}
	\right)
	\leq
	C_n d_{\mathcal S}([u],[v]),
	\]
	where \(C_n>0\) depends only on \(n\); see \cite{DDL21a}. In the examples
	below, the term \(\ell=0\) already gives a uniform positive lower bound for
	\(d_{\mathcal S}\).
	
	For comparison with our capacity stability theorem, we recall the
	\(d_{\mathcal S}\)-stability theorem of Darvas-Di Nezza-Lu. Let
	\(\mathcal S_\delta(X,\theta)\) denote the set of singularity types of
	\(\theta\)-psh functions whose total non-pluripolar mass is at least
	\(\delta\).

	\begin{theorem}[{\cite[Theorem~1.4]{DDL21a}}]
		\label{thm:DDL-dS-stability}
		Let \((X,\omega)\) be a compact K\"ahler manifold, let \(\theta\) be a smooth
		closed real \((1,1)\)-form representing a big class, and let \(\delta>0\),
		\(p>1\). Let \(\phi_j,\phi\) be normalized model potentials and let
		\(f_j,f\geq0\). Assume that:
		\begin{enumerate}
			\item[\textup{(i)}] \([\phi_j],[\phi]\in\mathcal S_\delta(X,\theta)\) and
			\(d_{\mathcal S}([\phi_j],[\phi])\to0\);
			\item[\textup{(ii)}] \(f_j\to f\) in \(L^1(X,\omega^n)\) and
			\(\sup_j\|f_j\|_{L^p(X,\omega^n)}<+\infty\);
			\item[\textup{(iii)}]
			\[
			\int_X f_j\omega^n=\int_X\theta_{\phi_j}^n,
			\qquad
			\int_X f\omega^n=\int_X\theta_\phi^n .
			\]
		\end{enumerate}
		Let \(u_j\in\mathcal E(X,\theta,\phi_j)\) and
		\(u\in\mathcal E(X,\theta,\phi)\) be the normalized solutions
		\[
		\theta_{u_j}^n=f_j\omega^n,\quad \sup_Xu_j=0,
		\qquad
		\theta_u^n=f\omega^n,\quad \sup_Xu=0.
		\]
		Then \(u_j\to u\) in \(\operatorname{Cap}_\omega\).
	\end{theorem}

	\begin{remark}
		\label{rem:DoVu-quantitative-stability}
		Do-Vu \cite{DV25} obtained quantitative refinements of this theorem, allowing
		both the background class and the prescribed singularity type to vary, with
		right-hand side measures converging in total variation.
	\end{remark}

	\section{Stability of the Monge-Amp\`ere equations}\label{sec:stability}
	
	In this section, we prove the direct implication part of \cref{thm:intro-optimal-capacity-topology}.  We first prepare two lemma.
	The following lemma has essentially been discovered by Darvas-Xia \cite[Proposition 3.6]{DX24}.
	
	\begin{lemma}
		\label{lem:strict-subbarrier-slope-gap}
		Let $(X,\omega)$ be compact K\"ahler, and let $\theta$ be big. Put
		$V:=V_\theta$. Let $w\in\operatorname{PSH}(X,\theta)$ satisfy $w\le V$ and $\int_X\theta_w^n>0$.
		Then there exist $\chi\in\operatorname{PSH}(X,\theta)$ and constants
		$a,\sigma>0$ such that
		\[
		\theta_\chi\ge a\omega,\qquad
		\chi\le w-\sigma(V-w).
		\]
	\end{lemma}
	
	\begin{proof}
		Choose $b>1$ close enough to $1$ so that
		\[
		P_\theta\bigl(bw-(b-1)V\bigr)\in\operatorname{PSH}(X,\theta),
		\]
		which follows from \cite[Theorem 3.3]{DDL25} and the positivity of
		$\int_X\theta_w^n$. Set
		\[
		w_b:=P_\theta\bigl(bw-(b-1)V\bigr).
		\]
		Choose $\psi\in\operatorname{PSH}(X,\theta)$ and $\varepsilon>0$ such that
		$\theta_\psi\ge\varepsilon\omega$. Since $V$ has minimal singularities, after
		adding a constant we may assume $\psi\le V+C_\psi$.
		
		Pick $0<\eta<(b-1)/b$ and set
		\[
		\gamma:=(b-1)-b\eta>0,\qquad
		\chi_0:=(1-\eta)w_b+\eta\psi .
		\]
		Then $\theta_{\chi_0}\ge\eta\varepsilon\omega$. Moreover, since
		$w_b\le bw-(b-1)V$,
		\[
		\begin{aligned}
			\chi_0
			&\le
			(1-\eta)(bw-(b-1)V)+\eta(V+C_\psi)  \\
			&=
			w-\gamma(V-w)+\eta C_\psi .
		\end{aligned}
		\]
		Thus $\chi:=\chi_0-\eta C_\psi-1$ satisfies the desired inequalities, with
		$a:=\eta\varepsilon$ and $\sigma:=\gamma$.
	\end{proof}
	
	We shall also need the following generalization of \cite[Theorem 4.1]{DDL25}:
	\begin{lemma}
		\label{lem:uniform-DDL-capacity-moving-models}
		Let $(X,\omega)$ be compact K\"ahler. Let $\theta_j$ be smooth closed real
		$(1,1)$-forms whose cohomology classes are big, and assume that
		$\theta_j\to\theta$ in $C^0$. Let
		$\phi_j\in {\rm PSH}(X,\theta_j)$ be normalized positive mass model potentials:
		\[
		P_{\theta_j}[\phi_j]=\phi_j,\qquad \sup_X\phi_j=0 .
		\]
		Then, after discarding finitely many indices, there exists a constant $C>0$,
		independent of $j$, such that for all Borel sets $E\subset X$,
		\[
		\operatorname{Cap}_{\theta_j,\phi_j}(E)
		\le F(\operatorname{Cap}_\omega(E))=C\,\operatorname{Cap}_\omega(E)^{1/n}.
		\]
		where we take $F(s):=Cs^{1/n}$.
	\end{lemma}
	
	\begin{proof} The proof is inspired by \cite{DDL25}.
		Discarding finitely many indices, choose $A>1$ and $C_0>0$ such that $\Theta_j:=\theta_j+A\omega$ is K\"ahler for all $j$, and
		\[
		|\theta_j^n|\le C_0\omega^n,\qquad
		\int_X\Theta_j^n\le C_0 .
		\]
		We first record a uniform mixed-mass bound. Since
		$\phi_j\in{\rm PSH}(X,\Theta_j)$ and $\Theta_j$ is K\"ahler, the total
		non-pluripolar mass satisfies
		\[
		\int_X(\Theta_j+dd^c\phi_j)^n\le \int_X\Theta_j^n\le C_0 .
		\]
		By multilinearity,
		\[
		(\Theta_j+dd^c\phi_j)^n
		=
		(\theta_{j,\phi_j}+A\omega)^n
		=
		\sum_{\ell=0}^n \binom n\ell A^{n-\ell}
		\theta_{j,\phi_j}^{\ell}\wedge \omega^{n-\ell}.
		\]
		All terms are positive measures. Hence $nA\int_X\theta_{j,\phi_j}^{\,n-1}\wedge\omega
		\le C_0$ and thus
		\[
		\int_X\theta_{j,\phi_j}^{\,n-1}\wedge\omega\le C_1
		\]
		with $C_1$ independent of $j$.
		
		It is enough, by inner regularity of the relative capacity, to prove the
		estimate for compact sets. Let $K\subset X$ be compact. If $K$ is pluripolar,
		then $\operatorname{Cap}_{\theta_j,\phi_j}(K)=0$, since non-pluripolar
		Monge-Amp\`ere measures do not charge pluripolar sets. So we can assume that $K$ is
		non-pluripolar.
		
		Let $V_K^*$ be the global $\omega$-extremal function and put
		\[
		M_K:=\sup_X V_K^*,\qquad
		h_K:=V_K^*-M_K .
		\]
		Then $M_K<+\infty$ iff $K$ is non-pluripolar (see \cite[Theorem 4.2]{GZ05}) and
		\[
		h_K\in{\rm PSH}(X,\omega),\qquad
		\sup_X h_K=0,\qquad
		-M_K\le h_K\le0,
		\]
		and $h_K=-M_K$ on $K$ outside a pluripolar set.
		
		Fix a competitor $u\in{\rm PSH}(X,\theta_j)$ with $\phi_j-1\le u\le \phi_j$.
		Then $u$ and $\phi_j$ have the same singularity type. Since
		$\theta_{j,u}^n$ does not charge pluripolar sets, we can write
		\[
		M_K\int_K\theta_{j,u}^n
		\le
		\int_X (-h_K)\theta_{j,u}^n .
		\]
		We claim that the right-hand side is uniformly bounded, independently of $K$ and $j$. Set $ S_{j,u}:=\sum_{\ell=0}^{n-1}
		\theta_{j,u}^{\ell}\wedge
		\theta_{j,\phi_j}^{\,n-1-\ell}$. By the integration-by-parts formula for non-pluripolar products
		\cite[Theorem~1.1]{Xia19}, \cite[Theorem~1.2]{Lu21},
		and \cite[Theorem~4.2]{DDL25} (see also \cite{Vu21}),
		applied by regarding \(h_K\) as the bounded difference \(h_K-0\), we obtain
		\[
		\begin{aligned}
			\int_X h_K(\theta_{j,u}^n-\theta_{j,\phi_j}^n)
			&=\int_Xh_Kdd^c(u-\phi_j)\wedge S_{j,u}=\int_X (u-\phi_j)dd^ch_K\wedge S_{j,u}\\
			&=
			\int_X (u-\phi_j)\omega_{h_K}\wedge S_{j,u}
			-
			\int_X (u-\phi_j)\omega\wedge S_{j,u}  \\
			&\ge
			-\int_X\omega_{h_K}\wedge S_{j,u},
		\end{aligned}
		\]
		where the inequality is because $-1\le u-\phi_j\le0$. The monotonicity theorem \cref{thm:mixed-monotonicity} gives
		\[
		\int_X\omega_{h_K}\wedge S_{j,u}
		=
		\int_X\omega\wedge S_{j,u}.
		\]
		Therefore
		\[
		\int_X (-h_K)\theta_{j,u}^n
		\le
		\int_X (-h_K)\theta_{j,\phi_j}^n
		+
		\int_X\omega\wedge S_{j,u}.
		\]
		
		We bound the two terms separately. Since $\phi_j$ is a normalized
		$\theta_j$-model potential, \cite[Theorem 3.2]{DDL25} gives
		\[
		\theta_{j,\phi_j}^n
		\le
		\mathbf 1_{\{\phi_j=0\}}\theta_j^n .
		\]
		Thus, using $|\theta_j^n|\le C_0\omega^n$,
		\[
		\int_X (-h_K)\theta_{j,\phi_j}^n
		\le
		C_0\int_X(-h_K)\omega^n
		\le C_2,
		\]
		where $C_2$ is uniform, because normalized $\omega$-psh functions are
		uniformly bounded in $L^1(X,\omega^n)$.
		
		For the second term, since $u$ and $\phi_j$ have the same singularity type,
		the monotonicity \cref{thm:mixed-monotonicity} gives
		\[
		\begin{aligned}
			\int_X\omega\wedge S_{j,u}
			&=
			\sum_{\ell=0}^{n-1}
			\int_X
			\omega\wedge\theta_{j,u}^{\ell}
			\wedge\theta_{j,\phi_j}^{\,n-1-\ell}  \\
			&=
			n\int_X\omega\wedge\theta_{j,\phi_j}^{\,n-1}
			\le C_3 .
		\end{aligned}
		\]
		Hence $\int_X (-h_K)\theta_{j,u}^n\le C_4$ uniformly in $j,u,K$. Therefore
		\[
		M_K\int_K\theta_{j,u}^n\le C_4 .
		\]
		
		Taking the supremum over all competitors $u$ gives
		\[
		M_K\,\operatorname{Cap}_{\theta_j,\phi_j}(K)\le C_4 .
		\]
		
		If $M_K\ge1$, then using $0\le V_K^*\le M_K$ and the fact that $\omega_{V_K^*}^n$ is supported on
		$K$ (cf. \cite[Theorem 4.2]{GZ05}), we have
		\[
		M_K^{-n}\int_X\omega^n=M_K^{-n}\int_X\omega_{V_K^*}^n=M_K^{-n}\int_K\omega_{V_K^*}^n\le \operatorname{Cap}_\omega(K),
		\]
		hence
		\[
		\operatorname{Cap}_{\theta_j,\phi_j}(K)
		\le
		C\,\operatorname{Cap}_\omega(K)^{1/n}.
		\]
		
		If $M_K<1$, then $V_K^*$ is an admissible test function for
		$\operatorname{Cap}_\omega(K)$, and since $\omega_{V_K^*}^n$ is supported on
		$K$,
		\[
		\int_X\omega^n=\int_X\omega_{V_K^*}^n=\int_K\omega_{V_K^*}^n\le \operatorname{Cap}_\omega(K).
		\]
		On the other hand,
		\[
		\operatorname{Cap}_{\theta_j,\phi_j}(K)
		\le
		\int_X\theta_{j,\phi_j}^n
		\le C,
		\]
		so the same estimate also follows in this case.
		
		Finally, by inner regularity of $\operatorname{Cap}_{\theta_j,\phi_j}$ \cite[Lemma 4.2]{DDL25},
		the estimate extends from compact sets  $K$ to all Borel sets $E$.
	\end{proof}
	
	We now prove the main stability theorem.
	\begin{theorem}[=direct implication part of \cref{thm:intro-optimal-capacity-topology}]
		\label{thm:moving-background-capacity-stability}
		Let $(X,\omega)$ be a compact K\"ahler manifold of dimension $n$.
		Let $\theta_j,\theta$ be smooth closed real $(1,1)$-forms whose
		cohomology classes are big. Assume that
		\[
		-\varepsilon_j\omega\le \theta_j-\theta\le \varepsilon_j\omega,
		\qquad
		\varepsilon_j\to0 .
		\]
		Let $\phi_j\in\operatorname{PSH}(X,\theta_j), \phi\in\operatorname{PSH}(X,\theta)$ be normalized positive mass model potentials such that
		\[
		\sup_X\phi_j=\sup_X\phi=0,
		\qquad
		\phi_j\to\phi
		\quad\text{in } \operatorname{Cap}_\omega .
		\]
		Let $\mu_j,\mu$ be non-pluripolar positive Radon measures satisfying $ \|\mu_j-\mu\|_{TV}\to0$ and
		\[
		\mu_j(X)=\int_X\theta_{j,\phi_j}^n,
		\qquad
		\mu(X)=\int_X\theta_\phi^n>0,
		\]
		Let $\varphi_j\in\mathcal{E}(X,\theta_j,\phi_j)$, $\varphi\in\mathcal{E}(X,\theta,\phi)$ be the unique normalized solutions of
		\[
		\theta_{j,\varphi_j}^n=\mu_j,\quad \sup_X\varphi_j=0,
		\qquad
		\theta_\varphi^n=\mu,\quad \sup_X\varphi=0.
		\]
		Then $\varphi_j\to\varphi$ in $\operatorname{Cap}_\omega$.
	\end{theorem}
	
	\begin{proof}
		Throughout the proof, we write $\theta_{j,w}:=\theta_j+dd^cw,
		\theta_w:=\theta+dd^cw$ and set \(V_j:=V_{\theta_j}\), \(V:=V_\theta\).
		
		\medskip
		\noindent\textbf{Step 1. $L^1$-convergence of $\varphi_j$.}
		Take an arbitrary subsequence. Since the forms $\theta_j$ are uniformly
		bounded above by a fixed multiple of $\omega$ and $\sup_X\varphi_j=0$, compactness of
		normalized quasi-psh functions gives, after passing to a further subsequence,
		$\varphi_j\to v$ in $L^1(X)$. As $\theta_j\to\theta$ continuously, we have
		$v\in\operatorname{PSH}(X,\theta)$.
		
		Since $\phi_j$ is model and $\varphi_j\in\mathcal E(X,\theta_j,\phi_j)$ is normalized,
		we clearly have $\varphi_j\le\phi_j$. As $\phi_j\to\phi$ in capacity,
		hence in $L^1$, we obtain $v\le\phi$. Thus $v\preceq\phi$, and by monotonicity
		of total non-pluripolar masses \cref{thm:mixed-monotonicity},
		\[
		\int_X\theta_v^n\le\int_X\theta_\phi^n .
		\]
		
		Set $\nu:=\mu+\sum_{k=1}^\infty2^{-k}\mu_k$. Then $\nu$ is a finite
		non-pluripolar Radon measure. Write $\mu_j=f_j\nu$ and $\mu=f\nu$ with $f_j,f$ bounded measurable functions. Since
		$\|\mu_j-\mu\|_{\rm TV}\to0$, we have $f_j\to f$ in $L^1(X,\nu)$.
		
		Fix $s>0$. For $j$ large enough, $\theta+s\omega-\theta_j\ge0$. Hence
		$\varphi_j\in\operatorname{PSH}(X,\theta+s\omega)$ and, by the multilinearity of non-pluripolar products,
		\[
		(\theta+s\omega+dd^c\varphi_j)^n
		\ge
		(\theta_j+dd^c\varphi_j)^n
		=
		f_j\nu .
		\]
		Applying \cite[Lemma 5.10]{DDL25} with the fixed background form
		$\theta+s\omega$ gives
		\[
		(\theta+s\omega+dd^cv)^n\ge f\nu=\mu .
		\]
		Letting $s\downarrow0$ and using the multilinearity of non-pluripolar products again,
		we get $\theta_v^n\ge\mu$. Therefore
		\[
		\int_X\theta_v^n\ge\mu(X)=\int_X\theta_\phi^n .
		\]
		Combining the two mass inequalities yields $\int_X\theta_v^n=\int_X\theta_\phi^n$
		and $\theta_v^n=\mu$. Hence $v\in\mathcal E(X,\theta,\phi)$ and solves the same
		equation as $\varphi$.
		
		Finally, by uniqueness in $\mathcal E(X,\theta,\phi)$, we get
		$v=\varphi$. Since every subsequence has a further subsequence converging to $\varphi$ in
		$L^1$, the whole sequence satisfies $\varphi_j\to \varphi$ in $L^1(X)$.
		
		\medskip
		\noindent\textbf{Step 2. Capacity stability of moving minimal potentials.}
		We first show that \(V_j\to V\) in
		\(\operatorname{Cap}_\omega\). Choose
		\(\chi_0\in\operatorname{PSH}(X,\theta)\), shifted so that \(\chi_0\le V\), and
		\(a_0>0\) such that \(\theta_{\chi_0}\ge4a_0\omega\). Put
		\(\lambda_j:=\varepsilon_j/a_0\). For \(j\) large, \(0<\lambda_j<1/2\), and
		\[
		(1-\lambda_j)V+\lambda_j\chi_0\in\operatorname{PSH}(X,\theta_j),
		\qquad
		(1-\lambda_j)V_j+\lambda_j\chi_0\in\operatorname{PSH}(X,\theta).
		\]
		Both functions are \(\le0\), hence
		\[
		(1-\lambda_j)V+\lambda_j\chi_0\le V_j,
		\qquad
		(1-\lambda_j)V_j+\lambda_j\chi_0\le V.
		\]
		Thus, writing \(W:=V-\chi_0\ge0\),
		\[
		V-\lambda_j W\le V_j\le V+\frac{\lambda_j}{1-\lambda_j}W .
		\]
		Since \(\operatorname{Cap}_\omega(\{W>S\})\to0\) as \(S\to+\infty\), we get $V_j\to V$ in $\operatorname{Cap}_\omega$.
		
		\medskip
		\noindent\textbf{Step 3. Some measure theoretic estimates.}
		By Step 1, we have $\varphi_j\to\varphi$ in $L^1(X)$. Define
		\[
		M_j(\delta):=
		\mu_j(\{\varphi_j<\varphi-\delta/2\}).
		\]
		Then
		\begin{equation}
			\label{eq:general-Mj-delta}
			M_j(\delta)\to0 .
		\end{equation}
		Indeed,
		\[
		\begin{aligned}
			M_j(\delta)
			&\le
			\mu(\{\varphi_j<\varphi-\delta/2\})+\|\mu_j-\mu\|\\
			&\le
			c_\delta^{-1}
			\int_X\min\{|\varphi_j-\varphi|,1\}\,d\mu+\|\mu_j-\mu\|,
		\end{aligned}
		\]
		where $c_\delta:=\min\{\delta/2,1\}$.
		By \cref{lem:DoVu-nonpluripolar-measure-convergence}, the first term tends to zero, and the second term tends to zero by assumption.
		
		Next, for $k>2$, set
		\[
		N_j(k):=\mu_j(\{\varphi_j<\phi_j-k+1\}).
		\]
		Since $\phi_j\le0$, we have
		\[
		\{\varphi_j<\phi_j-k+1\}\subset \{\varphi_j<-k+1\}.
		\]
		Thus
		\[
		\begin{aligned}
			N_j(k)
			&\le
			\mu(\{\varphi_j<-k+1\})+\|\mu_j-\mu\|\\
			&\le
			\mu(\{\varphi<-k+2\})
			+
			\int_X\min\{|\varphi_j-\varphi|,1\}\,d\mu
			+
			\|\mu_j-\mu\|.
		\end{aligned}
		\]
		Again using \cref{lem:DoVu-nonpluripolar-measure-convergence}, we obtain for each $\varepsilon>0$, there are constants $j_0,k_0>0$ such that for each $j\ge j_0,k\ge k_0$,
		\begin{equation}
			\label{eq:general-relative-tail-mass}
			N_{j}(k)<\varepsilon,
		\end{equation}
		because $\mu$ does not charge the pluripolar set $\{\varphi=-\infty\}$.
		Moreover, since the \(\theta_j\)'s are uniformly dominated by a fixed multiple of
		\(\omega\), the normalized quasi-psh functions \(\varphi_j\) have uniformly bounded
		\(L^1\)-norm. The standard Chern-Levine-Nirenberg estimate (cf. \cite[Proposition 9.10]{GZ17}) yields a uniform constant $C$, independent of $j$, such that
		\begin{equation}
			\label{eq:moving-background-absolute-tail}
			\operatorname{Cap}_\omega(\{\varphi_j<-S\})\le \frac{C}{S} .
		\end{equation}
		By \cref{lem:uniform-DDL-capacity-moving-models}, for every fixed \(T>1\), there exists an increasing continuous
		function \(\Gamma_T\), independent of \(j\), with \(\Gamma_T(0)=0\), such that for
		all large \(j\), all Borel sets \(E\), and all
		\(w\in\operatorname{PSH}(X,\theta_j)\) satisfying \(V_j-T\le w\le V_j\),
		\begin{equation}
			\label{eq:moving-background-uniform-domination}
			\int_E\theta_{j,w}^n
			\le
			\Gamma_T(\operatorname{Cap}_\omega(E)).
		\end{equation}
		
		\medskip
		\noindent\textbf{Step 4. A fixed barrier for the limit solution.}
		Fix \(\delta>0\). Let
		\[
		A_j:=\{\varphi_j<\varphi-\delta\}.
		\]
		Choose \(\chi\in\operatorname{PSH}(X,\theta)\) and constants \(a,\sigma>0\) such
		that
		\[
		\theta_\chi\ge a\omega,\qquad
		\chi\le\varphi-\sigma(V-\varphi).
		\]
		For \(R>0\), set \(G_R:=\{\chi>\varphi-R\}\). Then
		\begin{equation}
			\label{eq:moving-background-GR}
			\operatorname{Cap}_\omega(X\setminus G_R)\to0
			\quad (R\to+\infty).
		\end{equation}
		Fix \(R\), choose \(0<b<a\), and put
		\[
		\tau:=\min\left\{\frac12,\frac{\delta}{8(R+b)}\right\}.
		\]
		The number \(\tau\) depends only on \(\delta,R\).
		
		\medskip
		\noindent\textbf{Step 5. The capacity estimate on the good set.}
		Let \(k>2\) and \(T>1\). Define
		\[
		Y_j^k:=\max(\varphi_j,\phi_j-k),
		\qquad
		U_j^{T,k}:=\max(Y_j^k,V_j-T).
		\]
		Then \(U_j^{T,k}\in\mathcal E(X,\theta_j,V_j)\). Define $K_{j,R,k,T}$ to be the Borel set
		\[
		K_{j,R,k,T}:=A_j\cap G_R\cap\{\varphi_j>\phi_j-k+1\}\cap\{\varphi_j>V_j-T\}.
		\]
		Take \(\rho\in\operatorname{PSH}(X,\omega)\), \(0\le\rho\le1\), and write
		\(\omega_\rho:=\omega+dd^c\rho\). Define
		\[
		\psi_\rho:=(1-\tau)\varphi+\tau\chi+\tau b\rho-\frac{\delta}{2}-\tau b .
		\]
		Then \(\psi_\rho\le\varphi-\delta/2\). For \(j\) large enough,
		\[
		\begin{aligned}
			\theta_{j,\psi_\rho}
			&=
			(1-\tau)\theta_\varphi+\tau\theta_\chi+\tau b\,dd^c\rho+(\theta_j-\theta)\\
			&\ge
			\tau(a-b)\omega+\tau b\omega_\rho-\varepsilon_j\omega
			\ge
			\tau b\omega_\rho .
		\end{aligned}
		\]
		Thus \(\psi_\rho\in\operatorname{PSH}(X,\theta_j)\). Also, on \(G_R\),
		\[
		\psi_\rho\ge\varphi-\frac{\delta}{2}-\tau(R+b)\ge\varphi-\frac{5\delta}{8}.
		\]
		Since \(K_{j,R,k,T}\subset A_j\cap G_R\), we have \(\psi_\rho>\varphi_j\) on \(K_{j,R,k,T}\).
		Moreover, on \(K_{j,R,k,T}\), the inequalities
		\[
		\varphi_j>\phi_j-k+1>\phi_j-k,\qquad
		\varphi_j>V_j-T
		\]
		give \(U_j^{T,k}=\varphi_j\). Hence \(K_{j,R,k,T}\subset\{U_j^{T,k}<\psi_\rho\}\).
		
		Set \(H_{j,\rho}^{T,k}:=\max(U_j^{T,k},\psi_\rho)\). Since
		\(U_j^{T,k}\le H_{j,\rho}^{T,k}\le V_j\), both potentials lie in
		\(\mathcal E(X,\theta_j,V_j)\). The comparison principle \cref{thm:DDL-relative-comparison} gives
		\[
		\int_{\{U_j^{T,k}<\psi_\rho\}}\theta_{j,H_{j,\rho}^{T,k}}^n
		\le
		\int_{\{U_j^{T,k}<\psi_\rho\}}\theta_{j,U_j^{T,k}}^n .
		\]
		By plurifine locality, \(H_{j,\rho}^{T,k}=\psi_\rho\) on
		\(\{U_j^{T,k}<\psi_\rho\}\). Hence
		\begin{equation}
			\label{eq:moving-background-main-cap-estimate}
			(\tau b)^n\int_{K_{j,R,k,T}}\omega_\rho^n
			\le
			\int_{\{U_j^{T,k}<\psi_\rho\}}\theta_{j,U_j^{T,k}}^n .
		\end{equation}
		
		\medskip
		\noindent\textbf{Step 6. We split the right-hand side into three pieces.}
		On
		\[
		\{Y_j^k>V_j-T\}\cap\{\varphi_j>\phi_j-k\}
		\]
		we have \(U_j^{T,k}=\varphi_j\). Therefore, by plurifine locality and
		\(\psi_\rho\le\varphi-\delta/2\),
		\[
		\begin{aligned}
			&\int_{\{U_j^{T,k}<\psi_\rho\}
				\cap\{Y_j^k>V_j-T\}
				\cap\{\varphi_j>\phi_j-k\}}
			\theta_{j,U_j^{T,k}}^n                                      \\
			&\qquad\le
			\mu_j(\{\varphi_j<\varphi-\delta/2\})
			=
			M_j(\delta).
		\end{aligned}
		\]
		
		Next consider the relative contact part
		\[
		C_{j,T,k,\rho}:=
		\{U_j^{T,k}<\psi_\rho\}
		\cap\{Y_j^k>V_j-T\}
		\cap\{\varphi_j\le\phi_j-k\}.
		\]
		On \(\{Y_j^k>V_j-T\}\), \(U_j^{T,k}=Y_j^k\). Hence
		\[
		\int_{C_{j,T,k,\rho}}\theta_{j,U_j^{T,k}}^n
		\le
		\int_{\{\varphi_j\le\phi_j-k\}}\theta_{j,Y_j^k}^n .
		\]
		Since \(\varphi_j\le Y_j^k\le\phi_j\) and
		\(\varphi_j\in\mathcal E(X,\theta_j,\phi_j)\), we have
		\(Y_j^k\in\mathcal E(X,\theta_j,\phi_j)\). The total masses of
		\(\theta_{j,Y_j^k}^n\) and \(\theta_{j,\varphi_j}^n\) are equal, and the two
		measures coincide on the plurifine open set \(\{\varphi_j>\phi_j-k\}\). Thus
		\[
		\int_{\{\varphi_j\le\phi_j-k\}}\theta_{j,Y_j^k}^n
		=
		\int_{\{\varphi_j\le\phi_j-k\}}\theta_{j,\varphi_j}^n
		\le
		N_j(k).
		\]
		
		It remains to estimate
		\[
		\Gamma_{j,T,k,\rho}:=
		\int_{\{U_j^{T,k}<\psi_\rho\}\cap\{Y_j^k\le V_j-T\}}
		\theta_{j,U_j^{T,k}}^n .
		\]
		Fix a small number \(\eta>0\) and put
		\(Q_{j,\eta}:=\{|V_j-V|\le\eta\}\). On
		\[
		\{U_j^{T,k}<\psi_\rho\}\cap\{Y_j^k\le V_j-T\}\cap Q_{j,\eta}
		\]
		we have \(V_j-T<\psi_\rho\), hence \(V-\eta-T<\psi_\rho\). The construction of $\psi_\rho$ and \cref{lem:strict-subbarrier-slope-gap} gives
		\[
		V-\eta-T< \psi_\rho\le
		\varphi-\tau\sigma(V-\varphi)-\frac{\delta}{2}.
		\]
		Therefore
		\[
		\varphi>V-\frac{T+\eta-\delta/2}{1+\tau\sigma}.
		\]
		Since \(\varphi\le\phi\), while \(Y_j^k\le V_j-T\) implies
		\(\phi_j\le V_j-T+k\le V+\eta-T+k\), we get
		\[
		\phi_j<\phi-L_{T,k,\eta},
		\]
		where
		\[
		L_{T,k,\eta}:=
		T-k-\eta-\frac{T+\eta-\delta/2}{1+\tau\sigma}.
		\]
		For fixed \(R,\delta,k,\eta\), we have \(L_{T,k,\eta}\to+\infty\) linearly as
		\(T\to+\infty\). Hence, for \(T\) large enough so that \(L_{T,k,\eta}>0\),
		\[
		\{U_j^{T,k}<\psi_\rho\}\cap\{Y_j^k\le V_j-T\}
		\subset
		\{\phi_j<\phi-L_{T,k,\eta}\}\cup (X\setminus Q_{j,\eta}).
		\]
		Using \eqref{eq:moving-background-uniform-domination} and \cref{lem:uniform-DDL-capacity-moving-models}, we obtain
		\begin{equation}
			\label{eq:moving-background-Gamma-estimate}
			\Gamma_{j,T,k,\rho}
			\le
			\Gamma_T\!\left(
			\operatorname{Cap}_\omega(\{\phi_j<\phi-L_{T,k,\eta}\})
			+
			\operatorname{Cap}_\omega(X\setminus Q_{j,\eta})
			\right).
		\end{equation}
		For fixed \(T,k,\eta\) with \(L_{T,k,\eta}>0\), the right-hand side tends to zero
		as \(j\to\infty\), because \(\phi_j\to\phi\) and \(V_j\to V\) in capacity.
		
		Combining the three estimates with \eqref{eq:moving-background-main-cap-estimate},
		we get, for fixed \(R,k,T,\eta\) with \(L_{T,k,\eta}>0\),
		\[
		(\tau b)^n\int_{K_{j,R,k,T}}\omega_\rho^n
		\le
		M_j(\delta)+N_j(k)+o_j(1),
		\]
		where \(o_j(1)\to0\) uniformly in \(\rho\). Taking the supremum over
		\(\rho\) gives
		\begin{equation}
			\label{eq:moving-background-core-estimate}
			\begin{aligned}
				&\operatorname{Cap}_\omega
				\bigl(
				A_j\cap G_R
				\cap\{\varphi_j>\phi_j-k+1\}
				\cap\{\varphi_j>V_j-T\}
				\bigr)                                      \\
				&\qquad\le
				(\tau b)^{-n}\bigl(M_j(\delta)+N_j(k)+o_j(1)\bigr).
			\end{aligned}
		\end{equation}
		
		\medskip
		\noindent\textbf{Step 7. We choose the parameters in \(\varepsilon\)-$\delta$ language.}
		Let \(\varepsilon>0\). First choose \(R\) so large that
		\[
		\operatorname{Cap}_\omega(X\setminus G_R)<\frac{\varepsilon}{4}.
		\]
		Then \(\tau\) is fixed. Choose \(k>2\) so large that, for all large \(j\),
		\[
		\operatorname{Cap}_\omega(\{\varphi_j<-k+1\})<\frac{\varepsilon}{4},
		\qquad
		(\tau b)^{-n}N_j(k)<\frac{\varepsilon}{8}.
		\]
		This is possible by \eqref{eq:moving-background-absolute-tail} and
		\eqref{eq:general-relative-tail-mass}. Fix \(\eta>0\). Choose \(T\) so large that
		\(L_{T,k,\eta}>0\) and, for all large \(j\),
		\[
		\operatorname{Cap}_\omega(\{\varphi_j<-T\})<\frac{\varepsilon}{4}.
		\]
		Finally, by \eqref{eq:general-Mj-delta},
		\eqref{eq:moving-background-Gamma-estimate}, and
		\eqref{eq:moving-background-core-estimate}, take \(j\) large enough so that
		\[
		(\tau b)^{-n}\bigl(M_j(\delta)+N_j(k)+o_j(1)\bigr)
		<\frac{\varepsilon}{4}.
		\]
		
		For such \(j\),
		\[
		\begin{aligned}
			A_j
			\subset\;&
			(X\setminus G_R)
			\cup\{\varphi_j\le\phi_j-k+1\}
			\cup\{\varphi_j\le V_j-T\}                  \\
			&\cup
			\bigl(
			A_j\cap G_R
			\cap\{\varphi_j>\phi_j-k+1\}
			\cap\{\varphi_j>V_j-T\}
			\bigr).
		\end{aligned}
		\]
		Since \(\phi_j\le0\) and \(V_j\le0\),
		\[
		\{\varphi_j\le\phi_j-k+1\}\subset\{\varphi_j<-k+1\},
		\qquad
		\{\varphi_j\le V_j-T\}\subset\{\varphi_j<-T\}.
		\]
		The above choices imply $\operatorname{Cap}_\omega(A_j)<\varepsilon$ and thus
		\[
		\operatorname{Cap}_\omega(\{\varphi_j<\varphi-\delta\})\to0.
		\]
		
		\medskip
		\noindent\textbf{Step 8. The upper tail follows from Hartogs.}
		Since \(\varphi_j\to\varphi\) in \(L^1(X)\), the Hartogs lemma (cf. \cite[Proposition 8.4]{GZ17}) implies that
		\[
		\left(\sup_{k\ge j}\varphi_k\right)^*\searrow \varphi .
		\]
		Consequently, it follows from \cite[Proposition 4.25]{GZ17} that for every \(\delta>0\),
		\[
		\operatorname{Cap}_\omega\left(
		\left\{\left(\sup_{k\ge j}\varphi_k\right)^*>\varphi+\delta\right\}
		\right)\longrightarrow 0 .
		\]
		Since
		\[
		{\varphi_j>\varphi+\delta}
		\subset
		\left\{
		\left(\sup_{k\ge j}\varphi_k\right)^*>\varphi+\delta
		\right\},
		\]
		it follows that
		\[
		\operatorname{Cap}_\omega({\varphi_j>\varphi+\delta})\to0 .
		\]
		
		Combining this upper-tail estimate with the lower-tail estimate established above, we obtain
		\[
		\operatorname{Cap}_\omega({|\varphi_j-\varphi|>\delta})
		\le
		\operatorname{Cap}_\omega({\varphi_j>\varphi+\delta})
		+
		\operatorname{Cap}_\omega({\varphi_j<\varphi-\delta})
		\longrightarrow 0 ,
		\]
		for every \(\delta>0\). Therefore \(\varphi_j\to\varphi\) in \(\operatorname{Cap}_\omega\), and the proof is complete.
	\end{proof}
	
	As an immediate corollary, we have the following generalization of \cite[Theorem 1.4]{DDL21a}:
	\begin{corollary}
		\label{cor:capacity-stability-cap-only}
		Let $(X,\omega)$ be a compact K\"ahler manifold of dimension $n$, and let
		$\theta$ be a smooth closed real $(1,1)$-form whose cohomology class is big.
		Let $\phi_j,\phi\in\operatorname{PSH}(X,\theta)$ be normalized positive mass
		model potentials such that
		\[
		\sup_X\phi_j=\sup_X\phi=0,\qquad
		\phi_j\to\phi
		\quad\text{in }\operatorname{Cap}_\omega .
		\]
		Let $p>1$, and let $f_j,f\ge0$ satisfy
		\[
		f_j\to f \quad\text{in }L^1(X,\omega^n),
		\qquad
		\sup_j\|f_j\|_{L^p(X,\omega^n)}<+\infty ,
		\]
		with the mass normalizations $\int_X f_j\omega^n=\int_X\theta_{\phi_j}^n$ and $\int_X f\omega^n=\int_X\theta_\phi^n>0$.
		Let $  u_j\in\mathcal E(X,\theta,\phi_j),
		u\in\mathcal E(X,\theta,\phi)$ be the unique normalized solutions of
		\[
		\theta_{u_j}^n=f_j\omega^n,\quad \sup_Xu_j=0,
		\qquad
		\theta_u^n=f\omega^n,\quad \sup_Xu=0.
		\]
		Then $u_j\to u$ in $\text{Cap}_\omega$.
	\end{corollary}
	

	We give an example on Riemann surfaces to illustrate that the capacity topology is indeed much weaker than $d_{\mathcal S}$ topology, the same construction also gives numerous examples of moving divisors on higher dimension manifolds. 
	
	\begin{example}[Moving poles on \(\mathbb P^1\): capacity convergence without \(d_{\mathcal S}\)-convergence]
		\label{ex:moving-poles-P1-cap-not-dS}
		Let \(X=\mathbb P^1\), \(L=\mathcal O_{\mathbb P^1}(1)\), and let
		\(\omega=\omega_{\rm FS}=c_1(L,h_{\rm FS})\) be the Fubini-Study form normalized by
		\(\int_X\omega=1\). Use homogeneous coordinates \([Z_0:Z_1]\), and set
		\[
		D_1:=\{Z_0=0\}.
		\]
		Fix a rational number \(0<c<1\). Let \(\varepsilon_1:=0\), and let
		\(\varepsilon_j\in\mathbb C^*\), \(j\ge2\), satisfy \(\varepsilon_j\to0\).
		Put
		\[
		a_j:=(1+|\varepsilon_j|^2)^{-1/2},
		\qquad
		s_j:=a_j(Z_0-\varepsilon_j Z_1)\in H^0(\mathbb P^1,\mathcal O(1)),
		\]
		and
		\[
		D_j:=\{s_j=0\}=\{Z_0-\varepsilon_jZ_1=0\}.
		\]
		Thus the points \(D_j\to D_1\). With respect to the Fubini-Study metric,
		\[
		|s_j|_{h_{\rm FS}}^2([Z_0:Z_1])
		=
		\frac{|Z_0-\varepsilon_jZ_1|^2}
		{(1+|\varepsilon_j|^2)(|Z_0|^2+|Z_1|^2)}
		\le1,
		\]
		and \(\sup_X |s_j|_{h_{\rm FS}}=1\). Define
		\[
		u_j:=c\log |s_j|_{h_{\rm FS}}^2,
		\qquad
		\phi_j:=P_\omega[u_j].
		\]
		Then \(\sup_Xu_j=0\), hence \(\sup_X\phi_j=0\) and $\phi_j$ is a model potential.
		
		\smallskip
		\noindent\textbf{Step 1: \(u_j\to u_1\) and \(\phi_j\to\phi_1\) in capacity.}
		Choose
		\[
		g_j:=
		\begin{pmatrix}
			a_j & -a_j\varepsilon_j\\
			a_j\overline{\varepsilon_j} & a_j
		\end{pmatrix}
		\in U(2),
		\]
		and let \(G_j\in{\rm PU}(2)\) be the induced automorphism of \(\mathbb P^1\).
		Then \(G_j\to{\rm id}\) smoothly, \(G_j^*\omega=\omega\) (because the Fubini-Study metric is unitary invariant), and
		\[
		G_j^*Z_0=a_j(Z_0-\varepsilon_jZ_1)=s_j.
		\]
		Since \(h_{\rm FS}\) is invariant under \(U(2)\),
		\[
		u_j=G_j^*u_1.
		\]
		We claim that \(u_j\to u_1\) in \(\operatorname{Cap}_\omega\). Indeed, fix
		\(\delta>0\). By quasi-continuity of \(u_1\), for every \(\eta>0\) there is an
		open set \(O\subset X\) with \(\operatorname{Cap}_\omega(O)<\eta\) such that
		\(u_1\) is continuous on \(X\setminus O\). Since \(G_j\to{\rm id}\), for
		\(j\gg1\),
		\[
		|u_1\circ G_j-u_1|<\delta
		\quad\text{on }X\setminus(O\cup G_j^{-1}O).
		\]
		Moreover, \(G_j\) preserves \(\omega\), hence it preserves
		\(\operatorname{Cap}_\omega\). Therefore
		\[
		\operatorname{Cap}_\omega(\{|u_j-u_1|>\delta\})
		\le
		\operatorname{Cap}_\omega(O)+
		\operatorname{Cap}_\omega(G_j^{-1}O)
		=
		2\operatorname{Cap}_\omega(O).
		\]
		Letting \(\eta\to0\) gives \(u_j\to u_1\) in capacity.
		
		The envelope is equivariant under \(G_j\). More precisely, \(G_j^*\) gives a bijection
		between the candidates defining \(P_\omega[u_1]\) and those defining
		\(P_\omega[G_j^*u_1]\), since $G_j^*\omega=\omega$. Hence
		\[
		\phi_j=P_\omega[u_j]=P_\omega[G_j^*u_1]=G_j^*P_\omega[u_1]=G_j^*\phi_1.
		\]
		Repeating the same quasi-continuity argument with \(\phi_1\) in place of
		\(u_1\), we obtain
		\[
		\phi_j\to\phi_1
		\quad\text{in }\operatorname{Cap}_\omega.
		\]
		
		\smallskip
		\noindent\textbf{Step 2: \(\phi_j\) are positive mass model potentials and the masses are constant.}
		By the Poincar\'e-Lelong formula, as distributions
		\[
		\omega+dd^cu_j=\omega+c([D_j]-\omega)=(1-c)\omega+c[D_j].
		\]
		Thus \(u_j\in\operatorname{PSH}(X,\omega)\), and it is clear that \(u_j\) has analytic
		singularity type. Hence \(\phi_j=P_\omega[u_j]\) is a model potential with the
		same analytic singularity type as \(u_j\) (cf. \cite[Lemma 3.2]{DX24}). Since non-pluripolar products do
		not charge polar sets, we get as non-pluripolar measures, $\langle\omega+dd^cu_j\rangle=(1-c)\omega$, hence
		\begin{equation}
			\label{eq:P1-moving-pole-mass}
			\int_X\omega_{\phi_j}
			=
			\int_X\omega_{u_j}
			=
			1-c
			=
			\int_X\omega_{\phi_1}.
		\end{equation}
		In particular, \(\phi_j\to\phi_1\) in capacity and the total non-pluripolar
		masses converge, in fact they are equal.
		
		\smallskip
		\noindent\textbf{Step 3: the sequence does not converge in \(d_{\mathcal S}\).}
		For \(j\ge2\), the poles \(D_j\) and \(D_1\) are distinct. Since
		\(\phi_j\simeq u_j\) and \(\phi_1\simeq u_1\), the potential
		\(\max(\phi_j,\phi_1)\) has minimal singularity type since $D_1\cap D_j=\emptyset$ in $\mathbb{P}^1$: near \(D_j\) the
		function \(\phi_1\) is locally bounded, and near \(D_1\) the function
		\(\phi_j\) is locally bounded. Therefore
		\[
		\int_X\omega_{\max(\phi_j,\phi_1)}=1.
		\]
		Using \eqref{eq:P1-moving-pole-mass}, we have
		\[
		2\int_X\omega_{\max(\phi_j,\phi_1)}
		-
		\int_X\omega_{\phi_j}
		-
		\int_X\omega_{\phi_1}
		=
		2-2(1-c)=2c.
		\]
		It follows from \cite{DDL21a,DV25} that \([\phi_j]\not\to[\phi_1]\) in \(d_{\mathcal S}\).
		
		\smallskip
		\noindent\textbf{Step 4: consequence for capacity stability of prescribed Monge-Amp\`ere equations.}
		Let \(\mu_j,\mu\) be non-pluripolar positive Radon measures such that
		\[
		\|\mu_j-\mu\|_{\rm TV}\to0,
		\qquad
		\mu_j(X)=\int_X\omega_{\phi_j}=1-c,
		\qquad
		\mu(X)=\int_X\omega_{\phi_1}=1-c.
		\]
		Let \(\varphi_j\in\mathcal E(X,\omega,\phi_j)\) and
		\(\varphi\in\mathcal E(X,\omega,\phi_1)\) be the normalized solutions
		\[
		\omega_{\varphi_j}=\mu_j,\quad \sup_X\varphi_j=0,
		\qquad
		\omega_\varphi=\mu,\quad \sup_X\varphi=0.
		\]
		By the capacity stability theorem for moving prescribed singularities,
		\[
		\varphi_j\to\varphi
		\quad\text{in }\operatorname{Cap}_\omega.
		\]
		This example shows that our capacity-stability theorem is genuinely stronger than the \(d_{\mathcal S}\)-stability theorems of \cite{DDL21a,DV25}. Indeed, in this example one has \[ \phi_j \longrightarrow \phi_1 \quad\text{in } \operatorname{Cap}_\omega, \] whereas the corresponding singularity types do not converge in the \(d_{\mathcal S}\)-topology: \[ [\phi_j]\not\longrightarrow [\phi_1] \quad\text{in } d_{\mathcal S} . \]
	\end{example}

	\section{Continuity of envelopes and optimality of the capacity topology}\label{sec-4}
	In this section, we will prove the inverse implication part of \cref{thm:intro-optimal-capacity-topology}. 
	\subsection{Continuity of envelopes with respect to convergence in capacity}
	
	We record an application of the truncation method to the stability of
	quasi-continuous envelopes. The point is that one can allow the obstacles to
	have moving prescribed singularities, provided that these singularities converge
	in capacity. The proof is close in spirit to the proof of
	\Cref{thm:moving-background-capacity-stability}: instead of trying to control directly
	the mass defect of $\max(P_\theta(h_j),\psi_\rho)$, we truncate to the minimal
	singularity type and use the slope gap barrier to force the truncation contact
	mass into a set where $\phi_j$ is much smaller than $\phi$.
	
	We shall use the following elementary observation.
	
	\begin{lemma}
		\label{lem:qpsh-upper-bound-passes-to-limit}
		Let $u_j,u\in\operatorname{PSH}(X,\theta)$ with $u_j\to u$ in $L^1(X)$.
		Let $h_j,h$ be quasi-continuous functions such that
		$h_j\to h$ in $\operatorname{Cap}_\omega$. If
		\begin{equation}
			\label{eq:upper-bound-qe}
			u_j\le h_j \quad \text{q.e. on }X
		\end{equation}
		for all $j$, then
		\begin{equation}
			u\le h \quad \text{q.e. on }X .
		\end{equation}
	\end{lemma}
	
	\begin{proof}
		It is enough to prove that, for every \(\varepsilon>0\),
		\[
		\operatorname{Cap}_\omega(\{u>h+3\varepsilon\})=0 .
		\]
		Suppose by contradiction that, for some \(\varepsilon>0\),
		\[
		E_\varepsilon:=\{u>h+3\varepsilon\}
		\]
		has positive \(\omega\)-capacity. By the definition of
		\(\operatorname{Cap}_\omega\), there exists
		\(\rho\in\operatorname{PSH}(X,\omega)\), \(0\le\rho\le1\), such that, with $\nu:=(\omega+dd^c\rho)^n$,
		we have $\nu(E_\varepsilon)>0$. The measure \(\nu\) is a finite non-pluripolar measure and, by the definition
		of capacity, $\nu(A)\le \operatorname{Cap}_\omega(A)$ for every Borel set \(A\subset X\).
		
		Since \(u_j\to u\) in \(L^1(X)\) and \(\nu\) does not charge pluripolar sets,
		\cref{lem:DoVu-nonpluripolar-measure-convergence} gives
		\[
		\int_X\min\{|u_j-u|,1\}\,d\nu\to0 .
		\]
		In particular, $\nu(\{|u_j-u|>\varepsilon\})\to0$.
		
		On the other hand, since \(h_j\to h\) in \(\operatorname{Cap}_\omega\), the
		domination \(\nu\le \operatorname{Cap}_\omega\) gives
		\[
		\nu(\{|h_j-h|>\varepsilon\})
		\le
		\operatorname{Cap}_\omega(\{|h_j-h|>\varepsilon\})
		\to0 .
		\]
		
		The inequality \(u_j\le h_j\) holds q.e. by assumption. Since \(\nu\) does not
		charge pluripolar sets, it holds \(\nu\)-a.e. Now on \(E_\varepsilon\), if
		\[
		|u_j-u|\le\varepsilon,
		\qquad
		|h_j-h|\le\varepsilon,
		\]
		then
		\[
		u_j\ge u-\varepsilon>h+2\varepsilon\ge h_j+\varepsilon,
		\]
		which contradicts \(u_j\le h_j\) q.e. Hence, there is a $\nu$-negligible set $E_\nu$ such that
		\[
		E_\varepsilon
		\subset
		\{|u_j-u|>\varepsilon\}
		\cup
		\{|h_j-h|>\varepsilon\}
		\cup E_\nu
		.
		\]
		Therefore
		\[
		\begin{aligned}
			\nu(E_\varepsilon)
			&\le
			\nu(\{|u_j-u|>\varepsilon\})
			+
			\nu(\{|h_j-h|>\varepsilon\})
			\to0 .
		\end{aligned}
		\]
		This contradicts \(\nu(E_\varepsilon)>0\). Thus
		\[
		\operatorname{Cap}_\omega(\{u>h+3\varepsilon\})=0
		\]
		for every \(\varepsilon>0\). Letting \(\varepsilon\downarrow0\) along a
		sequence gives the desired result.
	\end{proof}
	
	The following property of the envelope will be used later.
	\begin{lemma}\label{lem: envelope equality}
		\label{lem:pointwise-qe-envelope-positive-subbarrier}
		Let \(h:X\to[-\infty,+\infty]\). Assume that there exists
		\(q_0\in\operatorname{PSH}(X,\theta)\) such that
		\[
		q_0\le h\quad\text{q.e. on }X,
		\qquad
		\int_X\theta_{q_0}^n>0 .
		\]
		Then
		\[
		P_\theta(h)
		=
		\left(
		\sup\{q\in\operatorname{PSH}(X,\theta):
		q\le h\ \text{q.e. on }X\}
		\right)^* .
		\]
		
	\end{lemma}
	
	\begin{proof}
		Denote by
		\[
		\widetilde{P}_{\theta}(h):=\left(\sup\{q\in\operatorname{PSH}(X,\theta): q\le h\ \text{q.e. on }X\}\right)^*.
		\]
		It is clear that \(P_\theta(h)\le \widetilde{P}_{\theta}(h)\), so it suffices to prove the reverse inequality.
		
		Let \(w\in\operatorname{PSH}(X,\theta)\) satisfy \(w\le h\) outside a pluripolar set \(E\). Enlarging \(E\) if necessary, we may assume that \(q_0\le h\) on \(X\setminus E\). After subtracting a constant from \(q_0\), we can also suppose \(q_0\le V_\theta\). By \cref{lem:strict-subbarrier-slope-gap}, there exist \(\chi\in\operatorname{PSH}(X,\theta)\) and \(a>0\) such that \(\theta_\chi\ge a\omega\) and \(\chi\le q_0\). Since \(E\) is pluripolar, there exists a quasi-plurisubharmonic function  \(\rho\leq 0\) on $X$,  such that \(dd^c\rho\ge -a\omega\), and \(E\subset\{\rho=-\infty\}\) (see \cite[Theorem~7.2]{GZ05}). Define \(\psi:=\chi+\rho\); then \(\psi\in\operatorname{PSH}(X,\theta)\) and \(\psi=-\infty\) on \(E\). For \(0<t<1\), set \(w_t:=(1-t)w+t\psi\). Then \(w_t\in\operatorname{PSH}(X,\theta)\) and \(w_t\le h\) on all of \(X\), hence \(w_t\le P_\theta(h)\). Letting \(t\downarrow 0\) yields \(w\le P_\theta(h)\) quasi-everywhere, and therefore everywhere. Taking the supremum over all admissible \(w\) (which satisfy the inequality quasi-everywhere) proves the claim.
	\end{proof}

	\begin{theorem}[=\cref{intro_thm:relative-bounded-obstacle-capacity-continuity}]
		\label{thm:relative-bounded-obstacle-capacity-continuity}
		Let $(X,\omega)$ be a compact K\"ahler manifold, and let $\theta$ be a smooth
		closed real $(1,1)$-form whose cohomology class is big. Let
		$\phi_j,\phi\in\operatorname{PSH}(X,\theta)$ be normalized positive mass model
		potentials such that
		\begin{equation}
			\sup_X\phi_j=\sup_X\phi=0,
			\qquad
			\phi_j\to\phi
			\quad\text{in }\operatorname{Cap}_\omega .
		\end{equation}
		Let $h_j,h$ be quasi-continuous functions such that $h_j\to h$ in capacity,
		and suppose that there exists a constant $C>0$ such that, q.e. on $X$,
		\begin{equation}
			\phi_j-C\le h_j\le \phi_j,
			\qquad
			\phi-C\le h\le \phi .
		\end{equation}
		Then
		\begin{equation}
			P_{\theta}(h_j)\to P_\theta(h)
			\quad\text{in }\operatorname{Cap}_\omega .
		\end{equation}
	\end{theorem}
	\begin{proof}
		Put \(V:=V_\theta\), \(v_j:=P_\theta(h_j)\), and \(v:=P_\theta(h)\). Then
		\[
		\phi_j-C\le v_j\le \phi_j,\qquad
		\phi-C\le v\le \phi .
		\]
		We first prove the lower-tail estimate. Fix \(\delta>0\) and set
		\(A_j:=\{v_j<v-\delta\}\). By \Cref{lem:strict-subbarrier-slope-gap}, applied
		to \(v\) and \(\phi\), there exist \(\chi\in\operatorname{PSH}(X,\theta)\) and
		constants \(a,\sigma>0\) such that
		\[
		\theta_\chi\ge a\omega,\qquad
		\chi\le v-\sigma(V-v).
		\]
		For \(R>0\), let \(G_R:=\{\chi>v-R\}\). Since \(X\setminus G_R\) decreases to a
		pluripolar set, \(\operatorname{Cap}_\omega(X\setminus G_R)\to0\) as
		\(R\to+\infty\).
		
		We shall use the following uniform domination. Let \(C_1:=\max\{C,1\}\) and set
		\(w_j:=C_1^{-1}v_j+(1-C_1^{-1})\phi_j\). Then
		\(\phi_j-1\le w_j\le\phi_j\), and \(\theta_{v_j}\le C_1\theta_{w_j}\). By \cref{lem:uniform-DDL-capacity-moving-models}, after increasing \(j\) if
		necessary, there exists an increasing continuous function \(F\), independent of
		\(j\), with \(F(0)=0\), such that
		\[
		\int_E\theta_{v_j}^n
		\le C_1^n\operatorname{Cap}_{\theta,\phi_j}(E)
		\le C_1^nF(\operatorname{Cap}_\omega(E))
		\]
		for all Borel sets \(E\subset X\).
		
		By the envelope contact theorem \cref{thm:envelope-contact}, \(\theta_{v_j}^n\) is concentrated on
		\(\{v_j=h_j\}\). Since \(v\le h\) q.e., we get
		\[
		\begin{aligned}
			M_j(\delta)
			&:=
			\int_{\{v_j<v-\delta/2\}}\theta_{v_j}^n                                      \\
			&\le
			\int_{\{h_j<h-\delta/2\}}\theta_{v_j}^n                                      \\
			&\le
			C_1^nF\bigl(\operatorname{Cap}_\omega(\{h_j<h-\delta/2\})\bigr).
		\end{aligned}
		\]
		Hence \(M_j(\delta)\to0\).
		
		We now prove that \(\operatorname{Cap}_\omega(A_j)\to0\). Let
		\(\varepsilon>0\). Choose \(R>0\) such that
		\[
		\operatorname{Cap}_\omega(X\setminus G_R)<\varepsilon/3 .
		\]
		Fix \(0<b<a\) and put
		\[
		\tau:=\min\left\{\frac12,\frac{\delta}{8(R+b)}\right\}.
		\]
		For \(T>1\), define
		\[
		L_T:=T-C-\frac{T-\delta/2}{1+\tau\sigma}.
		\]
		Choose \(T\) large enough so that \(L_T>0\) and
		\[
		\operatorname{Cap}_\omega(\{\phi\le V-T/2+C\})<\varepsilon/6 .
		\]
		Since \(\phi_j\to\phi\) in capacity, after increasing \(j\) if necessary,
		\[
		\operatorname{Cap}_\omega(\{\phi_j<\phi-T/2\})<\varepsilon/6 .
		\]
		For such \(j\),
		\[
		\{v_j\le V-T\}
		\subset
		\{\phi_j\le V-T+C\}
		\subset
		\{\phi\le V-T/2+C\}\cup\{\phi_j<\phi-T/2\},
		\]
		and hence
		\[
		\operatorname{Cap}_\omega(\{v_j\le V-T\})<\varepsilon/3 .
		\]
		
		It remains to estimate the core set
		\[
		B_{j,R,T}:=A_j\cap G_R\cap\{v_j>V-T\}.
		\]
		Fix
		\(\rho\in\operatorname{PSH}(X,\omega)\), \(0\le\rho\le1\). Write
		\(\omega_\rho:=\omega+dd^c\rho\), and set
		\[
		\psi_\rho:=(1-\tau)v+\tau\chi+\tau b\rho-\frac{\delta}{2}-\tau b .
		\]
		Then \(\psi_\rho\le v-\delta/2\), \(\theta_{\psi_\rho}\ge\tau b\,\omega_\rho\),
		and on \(G_R\),
		\[
		\psi_\rho\ge v-\frac{\delta}{2}-\tau(R+b)\ge v-\frac{5\delta}{8}.
		\]
		Thus \(B_{j,R,T}\subset\{U_j^T<\psi_\rho\}\), where \(U_j^T:=\max(v_j,V-T)\). Since
		\(V-T\le U_j^T\le V\), both \(U_j^T\) and
		\(H_{j,\rho}^T:=\max(U_j^T,\psi_\rho)\) have minimal singularity type. By the
		comparison principle \cref{thm:DDL-relative-comparison} in \(\mathcal E(X,\theta,V)\) and plurifine locality,
		\[
		(\tau b)^n\int_{B_{j,R,T}}\omega_\rho^n
		\le
		\int_{\{U_j^T<\psi_\rho\}}\theta_{U_j^T}^n .
		\]
		
		We split the last integral. On \(\{v_j>V-T\}\), plurifine locality gives
		\(\theta_{U_j^T}^n=\theta_{v_j}^n\), and the corresponding contribution is at
		most \(M_j(\delta)\). On the remaining part \(\{v_j\le V-T\}\), the inequality
		\(U_j^T<\psi_\rho\) implies \(V-T<\psi_\rho\). Using the slope gap,
		\[
		\psi_\rho\le v-\tau\sigma(V-v)-\frac{\delta}{2},
		\]
		we obtain
		\[
		v>V-\frac{T-\delta/2}{1+\tau\sigma}.
		\]
		Since \(v\le\phi\), while \(v_j\le V-T\) and \(v_j\ge\phi_j-C\) imply
		\(\phi_j\le V-T+C\), this part is contained in
		\[
		E_{j,T}:=\{\phi_j<\phi-L_T\}.
		\]
		Therefore, using \(V-T\le U_j^T\le V\) and the capacity comparison for
		\(\operatorname{Cap}_{\theta,V}\),
		\[
		\int_{\{U_j^T<\psi_\rho\}\cap\{v_j\le V-T\}}\theta_{U_j^T}^n
		\le
		T^nF_V(\operatorname{Cap}_\omega(E_{j,T})).
		\]
		Since \(L_T>0\) is fixed and \(\phi_j\to\phi\) in capacity,
		\(\operatorname{Cap}_\omega(E_{j,T})\to0\). Thus, after increasing \(j\) once
		more, we have both
		\[
		M_j(\delta)<\frac{(\tau b)^n\varepsilon}{6},
		\qquad
		T^nF_V(\operatorname{Cap}_\omega(E_{j,T}))
		<
		\frac{(\tau b)^n\varepsilon}{6}.
		\]
		It follows that $\int_{B_{j,R,T}}\omega_\rho^n<\varepsilon/3$.
		Taking the supremum over \(\rho\), we get
		\[
		\operatorname{Cap}_\omega(B_{j,R,T})\le\varepsilon/3 .
		\]
		
		Finally, it is clear that
		\[
		A_j\subset
		(X\setminus G_R)\cup\{v_j\le V-T\}\cup B_{j,R,T} .
		\]
		For all sufficiently large \(j\), each of the three terms has capacity \(<\varepsilon/3\).
		Hence
		\begin{equation}\label{eq:envelope lower tail estimate}
			\operatorname{Cap}_\omega(\{v_j<v-\delta\})\to0 .
		\end{equation}
		
		We now prove \(L^1\)-convergence. Namely, $v_j\to v$ in $L^1(X)$. Take an arbitrary subsequence. Since
		\(v_j\le0\) and \(v_j\ge\phi_j-C\), and since \(\phi_j\to\phi\) in \(L^1\),
		compactness of quasi-psh functions gives a further subsequence, still denoted by
		\(v_j\), such that \(v_j\to w\) in \(L^1(X)\). The lower-tail estimate \eqref{eq:envelope lower tail estimate} implies
		\(w\ge v\). On the other hand, \(v_j\le h_j\) q.e. and \(h_j\to h\) in capacity.
		By \Cref{lem:qpsh-upper-bound-passes-to-limit}, \(w\le h\) q.e. Hence by \cref{lem: envelope equality},
		\(w\le P_\theta(h)=v\). Thus \(w=v\). Since every subsequence has a further
		subsequence converging to \(v\) in \(L^1\), we have \(v_j\to v\) in \(L^1(X)\).
		
		By the Hartogs lemma,
		\[
		\operatorname{Cap}_\omega(\{v_j>v+\delta\})\to0
		\]
		for every \(\delta>0\). Together with the lower-tail estimate, this proves
		\(v_j\to v\) in \(\operatorname{Cap}_\omega\).
	\end{proof}
	
	As immediate corollaries, we have the envelope stability under the capacity topology when the obstacles $h_j,h$ are uniformly of minimal singularity types or uniformly bounded. 
	\begin{corollary}
		\label{cor:minimal-type-obstacle-capacity-continuity}
		Let $h_j,h$ be quasi-continuous functions satisfying, for some constant $C>0$,
		\begin{equation}
			\label{eq:minimal-type-obstacle-bound}
			V_\theta-C\le h_j,h\le V_\theta .
		\end{equation}
		If $h_j\to h$ in capacity, then
		\begin{equation}
			P_\theta(h_j)\to P_\theta(h)
			\quad\text{in }\operatorname{Cap}_\omega .
		\end{equation}
	\end{corollary}
	
	\begin{proof}
		Apply \Cref{thm:relative-bounded-obstacle-capacity-continuity} with
		$\phi_j=\phi=V_\theta$.
	\end{proof}

	\begin{corollary}
		\label{cor:bounded-obstacle-capacity-continuity}
		Let $h_j,h$ be uniformly bounded  quasi-continuous functions. If $h_j\to h$ in capacity, then
		\begin{equation}
			P_\theta(h_j)\to P_\theta(h)
			\quad\text{in }\operatorname{Cap}_\omega .
		\end{equation}
	\end{corollary}
	
	\begin{proof} Up to subtracting a uniform upper bound from $h_j$ and $h$, we may assume $h_j,h\leq 0$.
		Since $V_\theta\le0$, the functions
		$\tilde h_j:=\min\{h_j,V_\theta\}$ and $\tilde h:=\min\{h,V_\theta\}$ satisfy
		\begin{equation}
			V_\theta-C\le \tilde h_j,\tilde h\le V_\theta .
		\end{equation}
		Moreover, $P_\theta(h_j)=P_\theta(\tilde h_j)$ and
		$P_\theta(h)=P_\theta(\tilde h)$. Since $\tilde h_j\to\tilde h$ in capacity,
		the result follows from
		\Cref{cor:minimal-type-obstacle-capacity-continuity}.
	\end{proof}
	
	For rooftop envelopes, we have the following:
	\begin{corollary}
		\label{cor:rooftop-capacity-continuity-moving-singularities}
		Let $f_j,g_j,f,g$ be quasi-continuous functions. Suppose that there exist
		normalized positive mass model potentials $\phi_j,\phi$ and a constant $C>0$
		such that $\phi_j\to\phi$ in $\operatorname{Cap}_\omega$ and, q.e. on $X$,
		\begin{equation}
			\label{eq:rooftop-relative-bounds}
			\phi_j-C\le f_j,g_j\le \phi_j,
			\qquad
			\phi-C\le f,g\le \phi .
		\end{equation}
		If $   f_j\to f$ and $g_j\to g$ in capacity, then
		\begin{equation}
			P_\theta(\min\{f_j,g_j\})
			\to
			P_\theta(\min\{f,g\})
			\quad\text{in }\operatorname{Cap}_\omega .
		\end{equation}
	\end{corollary}
	
	\begin{proof}
		Set $h_j:=\min\{f_j,g_j\}$ and $h:=\min\{f,g\}$. Then
		$\phi_j-C\le h_j\le\phi_j$, $\phi-C\le h\le\phi$, and $h_j\to h$ in capacity.
		The result follows from
		\Cref{thm:relative-bounded-obstacle-capacity-continuity}.
	\end{proof}

	\subsection{Optimality of the capacity topology}
	The capacity stability theorem, \Cref{thm:moving-background-capacity-stability},
	shows that capacity convergence of the prescribed model potentials is
	sufficient for capacity convergence of the corresponding Monge-Amp\`ere
	solutions. We now prove that, under the natural
	no-mass-loss assumption given by total variation convergence of the
	Monge-Amp\`ere measures, this condition is also necessary. In other words, the
	capacity topology is not merely a convenient replacement for the
	\(d_{\mathcal S}\)-topology: it is exactly the topology on moving prescribed singularities
	detected by stability of the equation.
	
	The key point is to recover the model envelope of the singularity type from
	the solution. We first prove a fixed-level rooftop continuity statement with
	moving background classes, and then pass from fixed rooftops to singularity
	envelopes.

	\begin{lemma}
		\label{lem:fixed-level-rooftop-moving-background}
		Let $\theta_j,\theta$ be smooth closed real $(1,1)$-forms with big classes and
		\[
		-\varepsilon_j\omega\le\theta_j-\theta\le\varepsilon_j\omega,
		\qquad
		\varepsilon_j\to0 .
		\]
		Set $V_j:=V_{\theta_j}$ and $V:=V_\theta$. Let
		$u_j\in\operatorname{PSH}(X,\theta_j)$ and
		$u\in\operatorname{PSH}(X,\theta)$ satisfy
		\[
		\sup_j\sup_Xu_j<+\infty,\qquad \sup_Xu<+\infty,
		\]
		and assume
		\[
		u_j\to u\quad\text{in }\operatorname{Cap}_\omega,\qquad
		\theta_{j,u_j}^n\to\theta_u^n\quad\text{in total variation}.
		\]
		Assume moreover that $\int_X\theta_u^n>0$. Then
		\[
		P_{\theta_j}(u_j,V_j)\to P_\theta(u,V)
		\quad\text{in }\operatorname{Cap}_\omega .
		\]
	\end{lemma}
	
	\begin{proof}
		Put
		\[
		p_j:=P_{\theta_j}(u_j,V_j),\qquad p:=P_\theta(u,V).
		\]
		After increasing a constant $A>0$, we may assume for all $j$,
		\[
		u_j-A\le p_j\le u_j,\qquad u-A\le p\le u.
		\]
		In particular $\int_X\theta_p^n>0$. By \Cref{lem:strict-subbarrier-slope-gap}
		applied to $p$, choose $\chi\in\operatorname{PSH}(X,\theta)$ and
		$a,\sigma>0$ such that
		\[
		\theta_\chi\ge a\omega,\qquad
		\chi\le p-\sigma(V-p).
		\]
		We first prove the lower tail. Fix $\delta>0$ and set
		$A_j:=\{p_j<p-\delta\}$. Define
		\[
		M_j(\delta):=
		\int_{\{p_j<p-\delta/2\}}\theta_{j,p_j}^n .
		\]
		By the rooftop contact inequality \cref{thm:rooftop-contact-inequality} for $p_j=P_{\theta_j}(u_j,V_j)$,
		\[
		\theta_{j,p_j}^n
		\le
		\mathbf 1_{\{p_j=u_j\}}\theta_{j,u_j}^n
		+
		\mathbf 1_{\{p_j=V_j\}}\theta_{j,V_j}^n .
		\]
		On $\{p_j<p-\delta/2\}$, the first term is supported in
		$\{u_j<u-\delta/2\}$, since $p\le u$. The second term is supported in
		$\{|V_j-V|>\delta/4\}$ for all large $j$: indeed, if
		$p_j=V_j$ and $|V_j-V|\le\delta/4$, then
		$V_j\ge V-\delta/4\ge p-\delta/4$, contradicting
		$p_j<p-\delta/2$. Therefore
		\[
		M_j(\delta)
		\le
		\mu_j(\{u_j<u-\delta/2\})
		+
		\theta_{j,V_j}^n(\{|V_j-V|>\delta/4\}),
		\]
		where $\mu_j:=\theta_{j,u_j}^n$. The first term tends to zero by the same arguments as in \cref{thm:moving-background-capacity-stability}, thanks to \cref{lem:DoVu-nonpluripolar-measure-convergence}. The second term tends to zero by $V_j\to V$ in capacity and \cref{lem:uniform-DDL-capacity-moving-models}. Hence
		$M_j(\delta)\to0$. Provided this, the subsequent arguments are almost the same as that in \cref{thm:moving-background-capacity-stability} and \cref{thm:relative-bounded-obstacle-capacity-continuity}.
		
		Now take $R>0$ and put $G_R:=\{\chi>p-R\}$; then
		$\operatorname{Cap}_\omega(X\setminus G_R)\to0$ as $R\to+\infty$. Fix
		$0<b<a$ and
		\[
		\tau:=\min\left\{\frac12,\frac{\delta}{8(R+b)}\right\}.
		\]
		For $\rho\in\operatorname{PSH}(X,\omega)$, $0\le\rho\le1$, set
		\[
		\psi_\rho:=(1-\tau)p+\tau\chi+\tau b\rho-\delta/2-\tau b .
		\]
		For $j$ large,
		\[
		\theta_{j,\psi_\rho}\ge \tau b(\omega+dd^c\rho),
		\qquad
		\psi_\rho\le p-\delta/2,
		\]
		and on $G_R$,
		\[
		\psi_\rho\ge p-\delta/2-\tau(R+b)\ge p-5\delta/8.
		\]
		Let $T>1$ and set $U_j^T:=\max(p_j,V_j-T)$. On the Borel set
		\[
		B_{j,R,T}:=A_j\cap G_R\cap\{p_j>V_j-T\}
		\]
		we have $U_j^T=p_j<\psi_\rho$. Applying the comparison principle \cref{thm:DDL-relative-comparison} in
		$\mathcal E(X,\theta_j,V_j)$ to $\max(U_j^T,\psi_\rho)$ and using plurifine
		locality gives
		\[
		(\tau b)^n\int_{B_{j,R,T}}(\omega+dd^c\rho)^n
		\le
		\int_{\{U_j^T<\psi_\rho\}}\theta_{j,U_j^T}^n .
		\]
		We split the right hand side. On $\{p_j>V_j-T\}$, it is bounded by
		$M_j(\delta)$. On
		$\{U_j^T<\psi_\rho\}\cap\{p_j\le V_j-T\}$, work on
		$Q_{j,\eta}:=\{|V_j-V|\le\eta\}$. There
		\[
		V-T-\eta<\psi_\rho\le p-\tau\sigma(V-p)-\delta/2,
		\]
		hence
		\[
		p>V-\frac{T+\eta-\delta/2}{1+\tau\sigma}.
		\]
		Since $p\le u$ and $p_j\ge u_j-A$, the same part is contained in
		\[
		\{u_j<u-L_{T,\eta}\}\cup (X\setminus Q_{j,\eta}),
		\]
		where
		\[
		L_{T,\eta}:=
		T-A-\eta-\frac{T+\eta-\delta/2}{1+\tau\sigma}
		\longrightarrow+\infty
		\quad (T\to+\infty).
		\]
		Using $V_j-T\le U_j^T\le V_j$ and \cref{lem:uniform-DDL-capacity-moving-models},
		\[
		\int_E\theta_{j,U_j^T}^n
		\le
		\Gamma_T(\operatorname{Cap}_\omega(E)),
		\]
		we see that, for fixed $T,\eta$ with $L_{T,\eta}>0$, this contribution
		tends to zero as $j\to\infty$, because $u_j\to u$ and $V_j\to V$ in capacity.
		
		Taking the supremum over $0\le\rho\le1$ gives, for fixed large $R,T$,
		\[
		\operatorname{Cap}_\omega(B_{j,R,T})\to0.
		\]
		Finally, it is elementary to see that
		\[
		\operatorname{Cap}_\omega(\{p_j<p-\delta\})\to0 .
		\]
		
		It remains to identify the upper tail. Take any subsequence and assume
		$p_j\to w$ in $L^1(X)$. The lower-tail estimate gives $w\ge p$ a.e. On the
		other hand $p_j\le m_j:=\min(u_j,V_j)$, and
		$m_j\to m:=\min(u,V)$ in capacity. \cref{lem:qpsh-upper-bound-passes-to-limit}
		gives $w\le m$ q.e., hence by \cref{lem: envelope equality}
		\[
		w\le P_\theta(m)=p.
		\]
		Thus $w=p$. Every $L^1$-cluster is $p$, so $p_j\to p$ in $L^1(X)$. The
		uniform Hartogs lemma for the uniformly quasi-psh family $p_j$ gives
		\[
		\operatorname{Cap}_\omega(\{p_j>p+\delta\})\to0 .
		\]
		Together with the lower-tail estimate, this proves $p_j\to p$ in capacity.
	\end{proof}
	
	For envelope of singularity types, we have the following:
	\begin{theorem}[=inverse implication part of  \cref{thm:intro-optimal-capacity-topology}]
		\label{thm:no-mass-loss-singularity-envelope-capacity}
		Let $\theta_j,\theta$ be smooth closed real $(1,1)$-forms with big classes and
		\[
		-\varepsilon_j\omega\le \theta_j-\theta\le \varepsilon_j\omega,
		\qquad \varepsilon_j\to0 .
		\]
		Set $V_j:=V_{\theta_j}$ and $V:=V_\theta$. Let
		$u_j\in\operatorname{PSH}(X,\theta_j)$ and
		$u\in\operatorname{PSH}(X,\theta)$ be normalized by
		$\sup_Xu_j=\sup_Xu=0$, and assume $u_j\to u$ in
		$\operatorname{Cap}_\omega$. Put
		\[
		\phi_j:=P_{\theta_j}[u_j],\qquad \phi:=P_\theta[u].
		\]
		Assume that $\phi_j,\phi$ have positive mass, that
		$u_j\in\mathcal E(X,\theta_j,\phi_j)$, $u\in\mathcal E(X,\theta,\phi)$, and that
		\[
		\theta_{j,u_j}^n\to\theta_u^n
		\quad\text{in total variation}.
		\]
		Then $\phi_j\to\phi$ in $\operatorname{Cap}_\omega$.
	\end{theorem}
	
	\begin{proof}
		We write $\mu_j:=\theta_{j,u_j}^n$ and $\mu:=\theta_u^n$.
		
		\smallskip
		\noindent\textbf{Step 1. Lower tail estimate of $\phi_j-\phi$.}
		For $C>0$ set
		$p_{j,C}:=P_{\theta_j}(u_j+C,V_j)$ and $p_C:=P_\theta(u+C,V)$. By
		\Cref{lem:fixed-level-rooftop-moving-background}, $p_{j,C}\to p_C$ in capacity
		for every fixed $C$.
		
		We first note that $p_{j,C}\le\phi_j$. Indeed, $u_j\preceq\phi_j$, hence
		$u_j\le\phi_j+A_j$ for some constant $A_j$. Thus
		\[
		p_{j,C}
		\le P_{\theta_j}(\phi_j+A_j+C,V_j)
		\le P_{\theta_j}[\phi_j]=\phi_j .
		\]
		Since $p_C\uparrow P_\theta[u]=\phi$ q.e., monotone convergence almost everywhere gives
		$p_C\to\phi$ in capacity as $C\to+\infty$ \cite[Proposition4.25]{GZ17}.
		
		Fix $\delta,\varepsilon>0$. Choose $C>0$ such that
		\[
		\operatorname{Cap}_\omega(\{p_C<\phi-\delta/2\})<\varepsilon/2 .
		\]
		For this fixed $C$, choose $J$ such that, for $j\ge J$,
		\[
		\operatorname{Cap}_\omega(\{p_{j,C}<p_C-\delta/2\})<\varepsilon/2 .
		\]
		Since $p_{j,C}\le\phi_j$,
		\[
		\{\phi_j<\phi-\delta\}
		\subset
		\{p_C<\phi-\delta/2\}
		\cup
		\{p_{j,C}<p_C-\delta/2\}.
		\]
		Hence
		\[
		\operatorname{Cap}_\omega(\{\phi_j<\phi-\delta\})\to0 .
		\]
		
		\smallskip
		\noindent\textbf{Step 2. Uniform relative tails from TV convergence.}
		For every $\eta>0$ there exist $L_0>0$ and $J_0>0$ such that, for all
		$L\ge L_0$ and $j\ge J_0$,
		\[
		\mu_j(\{u_j<\phi_j-L\})<\eta .
		\]
		Indeed, $\phi_j\le V_j\le0$, hence
		\[
		\{u_j<\phi_j-L\}\subset\{u_j<-L\}.
		\]
		Choose $L_0$ so large that $\mu(\{u<-L_0+1\})<\eta/3$. Since
		$u_j\to u$ in capacity and $\mu$ does not charge pluripolar sets, after
		increasing $J_0$ we have, for $j\ge J_0$,
		\[
		\mu(\{|u_j-u|>1\})<\eta/3,
		\qquad
		\|\mu_j-\mu\|_{\rm TV}<\eta/3 .
		\]
		Then, for $L\ge L_0$ and $j\ge J_0$,
		\[
		\begin{aligned}
			\mu_j(\{u_j<\phi_j-L\})
			&\le \mu_j(\{u_j<-L\})                                      \\
			&\le \mu(\{u_j<-L\})+\|\mu_j-\mu\|_{\rm TV}                  \\
			&\le \mu(\{u<-L+1\})
			+\mu(\{|u_j-u|>1\})
			+\|\mu_j-\mu\|_{\rm TV}
			<\eta .
		\end{aligned}
		\]
		
		\smallskip
		\noindent\textbf{Step 3. A volume recovery estimate.}
		We prove that, for every $\delta,\varepsilon>0$, there exist $C_0>0$ and
		$J_0>0$ such that, for all $C\ge C_0$ and $j\ge J_0$,
		\[
		\omega^n(\{\phi_j>p_{j,C}+\delta\})<\varepsilon .
		\]
		
		Put $m_j:=\int_X\theta_{j,\phi_j}^n=\mu_j(X)$ and
		$m:=\int_X\theta_\phi^n=\mu(X)>0$. Since $m_j\to m$ by the assumption $\mu_j\to\mu$ in total variation, after discarding finitely
		many indices we may assume $m_j\ge m/2$. Let
		$c_j:=m_j/\int_X\omega^n$ and $c_0:=m/(2\int_X\omega^n)$.
		
		By \cref{thm:relative-solvability}, we can find $\chi_j\in\mathcal E(X,\theta_j,\phi_j)$ solve
		\[
		\theta_{j,\chi_j}^n=c_j\omega^n,
		\qquad
		\sup_X(\chi_j-\phi_j)=0 .
		\]
		By the relative $L^\infty$ estimate \cref{cor:relative-Linfty-estimate-moving}, there exists
		$B>0$, independent of $j$, such that
		\[
		\phi_j-B\le\chi_j\le\phi_j .
		\]
		
		We next claim that $p_{j,C}\in\mathcal E(X,\theta_j,\phi_j)$ and
		$p_{j,C}\le\phi_j$. 
		
		Since $\sup_Xu_j=0$, we have $u_j\le V_j$, hence
		$u_j\le p_{j,C}$. As above, $p_{j,C}\le\phi_j$. By monotonicity of total
		non-pluripolar masses \cref{thm:mixed-monotonicity},
		\[
		\int_X\theta_{j,u_j}^n
		\le
		\int_X\theta_{j,p_{j,C}}^n
		\le
		\int_X\theta_{j,\phi_j}^n .
		\]
		The first and last terms are equal because
		$u_j\in\mathcal E(X,\theta_j,\phi_j)$. Hence
		$p_{j,C}\in\mathcal E(X,\theta_j,\phi_j)$.
		
		Fix $\delta,\varepsilon>0$ and put $B_1:=\max\{B,1\}$. Choose
		\[
		\tau:=\min\left\{\frac12,\frac{\delta}{4B_1}\right\},
		\qquad
		\psi_j:=(1-\tau)\phi_j+\tau\chi_j-\delta/2 .
		\]
		Then
		\[
		\phi_j-\frac{3\delta}{4}\le\psi_j\le\phi_j-\frac{\delta}{2},
		\qquad
		\psi_j\in\mathcal E(X,\theta_j,\phi_j).
		\]
		Moreover
		\[
		\theta_{j,\psi_j}^n
		\ge
		\tau^n\theta_{j,\chi_j}^n
		=
		\tau^n c_j\omega^n
		\ge
		\tau^n c_0\omega^n .
		\]
		Let $E_{j,C}:=\{\phi_j>p_{j,C}+\delta\}$. Since
		$\psi_j\ge\phi_j-3\delta/4$, we have
		$E_{j,C}\subset\{p_{j,C}<\psi_j\}$. Since
		$p_{j,C},\psi_j\in\mathcal E(X,\theta_j,\phi_j)$, the relative comparison principle \cref{thm:DDL-relative-comparison}
		gives
		\[
		\tau^n c_0\,\omega^n(E_{j,C})\le  \int_{\{p_{j,C}<\psi_j\}}\theta_{j,\psi_j}^n
		\le
		\int_{\{p_{j,C}<\psi_j\}}\theta_{j,p_{j,C}}^n .
		\]
		By the rooftop contact inequality \cref{thm:rooftop-contact-inequality} for
		$p_{j,C}=P_{\theta_j}(u_j+C,V_j)$,
		\[
		\theta_{j,p_{j,C}}^n
		\le
		\mathbf 1_{\{p_{j,C}=u_j+C\}}\theta_{j,u_j}^n
		+
		\mathbf 1_{\{p_{j,C}=V_j\}}\theta_{j,V_j}^n .
		\]
		The second term does not contribute on $\{p_{j,C}<\psi_j\}$, because
		$\psi_j\le\phi_j-\delta/2\le V_j-\delta/2$. On the contact set
		$\{p_{j,C}=u_j+C\}$, the inequality
		$p_{j,C}<\psi_j\le\phi_j-\delta/2$ gives
		\[
		u_j<\phi_j-C-\delta/2 .
		\]
		Thus
		\[
		\tau^n c_0\,\omega^n(E_{j,C})
		\le
		\mu_j(\{u_j<\phi_j-C-\delta/2\}) .
		\]
		
		Apply Step 2 with $\eta:=\tau^n c_0\,\varepsilon$. There exist $C_0,J_0$ such
		that, for all $C\ge C_0$ and $j\ge J_0$,
		\[
		\omega^n(\{\phi_j>p_{j,C}+\delta\})<\varepsilon .
		\]
		
		\smallskip
		\noindent\textbf{Step 4. $L^1$ convergence of the singularity envelopes.}
		We prove that $\phi_j\to\phi$ in $L^1(X)$. Take an arbitrary subsequence. By
		compactness of normalized quasi-psh functions, after passing to a further
		subsequence we may assume $\phi_j\to\psi$ in $L^1(X)$, where
		$\psi\in\operatorname{PSH}(X,\theta)$ and $\sup_X\psi=0$.
		
		Step 1 implies $\psi\ge\phi$ a.e. Indeed, if
		$\omega^n(\{\psi<\phi-\delta\})>0$ for some $\delta>0$, then, since
		$\phi_j\to\psi$ in $\omega^n$-measure,
		$\omega^n(\{\phi_j<\phi-\delta/2\})$ is uniformly positive along a subsequence,
		contradicting the lower capacity-tail estimate from Step 1.
		
		It remains to prove $\psi\le\phi$ a.e. For arbitrary fixed  $\delta,\varepsilon>0$. By Step 3,
		choose $C>0$ and $J_1>0$ such that, for $j\ge J_1$,
		\[
		\omega^n(\{\phi_j>p_{j,C}+\delta\})<\varepsilon .
		\]
		For this fixed $C$, \Cref{lem:fixed-level-rooftop-moving-background} gives
		$p_{j,C}\to p_C$ in capacity, hence in $\omega^n$-measure. Since
		$\phi_j\to\psi$ in $L^1$, after increasing $J_1$ we have, for $j\ge J_1$,
		\[
		\omega^n(\{|\phi_j-\psi|>\delta\})<\varepsilon,
		\qquad
		\omega^n(\{|p_{j,C}-p_C|>\delta\})<\varepsilon .
		\]
		Outside the set $\{|\phi_j-\psi|>\delta\}\cup \{|p_{j,C}-p_C|>\delta\}$, $\psi>p_C+3\delta$ implies
		$\phi_j>p_{j,C}+\delta$. Hence
		\[
		\omega^n(\{\psi>p_C+3\delta\})\le3\varepsilon .
		\]
		Since $p_C\le\phi$, we get
		\[
		\omega^n(\{\psi>\phi+3\delta\})\le3\varepsilon .
		\]
		Letting $\varepsilon\to0$ and then $\delta\to0$ gives $\psi\le\phi$ a.e. Thus
		$\psi=\phi$ a.e. Since every subsequence has a further subsequence converging
		to $\phi$ in $L^1(X)$, the whole sequence satisfies
		\[
		\phi_j\to\phi
		\quad\text{in }L^1(X).
		\]
		
		\smallskip
		\noindent\textbf{Step 5. Upper tail and conclusion.}
		By Step 4 and the uniform Hartogs lemma for quasi-psh functions, for every
		$\delta>0$,
		\[
		\operatorname{Cap}_\omega(\{\phi_j>\phi+\delta\})\to0 .
		\]
		Together with the lower-tail estimate from Step 1, this gives
		\[
		\operatorname{Cap}_\omega(\{|\phi_j-\phi|>\delta\})\to0
		\]
		for every $\delta>0$. Hence $\phi_j\to\phi$ in $\operatorname{Cap}_\omega$.
	\end{proof}
	
	We are now able to show that capacity topology is the optimal topology under which one can expect the corresponding prescribed solutions converges in capacity:
	\begin{corollary}[=\cref{thm:intro-optimal-capacity-topology}: Capacity is the optimal topology for TV-stable equations]
		\label{cor:capacity-optimal-topology}
		Let $\theta_j,\theta$ be smooth closed real $(1,1)$-forms with big classes and
		\[
		-\varepsilon_j\omega\le\theta_j-\theta\le\varepsilon_j\omega,
		\qquad \varepsilon_j\to0 .
		\]
		Let $\phi_j\in\operatorname{PSH}(X,\theta_j)$ and
		$\phi\in\operatorname{PSH}(X,\theta)$ be normalized positive mass model
		potentials. Let $\mu_j,\mu$ be non-pluripolar positive Radon measures such that
		\[
		\|\mu_j-\mu\|_{\rm TV}\to0,\qquad
		\mu_j(X)=\int_X\theta_{j,\phi_j}^n,\qquad
		\mu(X)=\int_X\theta_\phi^n>0 .
		\]
		Let $u_j\in\mathcal E(X,\theta_j,\phi_j)$ and
		$u\in\mathcal E(X,\theta,\phi)$ be normalized solutions of
		\[
		\theta_{j,u_j}^n=\mu_j,\quad \sup_Xu_j=0,\qquad
		\theta_u^n=\mu,\quad \sup_Xu=0 .
		\]
		Then
		\[
		\phi_j\to\phi \text{ in }\operatorname{Cap}_\omega
		\quad\Longleftrightarrow\quad
		u_j\to u \text{ in }\operatorname{Cap}_\omega .
		\]
	\end{corollary}
	
	\begin{proof}
		The implication from left to right is
		\Cref{thm:moving-background-capacity-stability}. Conversely, assume $u_j\to u$ in capacity. We first identify the absolute model
		projections. Since $u_j\in\mathcal E(X,\theta_j,\phi_j)$, \cite[Theorem 3.7]{DDL25} gives
		\[
		P_{\theta_j}[u_j](\phi_j)=\phi_j .
		\]
		As $\phi_j\le V_j$, we get
		\[
		P_{\theta_j}[u_j]\ge P_{\theta_j}[u_j](\phi_j)=\phi_j .
		\]
		On the other hand, $u_j\preceq\phi_j$, so for some constant $A_j$,
		$u_j\le\phi_j+A_j$. Hence, for every $C>0$,
		\[
		P_{\theta_j}(u_j+C,V_j)
		\le
		P_{\theta_j}(\phi_j+A_j+C,V_j)
		\le
		P_{\theta_j}[\phi_j]
		=
		\phi_j .
		\]
		Letting $C\to+\infty$ gives $P_{\theta_j}[u_j]\le\phi_j$. Therefore
		$P_{\theta_j}[u_j]=\phi_j$. The same argument gives $P_\theta[u]=\phi$.
		Combining the assumption
		$\theta_{j,u_j}^n=\mu_j\to\mu=\theta_u^n$ in total variation and 
		\Cref{thm:no-mass-loss-singularity-envelope-capacity} gives
		$\phi_j\to\phi$ in $\operatorname{Cap}_\omega$.
	\end{proof}
	
	\section{Equality of the ceiling and singularity envelopes}\label{sec-5}
	
	In this section we use our techniques to prove a conjecture of Darvas-Di Nezza-Lu \cite[Conjecture~2.5]{DDL21a}. We first recall the definition of the Ceiling operator.
	
	Let
	\((X,\omega)\) be compact K\"ahler, let \(\theta\) be big, and put
	\(V:=V_\theta\). For \(u\in\operatorname{PSH}(X,\theta)\), define
	\[
	\mathscr F_u
	:=
	\left\{
	v\in\operatorname{PSH}(X,\theta):
	u\preceq v,\ v\le0,\ 
	\int_X\theta_v^k\wedge\theta_V^{n-k}
	=
	\int_X\theta_u^k\wedge\theta_V^{n-k},
	\ 0\le k\le n
	\right\}.
	\]
	The ceiling of \(u\) is
	\[
	\mathscr C_\theta(u)
	:=
	\left(\sup_{v\in\mathscr F_u}v\right)^* .
	\]
	By \cite[Lemma~2.4]{DDL21a}, if \(u\le0\), then
	\[
	\mathscr C_\theta(u)
	=
	\lim_{\varepsilon\downarrow0}
	P_\theta\big[(1-\varepsilon)u+\varepsilon V\big],
	\]
	where the limit is decreasing. Moreover, if
	\(\int_X\theta_u^n>0\), they showed that
	\[
	\mathscr C_\theta(u)=P_\theta[u].
	\]
	The purpose of this section is to prove that this identity remains true  and the singularity envelope operator is idempotent without any
	positive mass assumption. This extends the positive mass model envelope calculus of Darvas-Di Nezza-Lu \cite[Theorem 3.6]{DDL25} to the setting of possibly zero mass.
	
	\begin{theorem}[=\cref{thm:intro-ceiling-equals-singularity-envelope}]
		\label{thm:ceiling-equals-singularity-envelope}
		Let \((X,\omega)\) be a compact K\"ahler manifold and let \(\theta\) be a smooth closed real $(1,1)$-form whose cohomology class is big. Then for every
		$u\in\operatorname{PSH}(X,\theta)$,
		\[
		\mathscr{C}_\theta(u)=P_\theta[u].
		\]
		Moreover, we have  $P_\theta[P_\theta[u]]=P_\theta[u]$.
	\end{theorem}
	
	\begin{proof}
		Put \(V:=V_\theta\). After adding a constant to \(u\), we may assume
		\(u\le V\le0\). Set
		\[
		\phi:=P_\theta[u].
		\]
		By \cite[Proposition~2.6(iv)]{DDL21a}, $\mathscr C_\theta(\phi)=\mathscr C_\theta(u)$. Thus it is enough to prove that
		\[
		\mathscr C_\theta(\phi)=\phi .
		\]
		For \(0<\varepsilon<1\), set
		\[
		v_\varepsilon:=(1-\varepsilon)\phi+\varepsilon V,
		\qquad
		w_\varepsilon:=(1-\varepsilon)u+\varepsilon V,
		\qquad
		\phi_\varepsilon:=P_\theta[v_\varepsilon].
		\]
		By \cite[Remark 3.1]{DDL25}, the mixed masses of \(u\) and
		\(\phi=P_\theta[u]\) with \(V\) are equal, namely, $\int_X\theta_u^k\wedge\theta_V^{n-k}=\int_X\theta_\phi^k\wedge\theta_V^{n-k}$. Hence, by multilinearity of non-pluripolar product and \cite[Remark 3.1]{DDL25} again,
		\[
		m_\varepsilon:= \int_X\theta_{w_\varepsilon}^n
		=
		\int_X\theta_{v_\varepsilon}^n=
		\int_X\theta_{\phi_\varepsilon}^n\geq \varepsilon^n\int_X\theta_V^n
		>0 .
		\]
		Since \(v_\varepsilon\) has positive mass, \(\phi_\varepsilon\) is a positive
		mass model potential, namely $P_\theta[\phi_\varepsilon]=\phi_\varepsilon$. Moreover \(w_\varepsilon\le v_\varepsilon\). Thus
		\(w_\varepsilon\in\mathcal E(X,\theta,\phi_\varepsilon)\), and by the relative full mass characterization \cite[Theorem 3.7]{DDL25},
		\[
		P_\theta[w_\varepsilon]=P_\theta[v_\varepsilon]=\phi_\varepsilon .
		\]
		By \cite[Lemma~2.4, Proposition~2.6]{DDL21a},
		\[
		\phi_\varepsilon
		=
		P_\theta[(1-\varepsilon)\phi+\varepsilon V]
		\searrow
		\mathscr C_\theta(\phi)
		\qquad(\varepsilon\downarrow0).
		\]
		Now it suffices to prove that
		\[
		\phi_\varepsilon\to\phi
		\quad\text{in }L^1(X).
		\]
		
		\medskip
		\noindent\textbf{Step 1. Fixed-level rooftops.}
		For \(A>0\), set
		\[
		p_{\varepsilon,A}:=P_\theta(w_\varepsilon+A,V),
		\qquad
		p_A:=P_\theta(u+A,V).
		\]
		For every fixed \(A\), we have
		\[
		p_{\varepsilon,A}\downarrow p_A
		\qquad(\varepsilon\downarrow0).
		\]
		Indeed, \(w_\varepsilon\searrow u\), hence
		\[
		h_{\varepsilon,A}:=\min\{w_\varepsilon+A,V\}
		\searrow
		h_A:=\min\{u+A,V\}.
		\]
		Thus \(p_{\varepsilon,A}\) decreases to some \(q\ge p_A\). Since
		\(p_{\varepsilon,A}\le h_{\varepsilon,A}\), passing to the limit gives
		\(q\le h_A\). Hence
		\[
		q\le P_\theta(h_A)=p_A,
		\]
		and so \(q=p_A\). Monotone convergence of quasi-psh functions gives convergence
		in capacity.
		
		\medskip
		\noindent\textbf{Step 2. A normalized tail estimate.}
		We claim that for every \(\eta>0\), there exists \(A_0>0\) such that for all
		\(A\ge A_0\) and all \(0<\varepsilon<1\),
		\begin{equation}
			\label{eq:attenuation-normalized-tail}
			\frac{
				\theta_{w_\varepsilon}^n(\{w_\varepsilon<V-A\})
			}{
				\int_X\theta_{w_\varepsilon}^n
			}
			<\eta .
		\end{equation}
		Let
		\[
		\mu_k:=\theta_u^k\wedge\theta_V^{n-k},
		\qquad
		a_k:=\mu_k(X),
		\qquad 0\le k\le n .
		\]
		The measures \(\mu_k\) do not charge pluripolar sets, and, as \(A\to+\infty\),
		\[
		\{u<V-A\}\downarrow \{u=-\infty\}
		\]
		modulo the pluripolar set \(\{V=-\infty\}\). Hence
		\[
		\mu_k(\{u<V-A\})\to0
		\qquad(A\to+\infty).
		\]
		Choose \(A_0\) so large that, for every \(k\) with \(a_k>0\),
		\[
		\mu_k(\{u<V-A\})<\eta a_k,
		\qquad A\ge A_0 .
		\]
		Since
		\[
		w_\varepsilon-V=(1-\varepsilon)(u-V),
		\]
		we have
		\[
		\{w_\varepsilon<V-A\}
		=
		\left\{
		u<V-\frac{A}{1-\varepsilon}
		\right\}
		\subset
		\{u<V-A\}.
		\]
		By multilinearity of non-pluripolar products,
		\[
		\theta_{w_\varepsilon}^n
		=
		\sum_{k=0}^n
		\binom nk(1-\varepsilon)^k\varepsilon^{n-k}\mu_k .
		\]
		Therefore, for \(A\ge A_0\),
		\[
		\begin{aligned}
			\theta_{w_\varepsilon}^n(\{w_\varepsilon<V-A\})
			&\le
			\sum_{k=0}^n
			\binom nk(1-\varepsilon)^k\varepsilon^{n-k}
			\mu_k(\{u<V-A\})                                      \\
			&\le
			\eta
			\sum_{k=0}^n
			\binom nk(1-\varepsilon)^k\varepsilon^{n-k}a_k         \\
			&=
			\eta\int_X\theta_{w_\varepsilon}^n .
		\end{aligned}
		\]
		This proves \eqref{eq:attenuation-normalized-tail}.
		
		\medskip
		\noindent\textbf{Step 3. Volume recovery.}
		We prove that, for every \(\delta,\eta>0\), there exists \(A_0>0\) such that
		for all \(A\ge A_0\) and all \(0<\varepsilon<1\),
		\begin{equation}
			\label{eq:volume-recovery-ceiling}
			\omega^n(\{\phi_\varepsilon>p_{\varepsilon,A}+\delta\})<\eta .
		\end{equation}
		Let
		\[
		E_{\varepsilon,A,\delta}:=
		\{\phi_\varepsilon>p_{\varepsilon,A}+\delta\}.
		\]
		Set
		\[
		\Omega:=\int_X\omega^n,
		\qquad
		c_\varepsilon:=\frac{m_\varepsilon}{\Omega}.
		\]
		By \cref{thm:relative-solvability}, choose
		\(\chi_\varepsilon\in\mathcal E(X,\theta,\phi_\varepsilon)\) such that
		\[
		\theta_{\chi_\varepsilon}^n=c_\varepsilon\omega^n,
		\qquad
		\sup_X(\chi_\varepsilon-\phi_\varepsilon)=0 .
		\]
		Since \(P_\theta[\phi_\varepsilon]=\phi_\varepsilon\), \cite[Lemma~2.7]{DDL21a}
		gives $\sup_X\chi_\varepsilon=0$.
		Thus standard compactness for normalized quasi-psh functions gives a constant
		\(C_1>0\), independent of \(\varepsilon\), such that
		\[
		\int_X(-\chi_\varepsilon)\omega^n\le C_1 .
		\]
		
		Fix \(B>1\), to be chosen later, and set
		\[
		G_{\varepsilon,B}:=
		\{\chi_\varepsilon>\phi_\varepsilon-B\}.
		\]
		Since \(\phi_\varepsilon\le0\),
		\[
		X\setminus G_{\varepsilon,B}
		\subset
		\{\chi_\varepsilon\le -B\}.
		\]
		Thus
		\begin{equation}
			\label{eq:bad-good-set-volume}
			\omega^n(X\setminus G_{\varepsilon,B})
			\le
			\omega^n(\{\chi_\varepsilon\le -B\})
			\le
			\frac{C_1}{B}.
		\end{equation}
		Choose
		\[
		\tau:=\min\left\{\frac12,\frac{\delta}{4B}\right\},
		\qquad
		\psi_\varepsilon:=
		(1-\tau)\phi_\varepsilon+\tau\chi_\varepsilon-\frac{\delta}{2}.
		\]
		Since \(\chi_\varepsilon\le\phi_\varepsilon\), we have
		\[
		\psi_\varepsilon
		\le
		\phi_\varepsilon-\frac{\delta}{2}
		\le
		V-\frac{\delta}{2}.
		\]
		On \(G_{\varepsilon,B}\),
		\[
		\psi_\varepsilon
		>
		\phi_\varepsilon-\tau B-\frac{\delta}{2}
		\ge
		\phi_\varepsilon-\frac{3\delta}{4}.
		\]
		Therefore
		\[
		E_{\varepsilon,A,\delta}\cap G_{\varepsilon,B}
		\subset
		\{p_{\varepsilon,A}<\psi_\varepsilon\}.
		\]
		
		We check the relative full mass condition. Since $w_\varepsilon\le p_{\varepsilon,A}\le P_\theta[w_\varepsilon]
		=
		\phi_\varepsilon$
		and $\int_X\theta_{w_\varepsilon}^n
		=
		\int_X\theta_{\phi_\varepsilon}^n$,
		monotonicity of total non-pluripolar masses \cref{thm:mixed-monotonicity} gives $p_{\varepsilon,A}\in\mathcal E(X,\theta,\phi_\varepsilon)$.
		Also, $\chi_\varepsilon-\frac{\delta}{2}
		\le
		\psi_\varepsilon
		\le
		\phi_\varepsilon-\frac{\delta}{2}$, hence $\psi_\varepsilon\in\mathcal E(X,\theta,\phi_\varepsilon)$.
		Now \cref{thm:DDL-relative-comparison} gives
		\[
		\begin{aligned}
			\tau^n c_\varepsilon
			\omega^n(E_{\varepsilon,A,\delta}\cap G_{\varepsilon,B})
			&\le
			\int_{\{p_{\varepsilon,A}<\psi_\varepsilon\}}
			\theta_{\psi_\varepsilon}^n                                      \\
			&\le
			\int_{\{p_{\varepsilon,A}<\psi_\varepsilon\}}
			\theta_{p_{\varepsilon,A}}^n .
		\end{aligned}
		\]
		Applying \cref{thm:rooftop-contact-inequality} to
		\[
		p_{\varepsilon,A}=P_\theta(w_\varepsilon+A,V),
		\]
		we get
		\[
		\theta_{p_{\varepsilon,A}}^n
		\le
		\mathbf 1_{\{p_{\varepsilon,A}=w_\varepsilon+A\}}\theta_{w_\varepsilon}^n
		+
		\mathbf 1_{\{p_{\varepsilon,A}=V\}}\theta_V^n .
		\]
		The second term gives no contribution on
		\(\{p_{\varepsilon,A}<\psi_\varepsilon\}\), since otherwise
		\[
		V=p_{\varepsilon,A}<\psi_\varepsilon\le V-\frac{\delta}{2},
		\]
		a contradiction. On the first contact set,
		\[
		p_{\varepsilon,A}=w_\varepsilon+A,
		\qquad
		p_{\varepsilon,A}<\psi_\varepsilon\le V-\frac{\delta}{2}
		\]
		imply
		\[
		w_\varepsilon<V-A-\frac{\delta}{2}.
		\]
		Consequently
		\[
		\tau^n c_\varepsilon
		\omega^n(E_{\varepsilon,A,\delta}\cap G_{\varepsilon,B})
		\le
		\theta_{w_\varepsilon}^n
		\left(\left\{w_\varepsilon<V-A-\frac{\delta}{2}\right\}\right).
		\]
		Since \(c_\varepsilon=m_\varepsilon/\Omega\), we obtain
		\[
		\omega^n(E_{\varepsilon,A,\delta})
		\le
		\frac{C_1}{B}
		+
		\frac{\Omega}{\tau^n}
		\frac{
			\theta_{w_\varepsilon}^n
			\left(\left\{w_\varepsilon<V-A-\frac{\delta}{2}\right\}\right)
		}{
			\int_X\theta_{w_\varepsilon}^n
		}.
		\]
		Now choose \(B\) so large that \(C_1/B<\eta/2\). Then \(\tau\) is fixed.
		By Step 2, choose \(A_0\) so large that, for all \(A\ge A_0\) and all
		\(0<\varepsilon<1\),
		\[
		\frac{
			\theta_{w_\varepsilon}^n
			\left(\left\{w_\varepsilon<V-A-\frac{\delta}{2}\right\}\right)
		}{
			\int_X\theta_{w_\varepsilon}^n
		}
		<
		\frac{\tau^n\eta}{2\Omega}.
		\]
		Then \eqref{eq:volume-recovery-ceiling} follows.
		
		\medskip
		\noindent\textbf{Step 4. \(L^1\)-convergence of \(\phi_\varepsilon\).}
		Take any sequence \(\varepsilon_j\downarrow0\). By compactness, after passing
		to a subsequence, we may assume
		\[
		\phi_{\varepsilon_j}\to\psi
		\quad\text{in }L^1(X).
		\]
		Since \(\phi\le\phi_{\varepsilon_j}\le V\), we have \(\psi\ge\phi\) a.e.
		
		We prove the reverse inequality. Fix \(\delta,\eta>0\). By Step 3, choose
		\(A>0\) such that
		\[
		\omega^n(\{\phi_{\varepsilon_j}>p_{\varepsilon_j,A}+\delta\})<\eta
		\]
		for every \(j\). For this fixed \(A\), Step 1 gives $p_{\varepsilon_j,A}\searrow p_A$, hence in \(\omega^n\)-measure. Since
		\(\phi_{\varepsilon_j}\to\psi\) in \(L^1\), we may take \(j\) large so that
		\[
		\omega^n(\{|\phi_{\varepsilon_j}-\psi|>\delta\})<\eta,
		\qquad
		\omega^n(\{|p_{\varepsilon_j,A}-p_A|>\delta\})<\eta .
		\]
		Hence
		\[
		\omega^n(\{\psi>p_A+3\delta\})\le\omega^n(\{\phi_{\varepsilon_j}>p_{\varepsilon_j,A}+\delta\}\cup\{|\phi_{\varepsilon_j}-\psi|>\delta\}\cup\{|p_{\varepsilon_j,A}-p_A|>\delta\})\le3\eta .
		\]
		Since
		\[
		p_A=P_\theta(u+A,V)\le P_\theta[u]=\phi,
		\]
		we get
		\[
		\omega^n(\{\psi>\phi+3\delta\})\le3\eta .
		\]
		Letting \(\eta\to0\) and then \(\delta\to0\), we obtain \(\psi\le\phi\) a.e.
		Thus \(\psi=\phi\). Every \(L^1\)-cluster point is \(\phi\), so
		\[
		\phi_\varepsilon\to\phi
		\quad\text{in }L^1(X).
		\]
		
		\medskip
		\noindent\textbf{Step 5. Conclusion.}
		Since
		\[
		\phi_\varepsilon\searrow\mathscr C_\theta(\phi)
		\qquad\text{and}\qquad
		\phi_\varepsilon\to\phi \text{ in }L^1(X),
		\]
		we have
		\[
		\mathscr C_\theta(\phi)=\phi,
		\]
		and the proof of $\mathscr C_\theta(\phi)=P_\theta[\phi]$ is  finished.
		
		We further apply \cite[Proposition~2.6(iv)]{DDL21a} to get 
		\[
		P_\theta[P_\theta[\phi]]
		=
		\mathscr C_\theta(P_\theta[\phi])
		=
		\mathscr C_\theta(\phi)
		=
		P_\theta[\phi].
		\]
		
		The proof of \cref{thm:ceiling-equals-singularity-envelope} is complete.
	\end{proof}
	
	We can also obtain the following characterization, generalizing \cite[Theorem 3.7]{DDL25} to the zero mass setting:
	\begin{proposition}
		\label{prop:full-mixed-mass-characterization}
		Put $V:=V_\theta$. Let \(u,\phi\in\operatorname{PSH}(X,\theta)\) and assume that
		\(u\preceq \phi\). Then the following are equivalent:
		\begin{enumerate}[(i)]
			\item $P_\theta[u]=P_\theta[\phi]$.
			
			\item For every \(0\le k\le n\),
			\[
			\int_X\theta_u^k\wedge\theta_V^{n-k}
			=
			\int_X\theta_\phi^k\wedge\theta_V^{n-k}.
			\]
		\end{enumerate}
	\end{proposition}
	
	\begin{proof}
		If \(P_\theta[u]=P_\theta[\phi]\), then the conclusion follows from
		\cite[Remark~3.1]{DDL25}, applied to mixed products.
		
		Conversely, assume the mixed masses are equal. Since \(u\preceq\phi\), by
		monotonicity of the ceiling operator,
		\[
		\mathscr C_\theta(u)\le \mathscr C_\theta(\phi).
		\]
		On the other hand, \(\mathscr C_\theta(\phi)\) is less singular than \(\phi\),
		hence less singular than \(u\), and it has the same mixed masses as \(\phi\),
		therefore also the same mixed masses as \(u\). Thus
		\(\mathscr C_\theta(\phi)\) is admissible in the definition of
		\(\mathscr C_\theta(u)\), and
		\[
		\mathscr C_\theta(\phi)\le \mathscr C_\theta(u).
		\]
		Hence \(\mathscr C_\theta(u)=\mathscr C_\theta(\phi)\). By
		\Cref{thm:ceiling-equals-singularity-envelope}, we get $P_\theta[u]=P_\theta[\phi]$.
	\end{proof}
	
	\section{Capacity stability of relative weak geodesic segments}
	
	In  this section, we prove   a capacity stability result for weak geodesic segments with moving
	prescribed singularities. The minimal singularity case is obtained by taking
	\(\phi_j=\phi=V_\theta\).
	
	Let \(u_0,u_1\in\operatorname{PSH}(X,\theta)\). A subgeodesic segment joining
	\(u_0,u_1\) is a family \(t\mapsto h_t\in\operatorname{PSH}(X,\theta)\),
	\(t\in(0,1)\), whose complexification is \(\pi_X^*\theta\)-psh on
	\(\{0<\operatorname{Re}\tau<1\}\times X\), and such that
	\[
	\limsup_{t\to  0}h_t\le u_0,\qquad
	\limsup_{t\to  1}h_t\le u_1,
	\]
	where $\pi_X$ is the projection $\pi_X:\{0<\operatorname{Re}\tau<1\}\times X\rightarrow X$. The weak geodesic segment joining \(u_0,u_1\) is the upper envelope of all such
	subgeodesic segments, we refer the reader to \cite{DDL18a,DDL18b} for more information about weak geodesic in big classes.
	
	\begin{lemma}
		\label{lem:relative-geodesic-segment-rooftop-legendre}
		Let \(\phi,u_0,u_1\in\operatorname{PSH}(X,\theta)\) and assume that, for some
		\(C>0\),
		\[
		\phi-C\le u_0,u_1\le\phi .
		\]
		Let \(t\mapsto u_t\) be the weak geodesic segment joining \(u_0,u_1\). Then
		\[
		u_t=
		\left(
		\sup_{\lambda\in\mathbb R}
		\bigl(P_\theta(u_0,u_1-\lambda)+t\lambda\bigr)
		\right)^*,
		\qquad 0\le t\le1.
		\]
		Moreover, the supremum may be taken over \([-C,C]\).
	\end{lemma}
	
	\begin{proof}
		Set
		\[
		w_t:=
		\left(
		\sup_{\lambda\in\mathbb R}
		\bigl(P_\theta(u_0,u_1-\lambda)+t\lambda\bigr)
		\right)^* .
		\]
		For fixed \(\lambda\), the curve
		\(t\mapsto P_\theta(u_0,u_1-\lambda)+t\lambda\) is a subgeodesic segment with
		boundary values bounded above by \(u_0\) and \(u_1\). Hence \(w_t\le u_t\).
		
		Conversely, let \(h_t\) be any subgeodesic segment with boundary values bounded
		above by \(u_0,u_1\). For \(\lambda\in\mathbb R\), put
		\[
		q_\lambda:=\inf_{0<s<1}(h_s-s\lambda). 
		\]
		By Kiselman's minimum principle \cite[Theorem 2.2]{Kis78}, \(q_\lambda\in\operatorname{PSH}(X,\theta)\)
		or \(q_\lambda\equiv-\infty\). The boundary conditions give
		\(q_\lambda\le u_0\) and \(q_\lambda\le u_1-\lambda\). Hence
		\[
		q_\lambda\le P_\theta(u_0,u_1-\lambda).
		\]
		For each fixed \(x\), the function \(t\mapsto h_t(x)\) is convex on \((0,1)\).
		
		We shall use the following elementary one-dimensional Fenchel-Moreau duality. Let
		\(f:(0,1)\to\mathbb R\) be convex and set
		\[
		f^*(\lambda):=\inf_{0<s<1}(f(s)-s\lambda).
		\]
		Then, for every \(t\in(0,1)\),
		\[
		f(t)=\sup_{\lambda\in\mathbb R}(f^*(\lambda)+t\lambda).
		\]
		Indeed, for every \(\lambda\) we have
		\(f^*(\lambda)\le f(t)-t\lambda\), hence
		\(\sup_\lambda(f^*(\lambda)+t\lambda)\le f(t)\). Conversely, since \(t\) is an
		interior point, the subdifferential \(\partial f(t)\) is nonempty. Namely, choose
		$\lambda\in[f^\prime_-(t),f^\prime_+(t)]$. Then
		\[
		f(s)\ge f(t)+\lambda(s-t),\qquad 0<s<1,
		\]
		and hence
		\[
		f(s)-s\lambda\ge f(t)-t\lambda .
		\]
		Taking the infimum over \(s\in(0,1)\), we get
		\(f^*(\lambda)\ge f(t)-t\lambda\), so
		\(f^*(\lambda)+t\lambda\ge f(t)\). This proves the equality.
		
		For each fixed $x$, the above duality gives
		\[
		h_t(x)\le\sup_{\lambda\in\mathbb R}\bigl(q_\lambda(x)+t\lambda\bigr).
		\]
		Consequently \(h_t\le w_t\). Taking the supremum over all admissible
		subgeodesic segments gives \(u_t\le w_t\).
		
		It remains to reduce the parameter interval. If \(\lambda\le -C\), then
		\(u_1-\lambda\ge \phi\ge u_0\), hence
		\(P_\theta(u_0,u_1-\lambda)=u_0\). Thus
		\[
		P_\theta(u_0,u_1-\lambda)+t\lambda
		\le u_0-tC,
		\]
		which is the value corresponding to \(\lambda=-C\). If \(\lambda\ge C\), then
		\(u_1-\lambda\le \phi-C\le u_0\), hence
		\(P_\theta(u_0,u_1-\lambda)=u_1-\lambda\). Thus
		\[
		P_\theta(u_0,u_1-\lambda)+t\lambda
		=
		u_1-(1-t)\lambda
		\le u_1-(1-t)C,
		\]
		which is the value corresponding to \(\lambda=C\). Therefore the supremum over
		\(\mathbb R\) equals the supremum over \([-C,C]\).
	\end{proof}
	The following theorem illustrates that the whole geodesic segment converges in capacity uniformly provided that their endpoints converge in capacity.
	\begin{theorem}[=\cref{intro_thm:weak geodesic segments convergence}]
		\label{thm:relative-weak-geodesic-segment-capacity-stability}
		Let \(\phi_j,\phi\in\operatorname{PSH}(X,\theta)\) be normalized positive mass
		model potentials such that
		\[
		\sup_X\phi_j=\sup_X\phi=0,
		\qquad
		\phi_j\to\phi
		\quad\text{in }\operatorname{Cap}_\omega .
		\]
		Let \(u_{j,0},u_{j,1},u_0,u_1\in\operatorname{PSH}(X,\theta)\). Assume that
		there exists \(C>0\) such that
		\[
		\phi_j-C\le u_{j,0},u_{j,1}\le\phi_j,
		\qquad
		\phi-C\le u_0,u_1\le\phi
		\]
		for all \(j\), and assume that
		\[
		u_{j,0}\to u_0,
		\qquad
		u_{j,1}\to u_1
		\quad\text{in }\operatorname{Cap}_\omega .
		\]
		Let \(t\mapsto u_{j,t}\) be the weak geodesic segment joining
		\(u_{j,0},u_{j,1}\), and let \(t\mapsto u_t\) be the weak geodesic segment
		joining \(u_0,u_1\). Then $u_{j,t}\to u_t$ in $\text{Cap}_\omega$ uniformly. Namely, for every \(\varepsilon>0\),
		\[
		\lim_{j\to\infty}
		\sup_{0\le t\le1}
		\operatorname{Cap}_\omega(\{|u_{j,t}-u_t|>\varepsilon\})=0 .
		\]
	\end{theorem}
	
	\begin{proof}
		For \(\lambda\in[-C,C]\), put
		\[
		p_{j,\lambda}:=P_\theta(u_{j,0},u_{j,1}-\lambda),
		\qquad
		p_\lambda:=P_\theta(u_0,u_1-\lambda).
		\]
		For each fixed \(\lambda\in[-C,C]\), we have
		\[
		\phi_j-2C\le \min\{u_{j,0},u_{j,1}-\lambda\}\le\phi_j,
		\qquad
		\phi-2C\le \min\{u_{0},u_{1}-\lambda\}\le\phi .
		\]
		Hence, by \Cref{thm:relative-bounded-obstacle-capacity-continuity}, $p_{j,\lambda}\to p_\lambda$ in $\operatorname{Cap}_\omega$
		for every fixed \(\lambda\in[-C,C]\). By \cref{thm:relative-weak-geodesic-segment-capacity-stability}, we can write
		\[
		u_{j,t}
		=
		\left(
		\sup_{\lambda\in[-C,C]}(p_{j,\lambda}+t\lambda)
		\right)^*,
		\qquad
		u_t
		=
		\left(
		\sup_{\lambda\in[-C,C]}(p_\lambda+t\lambda)
		\right)^* .
		\]
		Fix \(\varepsilon>0\). Choose a partition
		\[
		\Pi=\{-C=\lambda_0<\lambda_1<\cdots<\lambda_m=C\}
		\]
		with mesh \(|\Pi|<\varepsilon/3\). Define
		\[
		u_{j,t}^{\Pi}:=\max_{0\le \ell\le m}
		(p_{j,\lambda_\ell}+t\lambda_\ell),
		\qquad
		u_t^{\Pi}:=\max_{0\le \ell\le m}
		(p_{\lambda_\ell}+t\lambda_\ell).
		\]
		Since \(\lambda\mapsto p_{j,\lambda}\) and \(\lambda\mapsto p_\lambda\) are
		decreasing, for \(\lambda\in[\lambda_\ell,\lambda_{\ell+1}]\) and
		\(0\le t\le1\) we have
		\[
		p_{j,\lambda}+t\lambda
		\le p_{j,\lambda_\ell}+t\lambda_\ell+|\Pi|,
		\qquad
		p_\lambda+t\lambda
		\le p_{\lambda_\ell}+t\lambda_\ell+|\Pi|.
		\]
		Therefore, for all \(t\in[0,1]\),
		\[
		u_{j,t}^{\Pi}\le u_{j,t}\le u_{j,t}^{\Pi}+|\Pi|,
		\qquad
		u_t^{\Pi}\le u_t\le u_t^{\Pi}+|\Pi|.
		\]
		It follows that
		\[
		\{|u_{j,t}-u_t|>\varepsilon\}
		\subset
		\{|u_{j,t}^{\Pi}-u_t^{\Pi}|>\varepsilon-|\Pi|\}
		\subset
		\{|u_{j,t}^{\Pi}-u_t^{\Pi}|>2\varepsilon/3\}.
		\]
		Also,
		\[
		|u_{j,t}^{\Pi}-u_t^{\Pi}|
		\le
		\max_{0\le\ell\le m}|p_{j,\lambda_\ell}-p_{\lambda_\ell}|.
		\]
		Thus, uniformly in \(t\in[0,1]\),
		\[
		\begin{aligned}
			\operatorname{Cap}_\omega(\{|u_{j,t}-u_t|>\varepsilon\})
			&\le
			\sum_{\ell=0}^m
			\operatorname{Cap}_\omega
			\bigl(\{|p_{j,\lambda_\ell}-p_{\lambda_\ell}|>2\varepsilon/3\}\bigr).
		\end{aligned}
		\]
		The right hand side tends to zero because the partition is finite and
		\(p_{j,\lambda_\ell}\to p_{\lambda_\ell}\) in capacity for each \(\ell\). This
		proves the result.
	\end{proof}
	
	\begin{corollary}
		\label{cor:minimal-weak-geodesic-segment-capacity-stability}
		Let \(u_{j,0},u_{j,1},u_0,u_1\in\operatorname{PSH}(X,\theta)\) satisfy, for
		some \(C>0\),
		\[
		V_\theta-C\le u_{j,0},u_{j,1},u_0,u_1\le V_\theta .
		\]
		Assume that
		\[
		u_{j,0}\to u_0,
		\qquad
		u_{j,1}\to u_1
		\quad\text{in }\operatorname{Cap}_\omega .
		\]
		Let \(u_{j,t}\) and \(u_t\) be the corresponding weak geodesic segments. Then,
		for every \(\varepsilon>0\),
		\[
		\lim_{j\to\infty}
		\sup_{0\le t\le1}
		\operatorname{Cap}_\omega(\{|u_{j,t}-u_t|>\varepsilon\})=0 .
		\]
	\end{corollary}
	
	\begin{proof}
		Apply \Cref{thm:relative-weak-geodesic-segment-capacity-stability} with
		\(\phi_j=\phi=V_\theta\).
	\end{proof}
	
	\begin{remark}
		This result is a capacity-topology analogue of the \(d_1\)-stability of weak
		geodesic segments. The proof uses only the rooftop representation in
		\Cref{lem:relative-geodesic-segment-rooftop-legendre} and the capacity continuity of
		rooftop envelopes.
	\end{remark}
	
	\section{Stability of twisted K\"ahler-Einstein metrics in big classes}\label{sec:DZ-stability}
	
	In this subsection, we use our stability theorem to establish a stability result for twisted K\"ahler-Einstein equations in big classes whose existence was recently studied by Darvas-Zhang \cite{DZ24}. 
	
	We recall the following existence theorem of Darvas-Zhang in the form used below. \begin{theorem}[Darvas-Zhang existence theorem {\cite[Theorem~1.2]{DZ24}}] \label{thm:DZ-existence} Let \(X\) be compact K\"ahler, let \(\theta\) be a smooth closed real \((1,1)\)-form whose cohomology class is big, and put \(V:=V_\theta\), \(m:=\int_X\theta_V^n\). Let \(\eta\) be a smooth closed real \((1,1)\)-form such that \(c_1(X)=\{\theta\}+\{\eta\}\). Let \(\psi\) be quasi-psh with \(\eta+dd^c\psi\ge0\). Assume moreover that $e^{-\psi}$ is $L^1$ integrable and \(\delta_\psi(\{\theta\})>1\). Choose \(f\in C^\infty(X)\) such that \(\operatorname{Ric}(\omega)=\theta+\eta+dd^cf\), and set \(\mu:=e^{f-\psi}\omega^n\). Then there exists \(u\in\mathcal E(X,\theta,V)\), normalized by \(\sup_Xu=0\), such that \[ \theta_u^n = m\,\frac{e^{-u}\mu}{\int_X e^{-u}\,d\mu}. \] Equivalently, \(\theta_u=\theta+dd^cu\) has minimal singularities and solves \[ \operatorname{Ric}(\theta_u)=\theta_u+\eta+dd^c\psi . \] \end{theorem}
	
	\begin{lemma}
		\label{lem:DZ-weighted-TV-nonlinear-weight}
		Let $(X,\omega)$ be a compact K\"ahler manifold and let $\theta$ be a smooth
		closed real $(1,1)$-form whose cohomology class is big. Put $V:=V_\theta$.
		Let $\mu_j,\mu$ be tame measures such that $\|e^{-V}\mu_j-e^{-V}\mu\|_{\mathrm{TV}}\to0$.
		Let $u_j,u\in\operatorname{PSH}(X,\theta)$ be such that $u_j\to u$ in $L^1(X)$ and $V-C\le u_j,u\le V$ for some constant $C>0$. Then
		\[
		e^{-u_j}\mu_j\to e^{-u}\mu
		\quad\text{in total variation}.
		\]
	\end{lemma}
	
	\begin{proof}
		Set $\sigma_j:=e^{-V}\mu_j,\sigma:=e^{-V}\mu$. By assumption we have $ \|\sigma_j-\sigma\|_{\mathrm{TV}}\to0$. Write
		\[
		e^{-u_j}\mu_j
		=
		e^{V-u_j}\,e^{-V}\mu_j
		=
		g_j\sigma_j,
		\qquad
		g_j:=e^{V-u_j},
		\]
		and similarly
		\[
		e^{-u}\mu
		=
		e^{V-u}\,e^{-V}\mu
		=
		g\sigma,
		\qquad
		g:=e^{V-u}.
		\]
		Since $V-C\le u_j,u\le V$, we have $1\le g_j,g\le e^C$. Hence
		\[
		\begin{aligned}
			\|g_j\sigma_j-g\sigma\|_{\mathrm{TV}}
			&\le
			\|g_j\sigma_j-g_j\sigma\|_{\mathrm{TV}}
			+
			\|g_j\sigma-g\sigma\|_{\mathrm{TV}}       \\
			&\le
			e^C\|\sigma_j-\sigma\|_{\mathrm{TV}}
			+
			\int_X |g_j-g|\,d\sigma .
		\end{aligned}
		\]
		The first term tends to zero. It remains to prove that the second term tends
		to zero.
		
		Since $0\le V-u_j,V-u\le C$, the exponential function is Lipschitz on
		$[0,C]$. Thus there is a constant $A_C>0$ such that
		\[
		|g_j-g|
		=
		|e^{V-u_j}-e^{V-u}|
		\le
		A_C |u_j-u|.
		\]
		The measure $\sigma$ is finite and
		does not charge pluripolar sets. By \cref{lem:DoVu-nonpluripolar-measure-convergence}, the $L^1$ convergence $u_j\to u$ implies
		\[
		\int_X \min\{|u_j-u|,1\}\,d\sigma\to0 .
		\]
		Since $|u_j-u|\le C$, this gives
		\[
		\int_X |u_j-u|\,d\sigma\to0 ,
		\]
		hence $\int_X |g_j-g|\,d\sigma\to0$ and the proof is concluded.
	\end{proof}
	
	We first give a stability theorem under some conditional assumptions, these conditions will be justified later.
	\begin{theorem}
		\label{thm:DZ-conditional-capacity-stability}
		Let $(X,\omega)$, $\theta$, $V:=V_\theta$, and
		$m:=\int_X\theta_V^n>0$ be as above. Let $\mu_j,\mu$ be tame measures satisfying $\|e^{-V}\mu_j-e^{-V}\mu\|_{\mathrm{TV}}\to0$.
		Let $u_j\in\mathcal E(X,\theta,V)$ be normalized by $\sup_Xu_j=0$ and solve (assume its existence)
		\[
		\theta_{u_j}^n
		=
		\nu_j
		:=
		m\frac{e^{-u_j}\mu_j}{\int_X e^{-u_j}\,d\mu_j}.
		\]
		Assume that we have the uniform estimates $V-C\le u_j\le V$ for all $j$, and that
		$u_j\to u$ in $L^1(X)$ for some $u\in PSH(X,\theta)$. Then, $u\in\mathcal E(X,\theta,V)$, $\sup_Xu=0$, and
		\[
		\theta_u^n
		=
		\nu_u
		:=
		m\frac{e^{-u}\mu}{\int_X e^{-u}\,d\mu}.
		\]
		Moreover,
		$$
		\nu_j\to\nu_u \ \text{in total variation,\quad and}\
		u_j\to u \ \text{in } \operatorname{Cap}_\omega .
		$$
	\end{theorem}
	
	\begin{proof}
		Since $V-C\le u_j\le V$ and $u_j\to u$ in $L^1(X)$, we get $V-C\le u\le V$.
		Thus $u$ has the same singularity type as $V$.  Also $\sup_Xu=0$ by the Hartogs lemma and the normalization
		$\sup_Xu_j=0$. By \Cref{lem:DZ-weighted-TV-nonlinear-weight},
		\[
		e^{-u_j}\mu_j
		\longrightarrow
		e^{-u}\mu
		\quad\text{in total variation}.
		\]
		In particular
		\[
		Z_j:=\int_X e^{-u_j}\,d\mu_j\to
		Z_u:=\int_X e^{-u}\,d\mu.
		\]
		Since $Z_u>0$, we obtain
		\[
		\nu_j
		=
		m\frac{e^{-u_j}\mu_j}{Z_j}
		\longrightarrow
		m\frac{e^{-u}\mu}{Z_u}
		=
		\nu_u
		\quad\text{in total variation}.
		\]
		It remains to check that $u$ really solves the limiting equation.
		Define a finite positive non-pluripolar measure
		\[
		\rho:=\nu_u+\sum_{j=1}^{\infty}2^{-j}\nu_j .
		\]
		Then there are bounded non-negative functions $F_j,F$ with
		\[
		\nu_j=F_j\rho,
		\qquad
		\nu_u=F\rho .
		\]
		The total variation convergence $\nu_j\to\nu_u$ gives $F_j\to F$ in $L^1(X,\rho)$.
		Since $\theta_{u_j}^n=\nu_j=F_j\rho$
		and $u_j\to u$ in $L^1(X)$, \cite[Lemma 5.10]{DDL25} gives $ \theta_u^n\ge F\rho=\nu_u$.
		Since both sides have total mass $m$, we get 
		\[
		\theta_u^n=\nu_u
		=
		m\frac{e^{-u}\mu}{\int_Xe^{-u}\,d\mu}.
		\]
		Now apply \cref{thm:moving-background-capacity-stability} with
		\[
		\theta_j=\theta,\qquad
		\phi_j=\phi=V_\theta,
		\qquad
		\nu_j\to\nu_u
		\quad\text{in total variation}.
		\]
		It gives $u_j\to u$ in $\operatorname{Cap}_\omega$.
	\end{proof}
	
	We first record a simple consequence of capacity convergence and the integrability estimate, which shows that under some clean conditions of $f_j$ and $\psi_j$, the strong measure convergence assumptions $e^{-V}\mu_j\to e^{-V}\mu$ in \cref{thm:DZ-conditional-capacity-stability} could be fulfilled.
	\begin{lemma}
		\label{lem:DZ-weighted-TV-from-twist-convergence}
		Let $   \mu_j=e^{f_j-\psi_j}\omega^n,
		\mu=e^{f-\psi}\omega^n$ be tame measures
		and put $V:=V_\theta$. Assume that $f_j\to f$ uniformly, that
		$\psi_j\to\psi$ in $\operatorname{Cap}_\omega$, and that there is a uniform constant
		$A>0$ such that $\psi_j\ge\psi-A$.
		Assume moreover that $ c_\mu[V]>1$. Then,
		\[
		e^{-V+f_j-\psi_j}
		\to
		e^{-V+f-\psi}
		\quad\text{in }L^1(X,\omega^n).
		\]
	\end{lemma}
	
	\begin{proof}
		Since \(\psi_j\to\psi\) in \(\operatorname{Cap}_\omega\), we have
		\(\psi_j\to\psi\) in \(\omega^n\)-measure. Together with \(f_j\to f\)
		uniformly, this gives
		\[
		-V+f_j-\psi_j
		\to
		-V+f-\psi
		\]
		in \(\omega^n\)-measure, since \(V,\psi,\psi_j\) are finite outside
		pluripolar sets. Hence, by continuity of the exponential function,
		\[
		e^{-V+f_j-\psi_j}
		\to
		e^{-V+f-\psi}
		\]
		in \(\omega^n\)-measure.
		
		It remains to prove uniform integrability. Since \(f_j\to f\) uniformly,
		there exists \(B>0\) such that \(f_j\le f+B\) for all \(j\). Using
		\(\psi_j\ge\psi-A\), we obtain
		\[
		e^{-V+f_j-\psi_j}
		\le
		e^{A+B}e^{-V+f-\psi}
		=
		e^{A+B}e^{-V}\frac{d\mu}{d\omega^n}.
		\]
		The right-hand side belongs to \(L^1(X,\omega^n)\), because
		\(c_\mu[V]>1\) implies $\int_X e^{-V}\,d\mu<+\infty$.
		Therefore the sequence \(e^{-V+f_j-\psi_j}\) is uniformly dominated by an
		\(L^1(\omega^n)\)-function.
		
		Convergence in measure together with this \(L^1\)-domination implies
		convergence in \(L^1\). Indeed, every subsequence admits an a.e. convergent
		subsequence, and the dominated convergence theorem applies to that
		subsequence. Thus
		\[
		e^{-V+f_j-\psi_j}
		\to
		e^{-V+f-\psi}
		\quad\text{in }L^1(X,\omega^n).
		\]
	\end{proof}
	
	Thanks to Guan-Zhou's strong openness theorem \cite{GZ15-1} and Guan-Li-Zhou's stability of multiplier ideal sheaves \cite{GLZ22}, we can derive a uniform Skoda estimate on a energy-bounded set: 
	\begin{lemma}
		\label{lem:DZ-uniform-Skoda-on-energy-sublevel}
		Let $\mu=e^{f-\psi}\omega^n$ be a tame measure and assume that $ c_\mu[V_\theta]>1$.
		For $M>0$, set
		\[
		\mathcal K_M
		:=
		\left\{
		w\in\mathcal E^1(X,\theta):
		\sup_Xw=0,\ 
		-M\le I_\theta(w)\le 0
		\right\}.
		\]
		Then there exist $q>1$ and $C_M>0$ such that
		\[
		\sup_{w\in\mathcal K_M}
		\int_X e^{-q w}\,d\mu
		\le C_M .
		\]
	\end{lemma}
	
	\begin{proof}
		Choose $q$ such that $ 1<q<c_\mu[V_\theta]$. We argue by contradiction. Suppose that there are
		$w_k\in\mathcal K_M$ such that
		\[
		\int_X e^{-q w_k}\,d\mu\to+\infty .
		\]
		The set $\mathcal K_M$ is compact in the $L^1$ topology. Hence, after passing
		to a subsequence, we may assume that
		\[
		w_k\to w
		\quad\text{in }L^1(X)
		\]
		for some $w\in\mathcal K_M$. Indeed, by the proof of \cite[Proposition 5.1]{DDL25} we get $\varphi_j=(\sup_{k\ge j}w_k)^*$ decreases to $w$ and $\varphi_j\in\mathcal K_M$. Then using \cite[Lemma 5.7]{DDL25} and the fact that $-M\le I_\theta(\varphi_k)\le0$, we conclude that $w\in\mathcal{E}^1(X,\theta)$. Since $w\in\mathcal E^1(X,\theta)\subset\mathcal E(X,\theta)$, \cite[Proposition 2.5]{DZ24}, which essentially comes from Guan-Zhou's strong openness theorem \cite{GZ15-1,GZ15-2}, yields that $c_\mu[w]=c_\mu[V_\theta]$, and thus $ q<c_\mu[w]$. Choose $q'$ such that $ q<q_1<c_\mu[w]$, then
		\begin{equation}\label{eq:q_1w}
			\int_X e^{-q_1w}\,d\mu<+\infty
		\end{equation}
		Then by Guan-Li-Zhou's stability of multiplier ideal sheaves \cite[Theorem 1.1]{GLZ22} (see also Xiao-Zhou \cite{XZ26}), we have 
		\[
		\sup_{k\gg1}\int_X e^{-q w_k}\,d\mu<+\infty .
		\]
		This contradicts the choice of \(w_k\).
	\end{proof}
	
	Under similar conditions, we prove the uniform estimates of solutions $u_j$:
	\begin{proposition}
		\label{prop:DZ-uniform-relative-bound-from-nonworsening-twists}
		Let $(X,\omega)$ be compact K\"ahler, and let $\theta$ be a smooth closed
		real $(1,1)$-form whose cohomology class is big. Put $  V:=V_\theta,
		m:=\int_X\theta_V^n>0$ and let $\mu_j=e^{f_j-\psi_j}\omega^n,
		\mu=e^{f-\psi}\omega^n$ be tames measures with $\delta_\psi(\{\theta\})>1$. Let $\eta,\eta_j$ be smooth closed
		real $(1,1)$-forms such that $c_1(X)=\{\theta\}+\{\eta_j\}$, $c_1(X)=\{\theta\}+\{\eta\}$ and
		\[\operatorname{Ric}(\omega)=\theta+\eta+dd^cf,\quad\operatorname{Ric}(\omega)=\theta+\eta_j+dd^cf_j,\quad
		\eta+dd^c\psi\ge0,
		\quad \eta_j+dd^c\psi_j\ge0,\quad\forall j\ge0.
		\]
		Assume that $f_j\to f$ uniformly, that
		$\psi_j\to\psi$ in $\operatorname{Cap}_\omega$, and that there is a uniform constant
		$A>0$ such that $\psi_j\ge\psi-A$. Let $u_j\in\mathcal E(X,\theta,V)$ be normalized solutions of
		\[
		\sup_Xu_j=0,
		\qquad
		\theta_{u_j}^n
		=
		m\frac{e^{-u_j}\mu_j}{\int_Xe^{-u_j}\,d\mu_j}.
		\]
		Then there exists a constant $C>0$, independent of $j$, such that
		\[
		V-C\le u_j\le V .
		\]
	\end{proposition}
	
	\begin{proof}
		First, since $\psi_j\ge\psi-A$, it is clear from definitions that $\delta_{\psi_j}(\{\theta\})
		\ge
		\delta_\psi(\{\theta\})
		>1 $.
		\cref{thm:DZ-existence} then gives the existence of $u_j\in\mathcal E(X,\theta,V)$.
		
		We next prove a uniform Ding energy bound. Let
		\[
		\mathcal D_j(w)
		:=
		-\log\int_X e^{-w}\,d\mu_j
		-
		I_\theta(w),
		\qquad
		\sup_Xw=0,
		\]
		and let
		\[
		\mathcal D(w)
		:=
		-\log\int_X e^{-w}\,d\mu
		-
		I_\theta(w).
		\]
		Since $\delta_\psi(\{\theta\})>1$ and $\eta+dd^c\psi\ge0$, \cite[Corollary 5.6]{DZ24} gives constants $a,b>0$ such that
		\[
		\mathcal D(w)\ge a(\sup_Xw-I_\theta(w))-b
		\quad
		\text{for all }w\in\mathcal E^1(X,\theta).
		\]
		The assumptions $f_j\to f$ uniformly and $\psi_j\ge\psi-A$ imply that there
		is a uniform constant $B>0$ such that $\mu_j\le B\mu$
		for all $j$. Hence, for every $w$,
		\[
		\int_X e^{-w}\,d\mu_j
		\le
		B\int_Xe^{-w}\,d\mu.
		\]
		Therefore
		\[
		\begin{aligned}
			\mathcal D_j(w)
			&=
			-\log\int_Xe^{-w}\,d\mu_j
			-
			I_\theta(w)                                      \\
			&\ge
			-\log\int_Xe^{-w}\,d\mu
			-
			I_\theta(w)
			-
			\log B                                           \\
			&=
			\mathcal D(w)-\log B                             \\
			&\ge
			a(\sup_Xw-I_\theta(w))-b-\log B .
		\end{aligned}
		\]
		Thus the functionals $\mathcal D_j$ are uniformly proper.
		
		Since $u_j$ minimizes $\mathcal D_j$, we have $\mathcal D_j(u_j)\le\mathcal D_j(V)$.
		By \Cref{lem:DZ-weighted-TV-from-twist-convergence},
		\[
		e^{-V}\mu_j\to e^{-V}\mu
		\quad\text{in total variation}.
		\]
		In particular, $\int_Xe^{-V}\,d\mu_j
		\to
		\int_Xe^{-V}\,d\mu>0$.
		Thus we have $ \sup_j\mathcal D_j(V)<+\infty$.
		The uniform properness therefore gives a constant $M>0$ such that
		\[
		I_\theta(u_j)\ge -M
		\quad\text{for all }j.
		\]
		Hence
		\[
		u_j\in\mathcal K_M
		:=
		\left\{
		w\in\mathcal E^1(X,\theta):
		\sup_Xw=0,\ 
		I_\theta(w)\ge -M
		\right\}.
		\]
		We are now going to derive a uniform $L^p$ bound for the normalized right hand side measures, in order to use the relative $L^\infty$-estimate.
		
		By \cite[Lemma 4.3]{DZ24}, $c_\mu[V]\ge\delta_\psi(\{\theta\})>1$. This combined with $\mu_j\le B\mu$ and \Cref{lem:DZ-uniform-Skoda-on-energy-sublevel} gives that
		\[
		\sup_j\int_Xe^{-q u_j}\,d\mu_j<+\infty .
		\]
		
		Write $h_j:=\frac{d\mu_j}{d\omega^n}=e^{f_j-\psi_j}$ and $h:=\frac{d\mu}{d\omega^n}=e^{f-\psi}$.
		Because $\mu$ is tame, Guan-Zhou strong openness \cite{GZ15-1} gives $r>1$ such that $h\in L^r(X,\omega^n)$.
		Since $\mu_j\le B\mu$, we have $h_j\le Bh$, and consequently
		\[
		\sup_j\|h_j\|_{L^r(X,\omega^n)}<+\infty .
		\]
		Choose $p>1$ sufficiently close to $1$ and $a>1$ such that, with
		$b:=a/(a-1)$,
		\[
		pa<q,
		\qquad
		1+(p-1)b<r.
		\]
		Then H\"older's inequality gives
		\[
		\begin{aligned}
			\int_X(e^{-u_j}h_j)^p\omega^n
			&=
			\int_Xe^{-pu_j}h_j^{p-1}\,d\mu_j                                      \\
			&\le
			\left(\int_Xe^{-pa u_j}\,d\mu_j\right)^{1/a}
			\left(\int_Xh_j^{(p-1)b}\,d\mu_j\right)^{1/b}                           \\
			&=
			\left(\int_Xe^{-pa u_j}\,d\mu_j\right)^{1/a}
			\left(\int_Xh_j^{1+(p-1)b}\omega^n\right)^{1/b}.
		\end{aligned}
		\]
		Since $pa<q$ and $u_j\le0$, the first factor is uniformly bounded by $\sup_j\int_Xe^{-q u_j}\,d\mu_j<+\infty$.
		The second factor is uniformly bounded by the uniform $L^r$ bound for $h_j$.
		Hence
		\[
		\sup_j\|e^{-u_j}h_j\|_{L^p(X,\omega^n)}<+\infty .
		\]
		Let \(Z_j:=\int_Xe^{-u_j}\,d\mu_j\). Since $e^{-V}\mu_j\to e^{-V}\mu$ in total variation and \(0\le e^V\le1\), we have
		\[
		\mu_j(X)
		=
		\int_X e^V\,d(e^{-V}\mu_j)
		\longrightarrow
		\int_X e^V\,d(e^{-V}\mu)
		=
		\mu(X)>0.
		\]
		As \(u_j\le0\), it follows that there is a uniform constant $c_0$ such that $Z_j\ge \mu_j(X)\ge c_0>0$
		for all large \(j\). Therefore the normalized
		right hand side density $F_j:=
		m\frac{e^{-u_j}h_j}{Z_j}$ satisfies
		\[
		\sup_j\|F_j\|_{L^p(X,\omega^n)}<+\infty .
		\]
		The equation is
		\[
		\theta_{u_j}^n=F_j\omega^n,
		\qquad
		u_j\in\mathcal E(X,\theta,V),
		\qquad
		\sup_Xu_j=0.
		\]
		By the relative $L^\infty$ estimate \cref{thm:relative-Linfty-estimate}, there is a constant $C>0$, independent of $j$,
		such that
		\[
		V-C\le u_j\le V .
		\]
		The upper bound follows also from $\sup_Xu_j=0$ and the definition of
		$V_\theta$.
	\end{proof}
	As an immediate corollary of \cref{thm:DZ-conditional-capacity-stability}, \cref{lem:DZ-weighted-TV-from-twist-convergence} and \cref{prop:DZ-uniform-relative-bound-from-nonworsening-twists}, we obtain:
	
	\begin{corollary}[=\cref{intro:stability-of-Darvas-Zhang}]
		Let $(X,\omega)$ be compact K\"ahler, and let $\theta$ be a smooth closed
		real $(1,1)$-form whose cohomology class is big. Put $  V:=V_\theta,
		m:=\int_X\theta_V^n>0$ and let $\mu_j=e^{f_j-\psi_j}\omega^n,
		\mu=e^{f-\psi}\omega^n$ be tame measures with $\delta_\psi(\{\theta\})>1$. Let $\eta,\eta_j$ be smooth closed
		real $(1,1)$-forms such that $c_1(X)=\{\theta\}+\{\eta_j\}$, $c_1(X)=\{\theta\}+\{\eta\}$ and
		\[\operatorname{Ric}(\omega)=\theta+\eta+dd^cf,\quad\operatorname{Ric}(\omega)=\theta+\eta_j+dd^cf_j,\quad
		\eta+dd^c\psi\ge0,
		\quad \eta_j+dd^c\psi_j\ge0,\quad\forall j\ge0.
		\]
		Assume that $f_j\to f$ uniformly, that
		$\psi_j\to\psi$ in $\operatorname{Cap}_\omega$, and that there is a uniform constant
		$A>0$ such that $\psi_j\ge\psi-A$. Let $u_j\in\mathcal E(X,\theta,V)$ be normalized solutions of
		\[
		\sup_Xu_j=0,
		\qquad
		\theta_{u_j}^n
		=
		m\frac{e^{-u_j}\mu_j}{\int_Xe^{-u_j}\,d\mu_j}.
		\]
		Then every subsequence of $u_j$ has a further subsequence converging in capacity to some $u\in\mathcal{E}(X,\theta)$ such that
		\[
		\theta_u^n
		=
		m\frac{e^{-u}\mu}{\int_X e^{-u}\,d\mu}.
		\]
	\end{corollary}

	\section{Quantization of partial equilibrium measures}
	
	In this final section we show that the capacity topology also governs envelopes
	defined by multiplier ideal constraints. We apply the stability method developed
	above to partial \(I\)-equilibrium envelopes in the sense of Darvas-Xia \cite{DX22,DX24}. Although
	these envelopes are defined through multiplier ideal conditions rather than
	ordinary singularity type inequalities, the same capacity mechanism applies
	after passing to the associated \(I\)-model projection. Thus moving multiplier
	ideals, and in particular moving algebraic vanishing conditions, give rise to
	partial equilibrium envelopes which vary continuously in capacity. As a result, we can give a moving quantization of partial equilibrium measures.
	
	We will use the following   uniform relative boundedness property for partial \(I\)-equilibrium envelopes. The proof is essentially contained in \cite[Lemma 5.2]{DX24}.
	\begin{lemma}
		\label{lem:uniform-bounds-partial-I-envelopes}
		Let $K\subset X$ be compact and non-pluripolar, and let
		$v_j,v\in C^0(K)$ satisfy $v_j\to v$ uniformly. Let $u_j,u\in\operatorname{PSH}(X,\theta)$ and put
		\[
		\Phi_j:=P[u_j]_I,\qquad \Phi:=P[u]_I ,\qquad W_j:=P_K[u_j]_I(v_j).
		\qquad
		W:=P_K[u]_I(v).
		\]
		Assume that $\sup_X\Phi_j=\sup_X\Phi=0$ and $\Phi_j\to\Phi$ in capacity. Then there exists a constant $A>0$, independent of $j$, such that
		\[
		\Phi_j-A\le W_j\le \Phi_j+A,
		\qquad
		\Phi-A\le W\le \Phi+A .
		\]
	\end{lemma}
	
	\begin{theorem}[=\cref{intro: stability of Darvas-Xia}]
		\label{thm:partial-I-envelope-capacity-stability}
		Let $(X,\omega)$ be a compact K\"ahler manifold, and let $\theta$ be a smooth
		closed real $(1,1)$-form whose cohomology class is big. Let
		$K\subset X$ be compact and non-pluripolar. Let $u_j,u\in\operatorname{PSH}(X,\theta)$ with $\sup_Xu_j=\sup_Xu=0$ and put
		\[
		\Phi_j:=P[u_j]_I,\qquad \Phi:=P[u]_I ,\qquad W_j:=P_K[u_j]_I(v_j),
		\qquad
		W:=P_K[u]_I(v).
		\]
		Assume that $\Phi_j\to\Phi$ in capacity. Let $v_j,v\in C^0(K)$ satisfy $v_j\to v$ uniformly on $K$. Then $W_j\to W$ in capacity.
		
		Consequently, if we moreover have $ \int_X\theta_{\Phi_j}^n\to \int_X\theta_\Phi^n>0$, then we have the weak convergence $\theta_{W_j}^n\rightharpoonup \theta_W^n$.
	\end{theorem}
	\begin{proof}
		By \Cref{lem:uniform-bounds-partial-I-envelopes}, after replacing
		$W_j,W$ by $W_j-A,W-A$, we may assume that there is $C_0>0$ such that
		\[
		\Phi_j-C_0\le W_j\le \Phi_j,
		\qquad
		\Phi-C_0\le W\le \Phi .
		\]
		The shift does not affect capacity convergence.
		
		We first sketch the lower estimate. Fix $\delta>0$ and set
		$A_j:=\{W_j<W-\delta\}$. Let $V:=V_\theta$. By
		\Cref{lem:strict-subbarrier-slope-gap}, applied to $W$ and $\Phi$, there are
		$\chi\in\operatorname{PSH}(X,\theta)$ and $a,\sigma>0$ such that
		$\theta_\chi\ge a\omega$ and
		\[
		\chi\le W-\sigma(V-W).
		\]
		For $R>0$ put $G_R:=\{\chi>W-R\}$; then
		$\text{Cap}_\omega(X\setminus G_R)\to0$ as $R\to+\infty$. Fix $R$, choose
		$0<b<a$, and put
		\[
		\tau:=\min\left\{\frac12,\frac{\delta}{8(R+b)}\right\}.
		\]
		For $T>1$, let $U_j^T:=\max(W_j,V-T)$. Let
		$K_{j,R,T}:= A_j\cap G_R\cap\{W_j>V-T\}$ and
		$\rho\in\operatorname{PSH}(X,\omega)$, $0\le\rho\le1$. Set
		\[
		\psi_\rho:=(1-\tau)W+\tau\chi+\tau b\rho-\frac{\delta}{2}-\tau b .
		\]
		Then $\psi_\rho\le W-\delta/2$, $\theta_{\psi_\rho}\ge\tau b\,\omega_\rho$,
		and on $G_R$ we have $\psi_\rho\ge W-5\delta/8$. Hence
		$K_{j,R,T}\subset\{U_j^T<\psi_\rho\}$. With
		$H_{j,\rho}^T:=\max(U_j^T,\psi_\rho)$, the comparison principle \cref{thm:DDL-relative-comparison} in the
		minimal full mass class gives
		\[
		(\tau b)^n\int_{K_{j,R,T}}\omega_\rho^n
		\le
		\int_{\{U_j^T<\psi_\rho\}}\theta_{U_j^T}^n .
		\]
		We split the last integral according to $\{W_j>V-T\}$ and $\{W_j\le V-T\}$.
		On $\{W_j>V-T\}$, plurifine locality gives
		$\theta_{U_j^T}^n=\theta_{W_j}^n$, and the contribution is bounded by
		\[
		\int_{\{W_j<W-\delta/2\}}\theta_{W_j}^n .
		\]
		This is the only new point compared with the ordinary envelope case. By \cite[Corollary 5.7]{DX24}, $\theta_{W_j}^n$ puts no mass on
		$(X\setminus K)\cup\{W_j<v_j\}$. Since $W\le v$ on $K$, this contribution is
		bounded by
		\[
		\int_{K\cap\{v_j<v-\delta/2\}}\theta_{W_j}^n,
		\]
		which is zero for all large $j$, as $v_j\to v$ uniformly on $K$.
		
		On $\{W_j\le V-T\}$, one can show that
		$\{U_j^T<\psi_\rho\}\cap\{W_j\le V-T\}$ is contained in
		\[
		\{\Phi_j<\Phi-L_T\},
		\qquad
		L_T:=T-C_0-\frac{T-\delta/2}{1+\tau\sigma}.
		\]
		For fixed $R,\delta$, we have $L_T\to+\infty$ as $T\to+\infty$. Hence, we can argue as in \cref{thm:relative-bounded-obstacle-capacity-continuity} to obtain
		\begin{equation}\label{eq:lower tail for I envelope}
			\text{Cap}_\omega(\{W_j<W-\delta\})\to0 .
		\end{equation}
		Now, the same arguments as in \cref{thm:relative-bounded-obstacle-capacity-continuity} yields
		\[
		W_j\to W\quad\text{in }\text{Cap}_\omega .
		\]
		
		Finally, by \cite[Lemma 5.2]{DX24},
		$[W_j]=[\Phi_j]$ and $[W]=[\Phi]$. Hence
		\[
		\int_X\theta_{W_j}^n=\int_X\theta_{\Phi_j}^n,
		\qquad
		\int_X\theta_W^n=\int_X\theta_\Phi^n .
		\]
		If $\int_X\theta_{\Phi_j}^n\to\int_X\theta_\Phi^n$, then
		$\int_X\theta_{W_j}^n\to\int_X\theta_W^n$. Since $W_j\to W$ in capacity, \cite[Theorem 2.1]{DDL25} gives the weak convergence $\theta_{W_j}^n\rightharpoonup\theta_W^n$.
	\end{proof}
	
	\begin{corollary}[Moving quantization of partial equilibrium measures]
		\label{cor:moving-quantization-partial-equilibrium}
		Assume the setting of \Cref{thm:partial-I-envelope-capacity-stability}.
		Let $\nu_j$ be Bernstein-Markov measures for $(K,v_j)$, and let $\beta^k_{v_j,u_j,\nu_j}$ be the partial Bergman measures associated to $H^0(X,L^k\otimes T\otimes \mathcal I(ku_j))$. Then
		\[
		\lim_{j\to\infty}\lim_{k\to\infty}
		\beta^k_{v_j,u_j,\nu_j}
		=
		\theta^n_{P_K[u]_I(v)}
		\]
		weakly. In particular, there exists a sequence $k_j\to\infty$ such that
		\[
		\beta^{k_j}_{v_j,u_j,\nu_j}
		\rightharpoonup
		\theta^n_{P_K[u]_I(v)} .
		\]
	\end{corollary}
	
	We give an example of quantization by potentials with moving singularities:
	\begin{example}\label{ex:P1-moving-pole-quantization}
		We continue the setup of \Cref{ex:moving-poles-P1-cap-not-dS}, where
		\(X=\mathbb P^1\), \(L=\mathcal O(1)\), \(D_1=\{Z_0=0\}\),
		\(D_j=\{Z_0-\varepsilon_jZ_1=0\}\), and
		\[
		u_j=c\log|s_j|_{h_{\rm FS}}^2,\qquad \phi_j=P_\omega[u_j],
		\qquad
		\Phi_j:=P_\omega[u_j]_I.
		\]
		Since \(u_j\) has analytic singularity type (in fact algebraic singularity type), \cite[Proposition 2.20]{DX22} yields that
		\[
		P_\omega[u_j]_I=P_\omega[u_j].
		\]
		In particular \(\Phi_j=\phi_j\). Hence, by
		\Cref{ex:moving-poles-P1-cap-not-dS},
		\[
		\Phi_j\to\Phi_1
		\quad\text{in }\operatorname{Cap}_\omega,
		\qquad
		\int_X\omega_{\Phi_j}
		=
		\int_X\omega_{\Phi_1}
		=
		1-c.
		\]
		Moreover, the convergence is still not \(d_{\mathcal S}\)-convergence:
		\[
		d_{\mathcal S}([\Phi_j],[\Phi_1])\ge \frac{2c}{C_S}>0.
		\]
		
		\smallskip
		\noindent\textbf{Step 1: the multiplier ideals are moving vanishing conditions.}
		Near the point \(D_j\), use the affine coordinate
		\[
		w_j:=\frac{Z_0}{Z_1}-\varepsilon_j.
		\]
		In this coordinate \(D_j=\{w_j=0\}\), and
		\[
		u_j=c\log|w_j|^2+O(1).
		\]
		Let \(f=w_j^m g\) be a holomorphic germ with \(g(0)\neq0\). Then
		\(f\in\mathcal I(ku_j)\) if and only if
		\[
		|f|^2e^{-ku_j}
		\sim
		|w_j|^{2m-2kc}
		\]
		is locally integrable. Since we are in one complex dimension, this is
		equivalent to
		\[
		\int_0^\epsilon r^{2m-2kc}r\,dr<+\infty,
		\]
		which is equivalent to \(m>kc-1\), or equivalently
		\(m\ge\lfloor kc\rfloor\). Therefore
		\begin{equation}
			\label{eq:P1-moving-pole-multiplier-ideal}
			\mathcal I(ku_j)
			=
			\mathcal O_X(-\lfloor kc\rfloor D_j).
		\end{equation}
		Thus
		\[
		H^0\bigl(X,L^k\otimes\mathcal I(ku_j)\bigr)
		=
		H^0\bigl(\mathbb P^1,\mathcal O(k)\otimes\mathcal{O}(-\lfloor kc\rfloor D_j)\bigr).
		\]
		Equivalently, this is the space of degree \(k\) homogeneous polynomials
		divisible by \((Z_0-\varepsilon_jZ_1)^{\lfloor kc\rfloor}\).
		
		\smallskip
		\noindent\textbf{Step 2: equivariance of the \(I\)-model projection.}
		Let \(G_j\in{\rm PU}(2)\) be the automorphism used above. We have
		\(G_j^*\omega=\omega\) and \(u_j=G_j^*u_1\). We claim that
		\[
		P_\omega[G_j^*u_1]_I=G_j^*P_\omega[u_1]_I.
		\]
		Indeed, if \(\psi\in\operatorname{PSH}(X,\omega)\), then
		\(G_j^*\psi\in\operatorname{PSH}(X,\omega)\). Moreover multiplier ideals are
		functorial under biholomorphisms:
		\[
		\mathcal I(tG_j^*\psi)=G_j^{-1}\mathcal I(t\psi).
		\]
		Hence
		\[
		\psi\preceq_I u_1
		\quad\Longleftrightarrow\quad
		G_j^*\psi\preceq_I G_j^*u_1.
		\]
		Pullback by \(G_j\) therefore identifies the candidate set defining
		\(P_\omega[u_1]_I\) with the candidate set defining \(P_\omega[G_j^*u_1]_I\).
		Taking upper envelopes gives
		\[
		\Phi_j=P_\omega[u_j]_I=G_j^*P_\omega[u_1]_I=G_j^*\Phi_1.
		\]
		
		\smallskip
		\noindent\textbf{Step 3: capacity stability of the \(I\)-partial envelopes.}
		Let \(K\subset X\) be compact nonpluripolar, let \(v_j,v\in C^0(K)\) be such that $v_j\to v$ uniformly on $K$, and let
		\(\nu_j,\nu\) be Bernstein-Markov probability measures for $(K,v_j)$ and \((K,v)\). By \cref{thm:partial-I-envelope-capacity-stability} we have
		\[
		P_K[u_j]_I(v_j)
		\longrightarrow
		P_K[u_1]_I(v)
		\quad\text{in }\operatorname{Cap}_\omega.
		\]
		Thanks to \cite[Lemma 5.2]{DX24} the total masses are constant:
		\[
		\int_X\omega_{P_K[u_j]_I(v)}
		=
		\int_X\omega_{\Phi_j}
		=
		1-c
		=
		\int_X\omega_{\Phi_1}
		=
		\int_X\omega_{P_K[u_1]_I(v)}.
		\]
		Therefore, we finally obtain the weak convergence
		\begin{equation}
			\label{eq:P1-moving-pole-partial-equilibrium-measure-conv}
			\omega_{P_K[u_j]_I(v)}
			\rightharpoonup
			\omega_{P_K[u_1]_I(v)} .
		\end{equation}
		
		\smallskip
		\noindent\textbf{Step 4: a new quantization.}
		Let \(\beta^k_{v_j,u_j,\nu}\) be the partial Bergman measure associated to
		\[
		H^0\bigl(X,L^k\otimes\mathcal I(ku_j)\bigr)
		=
		H^0\bigl(\mathbb P^1,\mathcal O(k)\otimes\mathcal O(-\lfloor kc\rfloor D_j)\bigr),
		\]
		with the \(L^2\)-norm induced by \((K,v_j,\nu_j)\). For each fixed \(j\), \cite[Theorem 1.2]{DX24} gives
		\[
		\beta^k_{v_j,u_j,\nu}
		\rightharpoonup
		\omega_{P_K[u_j]_I(v_j)}
		\qquad (k\to+\infty).
		\]
		Combining this fixed-\(j\) quantization with
		\eqref{eq:P1-moving-pole-partial-equilibrium-measure-conv}, we obtain, for
		every \(f\in C^0(X)\),
		\begin{equation}
			\label{eq:P1-moving-pole-iterated-quantization}
			\lim_{j\to\infty}\lim_{k\to\infty}
			\int_X f\,d\beta^k_{v_j,u_j,\nu}
			=
			\int_X f\,\omega_{P_K[u_1]_I(v)} .
		\end{equation}
		In words, the partial Bergman measures associated to degree \(k\) sections
		vanishing to order at least \(\lfloor kc\rfloor\) at the moving divisor \(D_j\)
		quantize the partial equilibrium measure, and after \(j\to\infty\) they converge
		to the partial equilibrium associated to the limiting point \(D_1=\{Z_0=0\}\).
	\end{example}
	
	\begin{remark}
		We believe that under some more conditions, such as the measures $\nu_j$ are uniformly Berstein-Markov, the quantization \eqref{eq:P1-moving-pole-iterated-quantization} in \cref{ex:P1-moving-pole-quantization} could be uniformly diagonal, i.e., we can exchange $k,j$ in the double limit. Thus, we can provide more flexible quantizations when the singularities are moving. These result could have applications in the equidistribution of Fekete points (cf. \cite{BB10,BBWN11}).
	\end{remark}

\end{document}